\documentclass[11pt, reqno]{amsart}
\usepackage[T1]{fontenc}

\usepackage{wrapfig}

\usepackage{stmaryrd}

\usepackage{indentfirst}
\usepackage{amstext}
\usepackage{amsopn}
\usepackage{amsfonts}
\usepackage{amsmath}
\usepackage{latexsym}
\usepackage{amscd}
\usepackage{amssymb}
\usepackage{amsmath}
\usepackage{leftidx}

\usepackage{graphicx}

\usepackage{colortbl}

\usepackage[figuresleft]{rotating}
=\oddsidemargin \font\teneufm=eufm10 \font\seveneufm=eufm7
\font\fiveeufm=eufm5
\newfam\eufmfam
\textfont\eufmfam=\teneufm \scriptfont\eufmfam=\seveneufm
\scriptscriptfont\eufmfam=\fiveeufm

\def\1{\mbox{\bf 1}}

\newtheorem{stheorem}{Theorem}[section]

\newtheorem{déf}[stheorem]{Définition}

\newtheorem{prop}[stheorem]{Proposition}
\newtheorem{lem}[stheorem]{Lemme}
\newtheorem{coro}[stheorem]{Corollaire}
\newtheorem{thm}[stheorem]{Théorème}
\newtheorem*{thm*}{Théorème}

\usepackage[numbers]{natbib}
\usepackage[colorlinks=true,pagebackref=false,linktocpage=true]{hyperref}

\usepackage{tikz}
\usepackage{tikz-cd}
\usepackage[french]{babel}
\usetikzlibrary{babel}

\usepackage{adjustbox}
\usepackage{enumitem}
\usepackage[scr=rsfs]{mathalpha}
\usepackage{mathtools}

\usepackage{dynkin-diagrams}
\usepackage{musicography}

\theoremstyle{remark}

\newtheorem{rmq}[stheorem]{Remarque}

\newtheorem{rmqs}[stheorem]{Remarques}

\newcommand{\GG}{\mathbb{G}}
\newcommand{\NN}{\mathbb{N}}
\newcommand{\ZZ}{\mathbb{Z}}
\newcommand{\QQ}{\mathbb{Q}}
\newcommand{\RR}{\mathbb{R}}
\newcommand{\CC}{\mathbb{C}}
\newcommand{\Ker}{\mathrm{Ker}}
\newcommand{\id}{\mathrm{id}}
\newcommand{\Hom}{\mathrm{Hom}}

\newcommand{\Aut}{\mathrm{Aut}}
\newcommand{\Gal}{\mathrm{Gal}}

\newcommand{\Kt}{K}
\newcommand{\Ktnr}{\Knr}
\newcommand{\Rt}{R}
\newcommand{\Rtnr}{\Rnr}

\newcommand{\Galnr}{\Gamma^\mathrm{n.r.}}
\newcommand{\Knr}{K^{\mathrm{n.r.}}}
\newcommand{\Lnr}{L^{\mathrm{n.r.}}}
\newcommand{\Rnr}{R^{\mathrm{n.r.}}}

\newcommand{\Lt}{L}
\newcommand{\Ltnr}{\Lnr}

\newcommand{\Ht}{\widetilde{H}}

\newcommand{\Gc}{\mathcal{G}}

\newcommand{\Ac}{\mathcal{A}} 
\newcommand{\Cc}{\mathcal{C}} 
\newcommand{\Fc}{\mathcal{F}}

\newcommand{\typebt}{\widetilde{\mathcal{T}}}
\newcommand{\ImmBT}{\mathscr{B}}
\newcommand{\DynD}{\mathscr{D}}

\defcitealias{stacks-project}{Stacks}
\defcitealias{SGA3}{SGA3}

\newcommand{\sga}[1]{\citepalias[#1]{SGA3}}

\begin{document}

\title[Arithmétique des sous-groupes de Bruhat-Tits (cas hensélien)]{Arithmétique des sous-groupes de Bruhat-Tits sur un anneau de valuation discrète hensélien}

\author[A.\ Zidani]{Anis Zidani}
\address{Anis Zidani\\
Christian-Albrechts-Universität zu Kiel\\
Mathematisches Seminar\\
Heinrich-Hecht-Platz 6\\
24118 Kiel, Deutschland\\
and Institute of Mathematics ”Simion Stoilow” of the Romanian Academy\\
21 Calea Grivitei Street\\
010702 Bucharest\\
Romania}

\date{\today}

\begin{abstract} 
L'objectif de cet article est de raffiner certains aspects de l'étude cohomologique des sous-groupes de Bruhat-Tits. Il se veut complémentaire aux travaux réalisés dans l'article de 1987. En guise d'application, on obtient un résultat "à la Grothendieck-Serre" pour les sous-groupes de Bruhat-Tits d'un groupe semi-simple simplement connexe et le calcul exact de l'obstruction dans le cas des groupes adjoints quasi-déployés.
\end{abstract}

\maketitle

\bigskip
\bigskip

\noindent{\bf Mots clés :} Schémas en groupes, Groupes algébriques, Groupes réductifs, Théorie de Bruhat-Tits, Cohomologie galoisienne, Anneaux de valuation discrète.

\medskip


\bigskip

\tableofcontents

\section*{Introduction}

Cet article a pour objectif de développer les aspects cohomologiques de la théorie de Bruhat-Tits. Il peut être vu comme un complément à \cite{BT3}.
\medskip

Rappelons que l'étude cohomologique de \cite{BT3} est concentrée sur les cocycles anisotropes (cf. \cite[3.6.]{BT3}) et les résultats de décomposition qui en découlent (cf. \cite[3.12. Théorème.]{BT3}). Cet article vise une toute autre problématique que l'on explique ci-dessous. 
\medskip

Considérons un anneau de valuation discrète hensélien $R$, de corps de fractions $K$. Notons $\Rnr$ son hensélisé strict et $\Knr$ le corps de fractions de $\Knr$. C'est l'extension maximale non ramifiée de $K$. Elle est galoisienne et on note son groupe de Galois $\Galnr$. Le corps résiduel de $R$ est noté $\kappa$, et n'est pas supposé nécessairement parfait.
\medskip

Étant donné un groupe réductif $G$ sur $K$, on se demande s'il est possible de comprendre le noyau suivant :
$$\Ker \left(H^1(\Galnr,G(\Knr)_{\widetilde{\mathcal{F}}})\rightarrow H^1(\Galnr,G(\Knr))\right)$$
où $\widetilde{\mathcal{F}}$ est une facette $\Galnr$-invariante de l'immeuble $\ImmBT(G_{\Knr})$ et $G(\Knr)_{\widetilde{\mathcal{F}}}$, son stabilisateur par l'action de $G(\Knr)$.
\medskip

Plus généralement, on introduit la notion de \textit{sous-groupe global}. Un sous-groupe de $G(K)$ est dit global s'il est ouvert pour la topologie adique, et s'il contient $G(K)^+$, i.e. le sous-groupe engendré par les $K$-points des sous-groupes de racines de $G$. L'étude de l'action sur l'immeuble de ces sous-groupes est d'un grand intérêt. En effet, ils agissent transitivement sur les couples $(\Ac,\Cc)$ tels que $\Ac$ est un appartement de $\ImmBT(G)$ et $\Cc$ une chambre de $\Ac$ (cf. lemme \ref{BNpaire}).
\medskip

On peut également considérer un sous-groupe global $\widetilde{H}$ de $G(K^\mathrm{n.r.})$ invariant pour l'action de $\Galnr$. Ces objets sont alors une généralisation des sous-groupes considérés par Tits dans \cite[3.5.]{BT3} (cf. remarque (2) de \ref{RmqSSGroupesGlobaux}).
\medskip

Dans ce cas, tout comme dans \cite{BT3}, on peut s'intéresser à des questions plus générales impliquant $\widetilde{H}$. En somme, on peut étudier le noyau :
$$\Ker\left(H^1(\Galnr,\widetilde{H}_{\widetilde{\mathcal{F}}})\rightarrow H^1(\Galnr,\widetilde{H})\right)$$
où $\widetilde{H}$ est cette fois un sous-groupe global $\Galnr$-invariant de $G(\Ktnr)$ et $\widetilde{H}_{\widetilde{\mathcal{F}}}$, le stabilisateur de $\widetilde{\mathcal{F}}$ sous l'action de $\widetilde{H}$.
\medskip

Quelques techniques de cohomologie des groupes classiques permettent de montrer la bijection (cf. point (1) (a) du Théorème \ref{NoyauH1}) :
\begin{equation}
    \tag{$\ast$}
    (\mathrm{Orb}(\widetilde{\Fc})_{\widetilde{H}})^{\Galnr}/H \overset{\sim}{\rightarrow} \Ker \left( H^1(\Galnr,\widetilde{H}_{\widetilde{\Fc}}) \rightarrow H^1(\Galnr,\widetilde{H}) \right ).
\end{equation}
où l'on a posé $H:=\widetilde{H}^{\Galnr}$. En d'autres termes, le noyau est en correspondance avec les éléments $\Galnr$-invariants de l'orbite $\widetilde{\Fc}$ par $\widetilde{H}$, modulo l'action de $H$.
\medskip

Dès lors, on constate que Bruhat et Tits s'étaient déjà penchés implicitement sur la question dans \cite{BT2}. En effet, par exemple, le résultat \cite[5.2.10.(ii) Proposition.]{BT2} signifie entre autres que $(\mathrm{Orb}(\widetilde{\Fc})_{G(\Knr)})^{\Galnr}/G(K)$ est trivial lorsque $G$ est semi-simple simplement connexe, quasi-déployé sur $\Knr$. En conséquence, $\Ker \left( H^1(\Galnr,G(\Knr)_{\widetilde{\Fc}}) \rightarrow H^1(\Galnr,G(\Knr) \right )$ est trivial dans ce cas de figure (cf. Remarque \ref{CasSSSC}).
\medskip

Également, dans \cite[5.2.13.]{BT2}, Bruhat et Tits donnent un cas de figure où \linebreak $(\mathrm{Orb}(\widetilde{\Fc})_{G(\Knr)})^{\Galnr}/G(K)$ est non trivial. Dans cet exemple, $G$ est quasi-déployé et adjoint.

En observant attentivement cet exemple, on observe que Bruhat et Tits raisonnent essentiellement au niveau des types, et donc au niveau des diagrammes de Dynkin affines. Il s'avère que ce phénomène est tout à fait général.
\medskip

En effet, on prouve dans cet article que la bijection $(\ast)$ reste toujours satisfaite si on la réduit au niveau des types. On obtient alors (cf. point (2) (a) du Théorème \ref{NoyauH1}) :
\begin{equation}
    \tag{$\ast'$}
    \left(\{\omega \cdot\typebt\prec \typebt_{\mathrm{max}} \mid \omega \in \widetilde{H} \}^{\Galnr}\right)/\,H \overset{\sim}{\rightarrow} \Ker\left(H^1(\Galnr,\widetilde{H}_{\widetilde{\mathcal{F}}})\rightarrow H^1(\Galnr,\widetilde{H})\right).
\end{equation}

Expliquons les objets en présence. Rappelons que l'indice de Tits affine de $G$ est la donnée de son diagramme de Dynkin affine sur $\Knr$, d'une action de $\Galnr$, et d'un ensemble $\Galnr$-stable de sommets, que l'on note $\typebt_{\mathrm{max}}$. Ce dernier est aussi le type de la plus grande facette $\Galnr$-invariante dans $\ImmBT(G_{\Knr})$ (aussi appelée $\Galnr$-chambre). Le type $\typebt$ est défini comme étant le type de $\widetilde{\Fc}$.
\medskip

En conséquence, la bijection $(*')$ donne une manière explicite et combinatoire de calculer le noyau, dépendant uniquement de l'indice de Tits affine de $G$ muni de l'action naturelle de $\widetilde{H}$ et de $H$. C'est le résultat théorique principal de cet article.
\medskip

De ce théorème, on en déduit immédiatement que $\Ker\left(H^1(\Galnr,\widetilde{H}_{\widetilde{\mathcal{F}}})\rightarrow H^1(\Galnr,\widetilde{H})\right)$ est trivial lorsque $\widetilde{H}$ agit trivialement sur l'indice de Tits affine. C'est notamment le cas lorsque $G$ est semi-simple simplement connexe et quasi-déployé sur $\Knr$ d'après \cite[5.2.10.(i) Proposition.]{BT2}. On retrouve alors le résultat de Bruhat et Tits indiqué plus haut.
\medskip

Par ailleurs, grâce à la bijection $(\ast')$, le cas $G$ quasi-déployé et adjoint peut être compris entièrement, généralisant ainsi l'exemple \cite[5.2.13.]{BT2} de Bruhat et Tits. On prouve ainsi dans cet article :

\begin{thm*}[cf. Théorème \ref{ThmStabQSplit}]
    Soit $G$ un groupe semi-simple adjoint et quasi-déployé sur $K$. Soit également $\widetilde{\mathcal{F}}$, une facette $\Galnr$-invariante de l'immeuble $\ImmBT(G_{\Knr})$. Alors le noyau :
    $$\Ker\left(H^1(\Galnr,G(\Knr)_{\widetilde{\mathcal{F}}})\rightarrow H^1(\Galnr,G(\Knr)\right)$$
    est de cardinal $2^k$ où $k$ est un entier majoré par le nombre de facteurs restriction de Weil d'un groupe absolument presque simple de type ${}^2D_n$ (pour $n\geq 4$) ou ${}^2A_{4n+3}$ (pour $n\geq 0$) déployé par une extension non ramifiée.
\end{thm*}

Ceci étant, on peut également s'intéresser au noyau suivant :
$$\Ker \left(H^1(\Galnr,G(\Knr)^0_{\widetilde{\mathcal{F}}})\rightarrow H^1(\Galnr,G(\Knr))\right)$$
où $G(\Ktnr)^0$ est le sous-groupe engendré par les sous-groupes parahoriques sur $\Ktnr$, aussi appelé la composante résiduellement neutre de $G_{\Ktnr}$.
\medskip

Cette question est nettement plus délicate. Malgré nos efforts et notre exploration de la littérature, on ignore s'il existe des situations où il est non trivial.
\medskip

Le cas où $\widetilde{\mathcal{F}}$ est un point hyperspécial (cf. la définition \ref{DefHyperspecial}) est en fait trivial d'après la conjecture de Grothendieck-Serre dans le cas d'un anneau de valuation discrète hensélien. Là encore, le résultat a en fait été déjà prouvé par Bruhat et Tits lorsque $G$ est semi-simple dans \cite[5.2.14. Proposition.]{BT2} en utilisant la bijection $(\ast)$. On montre également dans cet article que la preuve peut être ajustée pour montrer directement le cas réductif. C'est l'objet de la proposition \ref{CasReductif}.
\medskip

Un autre cas où l'on peut prouver la trivialité est, une fois encore, le cas quasi-déployé adjoint. C'est l'objet du théorème \ref{ThmParahoQSplit}.\\

Pour terminer, faisons une observation sur les hypothèses de l'article. Le corps résiduel $\kappa$ de $R$ n'est pas supposé parfait, contrairement à l'article \cite{BT3} de Bruhat et Tits. Le groupe $G$ n'est pas non plus supposé quasi-déployé sur $\Knr$, bien qu'il s'agisse du cadre des théorèmes de \cite{BT2} (rappelons d'ailleurs que si $\kappa$ est parfait, alors $G$ est quasi-déployé sur $\Ktnr$ comme mentionné dans \cite[5.1.1.]{BT2}). Il est donc nécessaire de faire quelques rappels sur la théorie de Bruhat-Tits dans cette généralité, notamment expliquer pourquoi l'immeuble de $G$ existe. C'est l'objet de la partie \ref{SectionExistenceImmeuble}.\\

\newpage

\noindent{\bf Remerciements.}
${}$
\medskip

L'auteur remercie Philippe Gille et Ralf Köhl, pour leurs soutiens, leurs accompagnements et leur relecture du présent article.

L'auteur est également reconnaissant vis-à-vis de la Studienstiftung des deutschen Volkes pour avoir supporté financièrement ce projet. Il a été également soutenu par le projet “Group schemes, root systems, and related representations”, financé par l'Union Européenne - NextGenerationEU à travers le "Romania’s National Recovery and Resilience Plan" (PNRR) call no. PNRR-III-C9-2023-I8, Project CF159/31.07.2023, et coordonné par le Ministre de la Recherche, de l'Innovation et de la Numérisation (MCID) de Roumanie. 

\section*{Notations et conventions} 

Pour tout corps $k$, la notation $k^s$ désigne une clôture séparable de $k$.
\medskip

On utilise la définition de groupe réductif de Chevalley et Borel (cf. \cite{BorelLinAlgGroups}). En particulier, ils sont affines, lisses et connexes.
\medskip

Soulignons que l'extension maximale non ramifiée d'un corps complet n'est pas toujours complète. Par exemple, l'extension maximale non ramifiée de $\kappa((t))$ n'est pas $\kappa^s((t))$ si $\kappa^s/\kappa$ est infini.

\section{Quelques rappels sur l'immeuble de Bruhat-Tits} \label{SectionExistenceImmeuble}

Faisons quelques rappels sur l'immeuble et son existence. Pour être le plus général possible, on suppose dans cette section que $\Kt$ n'est pas nécessairement hensélien. Les résultats de cette section seront en fait disponibles dans le livre \cite{RousseauNewBook} non encore publié. Par commodité, on se propose de réaliser cette section indépendamment de cette référence.\\

Dans \cite[Définitions 2.1.12]{TheseRousseau}, Rousseau propose une définition d'un immeuble associé à un groupe réductif $G$ sur $\Kt$. Comme prouvé dans le \cite[Théorème 2.1.14]{TheseRousseau}, cet immeuble existe si et seulement si l'immeuble au sens de \cite{BT1}, c'est-à-dire construit à partir d'une donnée radicielle valuée, existe ; et il est unique à isomorphisme près.

Par ailleurs, comme indiqué dans \cite[Théorème 2.1.14.2)c)]{TheseRousseau} et \cite[Théorème 2.1.15.c)]{TheseRousseau}, il est possible de canoniser cet immeuble en le construisant comme le produit de l'immeuble du groupe dérivé $D(G)$ avec une partie vectorielle donnée par le radical $R(G)$. Un immeuble sous cette forme est appelé \textit{immeuble centré}. Un immeuble centré est unique à unique isomorphisme (d'immeubles centrés) près.

Notons que l'immeuble construit à partir d'une donnée radicielle valuée associée à $G$ s'identifie avec l'immeuble de $D(G)$.
\medskip

L'immeuble proposé par Rousseau pour $G$ est exactement l'immeuble étendu sous la terminologie moderne. On le note $\ImmBT^e(G)$. Comme dit précédemment, il se décompose en un produit $\ImmBT(G) \times V_G$ où $V_G$ est la partie vectorielle de l'immeuble et où $\ImmBT(G)$ est l'immeuble de $D(G)$ (ou l'immeuble au sens d'une donnée radicielle valuée de $G$ comme dans \cite{BT1}). La partie $\ImmBT(G)$ est donc l'immeuble (réduit) de $G$ sous la terminologie moderne.

En particulier, lorsque $G$ est semi-simple, $\ImmBT^e(G)=\ImmBT(G)$ est unique à unique isomorphisme près et correspond à l'immeuble au sens d'une donnée radicielle valuée de $G$.
\medskip

Un autre point important à considérer est de savoir si l'immeuble $\ImmBT^e(G)$ est bornologique, c'est-à-dire si les stabilisateurs de parties bornées sont bornés, ou plus précisément s'il vérifie les conditions équivalentes données en \cite[Théorème 2.2.11]{TheseRousseau}. Il s'avère que d'après \cite[Corollaire 5.2.4.]{TheseRousseau}, un immeuble (étendu) est toujours bornologique lorsque $\Kt$ est hensélien.
\medskip

Comme indiqué en \cite[Exemples 2.2.14.f)]{TheseRousseau}, la question de l'existence se ramène au cas presque-simple et au cas des tores. Or, le cas des tores, lorsque $\Kt$ est hensélien, est traité en \cite[Proposition 2.4.8.2)]{TheseRousseau}. Par ailleurs, \cite[Proposition 2.3.9.]{TheseRousseau} nous dit que la question se ramène au cas où $\Kt$ est complet.
\medskip

Notons que la preuve de \cite[Proposition 5.1.5.]{TheseRousseau} montre exactement que l'immeuble d'une restriction de Weil séparable d'un groupe est naturellement un immeuble de ce dernier groupe. En particulier, le cas presque-simple se ramène au cas absolument presque-simple.
\medskip

Enfin, rappelons d'après Struyve (\cite{Struyve1} et \cite[Main Corollary.]{Struyve2}) que la conjecture \cite[13. Conjecture]{ImmeublesTypeAffineTits} est vérifiée. Cela signifie donc que tout groupe algébrique absolument presque simple sur un corps valué discrètement complet arbitraire admet une donnée radicielle valuée compatible avec la valuation du corps. On en déduit donc un immeuble associé à ce type de groupe d'après \cite{BT1}, et donc un immeuble au sens de Rousseau d'après la discussion précédente.
\medskip

En conclusion, on a :

\begin{prop}
    Soit $G$ un groupe réductif sur un corps hensélien valué discrètement $\Kt$. Alors $G$ admet un immeuble étendu unique à unique isomorphisme près. Il est par ailleurs bornologique.
\end{prop}

On peut en fait améliorer ce résultat grâce à \cite[Proposition 2.3.5]{TheseRousseau} :

\begin{thm}\label{ExistenceImmeuble}
    Soit $G$ un groupe réductif sur un corps valué discrètement $\Kt$. On suppose que $G$ a même rang relatif sur $\Kt$ et sur son hensélisé (ou encore son complété $\widehat{K}$).
    Alors $G$ admet un immeuble étendu, unique à unique isomorphisme près. Il est en outre bornologique. 
    
    Cet immeuble s'identifie canoniquement à celui de $G_{\widehat{K}}$ et ses appartements sont les \linebreak $\widehat{K}$-appartements correspondant aux $\widehat{K}$-tores déployés maximaux définis et déployés sur $\Kt$.
\end{thm}

\begin{rmq}
    Toutefois, on ne sait toujours pas si tout groupe réductif sur un corps valué arbitraire admet un immeuble non forcément bornologique.\\
\end{rmq}

Donnons maintenant quelques informations sur la partie vectorielle $V_G$ et les appartements.
\medskip

Notons $D:=G/D(G)$, le tore coradical de $G$. La partie vectorielle $V_G$ est donnée par $X_*(R(G))\otimes_\ZZ \RR \cong X_*(D)\otimes_{\ZZ}\RR$ et $G(\Kt)$ agit par translation grâce à $g\mapsto (\chi\mapsto-v(\chi(g)))$ de $G(\Kt)$ vers $\Hom(X^*(G),\RR)=\Hom(X^*(D),\RR)=X_*(D)\otimes_{\ZZ}\RR$. Il est vu à la fois comme espace affine sur lui-même et comme espace vectoriel. On note donc $G(\Kt)^1$ le fixateur sous $G(\Kt)$ de $V_G$ (ou de manière équivalente, d'un point de $V_G$). Autrement dit, c'est le noyau du morphisme $g\mapsto (\chi \mapsto -v(\chi(g)))$.

On en déduit immédiatement que $G(\Kt)^1$ est distingué dans $G(\Kt)$ et que le quotient est isomorphe à $\ZZ^r$ où 
$r$ est le rang du groupe des $\Kt$-caractères de $G$ (donc de $D$ ou de $R(G)$).

La définition de $G(\Kt)^1$ est fonctorielle en $G$ et sa construction est compatible aux extensions galoisiennes : pour toute extension galoisienne de corps valués $\Lt/\Kt$ de groupe de Galois $\Gamma$ (la valuation de $\Lt$ étant supposée $\Gamma$-invariante), le groupe $G(\Lt)^1$ est $\Gamma$-invariant et $(G(\Lt)^1)^{\Gamma}=G(\Lt)^1\cap G(\Kt)=G(\Kt)^1$. Par ailleurs, $G(\Kt)^1$ peut également être défini comme l'image réciproque de $D(\Kt)^1$ par $G(\Kt) \mapsto D(\Kt)$ (cf. \cite[Lemma 2.6.16]{kaletha_prasad_2023}). En conséquence, puisque $D$ est isogène à $R(G)$, on a $G(\Kt)^1=G(\Kt)$ si et seulement si $R(G)$ (ou $D$) ne contient aucun $\Kt$-sous-tore déployé (c'est en particulier le cas pour les groupes semi-simples).
\medskip

Plus généralement, comme expliqué dans \cite[2.1.7-2.1.11]{TheseRousseau}, étant donné un tore déployé maximal $S$ de $G$, un appartement $\Ac(S)$ de $\ImmBT^e(G)$ associé à $S$ est un espace affine sous $X_*(S)\otimes_{\ZZ} \RR$ muni d'une action $\nu : N_G(S)(\Kt) \rightarrow \Aut_{\mathrm{aff}}(\Ac(S))$ vérifiant les conditions de la définition \cite[2.1.8.a)]{TheseRousseau}. Il est unique à isomorphisme près.

Notamment, la restriction à $Z(\Kt)$ (où $Z:=Z_G(S)$) est une action par translation définie par $z\mapsto(\chi\mapsto-v(\chi(z)))$ allant vers : $$\Hom(X^*(Z),\RR)=\Hom(X^*(Z/D(Z)),\RR)=X_*(Z/D(Z)) \otimes_{\ZZ} \RR \cong X_*(S)\otimes_{\ZZ} \RR.$$ Par ailleurs, le noyau de $\nu$ est égal à $Z(\Kt)^1$.

Il s'avère qu'un appartement $\Ac(S)$ peut être aussi donné par $\Ac(S')\times V_G$, où $\Ac(S')$ est un appartement de $\ImmBT(G')$ associé au tore déployé maximal $S':=S\cap D(G)$. Sous cette forme, $\Ac(S)$ est appelé \textit{appartement centré} de $G$ associé à $S$. Un tel appartement a une structure affine, mais également une structure vectorielle donnée par $V_G$. Il est unique à unique isomorphisme d'appartements centrés près. Dans la suite, la notation $\Ac(S)$ désigne l'appartement centré de $G$ associé à $S$. Un tel appartement de $\ImmBT^e(G)$ respecte bien entendu la décomposition $\ImmBT^e(G)=\ImmBT(G)\times V_G$.\\

Intéressons-nous enfin aux facettes et aux types. 
\medskip

Rappelons qu'une facette $F$ désigne la réalisation géométrique ouverte du polysimplexe qu'elle représente. Son adhérence topologique $\overline{F}$ dans $\ImmBT(G)$ est exactement l'union (disjointe) de ses sous-polysimplexes ouverts (donc de ses sous-facettes) d'après \cite[(2.5.10.)]{BT1}. Une facette $F$ est dite incidente à une facette $F'$ si l'on a l'inclusion $\overline{F}\subseteq \overline{F'}$. On note alors $F \prec F'$.
\medskip

Par ailleurs, toute facette est incluse dans l'adhérence d'une chambre (qui est par définition une facette maximale pour l'incidence, ou encore de dimension maximale). 

L'adhérence d'une chambre est en correspondance naturelle avec le graphe de Dynkin de l'échelonnage de $G$. Ce graphe est appelé dans \cite{kaletha_prasad_2023} \textit{diagramme de Dynkin affine relatif}, et dans \cite{Corvallis} \textit{diagramme de Dynkin local relatif} ou encore \textit{$\Kt$-graphe résiduel} chez \cite{BT3}.

Le type d'une facette est alors défini comme étant son image sous cette correspondance. Cette image ne dépend pas du choix de l'adhérence d'une chambre où l'on a inclus la facette. En conséquence, deux facettes de même type dans la même adhérence d'une chambre sont égales.
\medskip

Comme indiqué précédemment, l'existence d'un immeuble $\ImmBT(G)$ pour $G$ est équivalente à l'existence d'une donnée radicielle valuée pour $G$. Cette dernière permet d'en déduire un double système de Tits muni d'un morphisme adapté dont l'immeuble associé est exactement $\ImmBT(G)$ d'après \cite[6.5. Théorème.]{BT1}.  
\medskip

On en déduit donc un morphisme type, noté $\xi$, de $G(\Kt)$ dans le groupe des automorphismes du diagramme de Dynkin affine relatif d'après \cite[1.2.16]{BT1}. Son image est notée $\Xi$ et son noyau est noté $G(\Kt)^c$. Il y a donc un isomorphisme $G(\Kt)/G(\Kt)^{c}\cong \Xi$. On peut aussi restreindre à $G(\Kt)^1$ ce morphisme. Son image est notée $\Xi^1$ et son noyau est noté $G(\Kt)^b:=G(\Kt)^{c}\cap G(\Kt)^1$. On en déduit un isomorphisme $G(\Kt)^1/G(\Kt)^b \cong \Xi^1$.
\medskip

Le lecteur intéressé par des exemples peut analyser le cas de $\mathrm{GL}_n$. En effet, il existe une manière d'interpréter l'immeuble et les types dans ce cas de figure au travers de chaînes de réseaux (cf. \cite{BTClassique1}).
\medskip

Le morphisme type mesure comment le type d'une facette change par action sous un élément de $G(\Kt)$. Autrement dit, une facette $F$ de type $T$ est telle que $g\cdot F$ est de type $\xi(g)\cdot T$ pour tout $g\in G(\Kt)$. Notons $\Xi_T$ le sous-groupe des $w\in\Xi$ tels que $w\cdot T=T$. \linebreak On peut également vérifier que $G(\Kt)_F$ se surjecte sur $\Xi_T$ et que son noyau vaut $G(\Kt)^{c}_F$, d'où un isomorphisme $G(\Kt)_F/G(\Kt)^{c}_F\cong \Xi_T$. On définit également tout cela de manière analogue pour $G(\Kt)^1$ et $\Xi^1$. Cf. \cite[1.2.13 - 1.2.20]{BT1} et \cite[2.7.]{BT1}.
\medskip

Le type d'une facette peut être aussi vu comme étant l'orbite de cette facette sous un groupe agissant transitivement sur les chambres tout en préservant les types. C'est notamment le cas de $G(\Kt)^c$, $G(\Kt)^b$ et même de $G(\Kt)^+$ (cf. le lemme \ref{HTransitif}).

\begin{rmq}\label{Xi1etXiDifferents}
    On n'a pas nécessairement égalité entre $\Xi^1$ et $\Xi$. Pour simplifier, considérons le cas où $\kappa$ est parfait. Pour réaliser le contre-exemple efficacement, on utilise le morphisme de Kottwitz (cf. \cite[11.5]{kaletha_prasad_2023}). Il s'agit d'un morphisme $G(\Kt)\rightarrow \pi_1(G)_I$ (fonctoriel en $G$) dont le noyau est $G(\Kt)^0$, le sous-groupe engendré par les sous-groupes parahoriques, aussi appelé la composante résiduellement neutre de $G(\Kt)$, (cf. \cite[Proposition 11.5.4]{kaletha_prasad_2023}), et dont l'image réciproque des éléments de torsion est $G(\Kt)^1$ (cf. \cite[Lemma 11.5.2]{kaletha_prasad_2023}). Le $\Gal(K^s/K)$-module $\pi_1(G)$ est le groupe fondamental algébrique, défini dans \cite[11.3]{kaletha_prasad_2023}. En conséquence, $\pi_1(G)_I$ désigne le $\Galnr$-module obtenu en prenant les coinvariants.
    
    Considérons le cas où $G=\mathrm{GL}_2$. Un calcul immédiat montre que $\pi_1(G)_I=\ZZ$ et est donc sans torsion. On en déduit que $G(\Kt)^0=G(\Kt)^1$. Comme $G(\Kt)^0$ agit trivialement sur les types (cf. \cite[5.2.12.(i) Proposition.]{BT2}), il en est donc de même pour $G(\Kt)^1$. D'où $\Xi^1=0$.

    Par ailleurs, $G(\Kt)$ n'agit pas trivialement sur les types. En effet, les deux parahoriques suivants sont associés à des points de types différents (cf. \cite[Chapter 3.1]{kaletha_prasad_2023}) :
    $$\begin{pmatrix}
    \Rt & \Rt\\
    \Rt & \Rt
    \end{pmatrix}  \text{ et } \begin{pmatrix}
    \Rt & t\Rt\\
    t^{-1}\Rt & \Rt
    \end{pmatrix}
    $$
    alors qu'ils sont conjugués par la matrice $\begin{pmatrix}
    t & 0\\
    0 & 1
    \end{pmatrix} \in G(\Kt)$.
\end{rmq}

Pour finir, prouvons que $\Xi$ est abélien fini. Pour cela, on a besoin de quelques résultats :

\begin{lem}\label{decompDG}
    Soit $Z$, un sous-groupe de Levi de $G$. On a $D(G(\Kt))=G(\Kt)^+\,D(Z(\Kt))$.
\end{lem}

\begin{proof}
    Puisque $G(\Kt)^+$ est parfait (\cite[6.4. Corollaire.]{HomAbstraits}), on a $G(\Kt)^+\subset D(G(\Kt))$. On en déduit l'inclusion $G(\Kt)^+\,D(Z(\Kt))\subset D(G(\Kt))$. Réciproquement, puisque \linebreak $G(\Kt)=G(\Kt)^+\,Z(\Kt)$ (\cite[6.11.(i) Proposition.]{HomAbstraits}), on peut prendre un élément \linebreak $d\in D(G(\Kt))$ de la forme $g_1z_1g_2z_2(g_1z_1)^{-1}(g_2z_2)^{-1}$, avec $g_1$, $g_2$ dans $G(\Kt)^+$ (resp. \linebreak $z_1$, $z_2$ dans $Z(\Kt)$). Puisque $G(\Kt)^+$ est distingué dans $G(\Kt)$, on peut considérer le quotient $G(\Kt)/G(\Kt)^+$ et voir que l'image de $d$ dans $G(\Kt)/G(\Kt)^+$ est égal à celle de $z_1z_2z_1^{-1}z_2^{-1}$, d'où $d\in G(\Kt)^+\,D(Z(\Kt))$. Comme $D(G(K))$ est engendré par ce type d'éléments, on en déduit l'inclusion $D(G(K))\subset G(\Kt)^+\,D(Z(\Kt))$ comme voulu.
\end{proof}

\begin{prop}\label{GKbQuotAb}
    $G(\Kt)^b$ est un sous-groupe distingué de $G(\Kt)$ dont le quotient est abélien de type fini et dont le nombre de générateurs est majoré par le rang relatif de $G$.
\end{prop}

\begin{proof}
    Prenons $Z$ un sous-groupe de Levi minimal de $G$. D'après le lemme \ref{decompDG}, on a $D(G(\Kt))=G(\Kt)^+\,D(Z(\Kt))$. Or, d'une part $G(\Kt)^b$ contient $G(\Kt)^+$, et d'autre part $D(Z(\Kt))\subset D(Z)(\Kt)\subset Z(\Kt)^1\subset G(\Kt)^b$. Le sous-groupe $G(\Kt)^b$ contient donc $D(G(\Kt))$ est donc distingué de quotient abélien.
    
    Par ailleurs, il y a une surjection $Z(\Kt)/Z(\Kt)^1\rightarrow G(\Kt)/G(\Kt)^+Z(\Kt)^1$ de telle sorte à ce que $G(\Kt)/G(\Kt)^+Z(\Kt)^1$ soit abélien de type fini puisque $Z(\Kt)/Z(\Kt)^1$ l'est. Son nombre de générateurs est donc majoré par celui de $Z(\Kt)/Z(\Kt)^1$, qui est simplement le rang relatif de $G$. On utilise alors le lemme \ref{décompSSGroupes} qui nous dit que $G(\Kt)^b=G(\Kt)^+Z(\Kt)^1$ pour conclure.
\end{proof}

\pagebreak

On en déduit alors ce que l'on voulait :

\begin{prop}
    Le groupe $G(\Kt)/G(\Kt)^{c}\cong \Xi$ (et donc $G(\Kt)^1/G(\Kt)^{b}\cong \Xi^1$) est abélien fini.
\end{prop}

\begin{proof}
    Notons que l'on a l'isomorphisme $G(\Kt)/G(\Kt)^{c}\cong \Xi$. Par ailleurs, puisque $G(\Kt)^b\subset G(\Kt)^c$, le corollaire \ref{GKbQuotAb} donne que $\Xi$ est abélien. D'autre part, le fait qu'il y ait un nombre fini de manières de permuter un nombre fini de sommets implique que $\Xi$ est fini. Par conséquent, il en est de même pour $G(\Kt)^1/G(\Kt)^{b}\cong \Xi^1\subset \Xi$.
\end{proof}

\section{Sous-groupes globaux et nouvelles notions}

On considère une extension galoisienne non ramifiée éventuellement infinie $\Lt/\Kt$, de groupe de Galois $\Gamma$. On peut donc supposer avoir l'inclusion $\Lt\subset \Ktnr$.
\medskip

Dans toute la suite, on va considérer les sous-groupes suivants de $G(\Kt)$ :

\begin{déf}\label{DefGlobal}
    Soit $H$ un sous-groupe ouvert de $G(\Kt)$.
    \begin{itemize}
        \item On dit que $H$ est un \textbf{sous-groupe global} de $G(\Kt)$ si $G(\Kt)^+\subset H$.
        \item On dit de plus que $H$, supposé global, est 
        $\boldsymbol{\Lt}$\textbf{-conforme} (ou juste \textbf{conforme} si $\Lt=\Kt$) si $H$ préserve les $\Lt$-types, ou de façon équivalente si, $H\subset G(\Lt)^{c}$. Par ailleurs, $H$ est dit \textbf{très conforme} si $H$ est $\Ktnr$-conforme. On dit aussi que $H$ est \textbf{uniforme} si $H\subset G(\Kt)^1$.
        \item On dit que $H$ est $\boldsymbol{\Lt}$\textbf{-bon} (ou juste \textbf{bon} si $\Lt=\Kt$) si $H$ est uniforme et $\Lt$-conforme, ou de façon équivalente si, $H\subset G(\Lt)^{b}$. Par ailleurs, $H$ est dit \textbf{très bon} si $H$ est $\Ktnr$-bon.
        \item On définit également $H^1$, $H^b$, $H^{tb}$ $H^{c}$, $H^{tc}$ comme étant les sous-groupes obtenus en prenant l'intersection de $H$ avec respectivement $G(\Kt)^1$, $G(\Kt)^b$, $G(\Ktnr)^{b}$, $G(\Kt)^{c}$ et $G(\Ktnr)^{c}$.
        \item Pour toute partie $\Omega$ de $\ImmBT(G)$, on note $H_{\Omega}$ (resp. $H_{\Omega}^{\mathrm{f}}$) le \textbf{stabilisateur} (resp. \textbf{fixateur}) de $\Omega$ sous $H$. Si on prend plusieurs parties $(\Omega_i)_{i \in I}$, on note \linebreak $H_{(\Omega_i)_{i \in I}}:= \bigcap_{i\in I}H_{\Omega_i}$. Ce dernier sous-groupe est appelé le \textbf{multistabilisateur} de $(\Omega_i)_{i \in I}$ sous $H$.
    \end{itemize}
\end{déf}

Observons que $$H^b=(H^1)^c=(H^c)^1=H^1\cap H^c \text{ et } H^{tb}=(H^1)^{tc}=(H^{tc})^1=H^1\cap H^{tc}.$$
En effet, il suffit de vérifier le résultat lorsque $H=G(\Kt)$. Dans ce cas, cela découle des définitions.

\medskip

Comme on le verra plus loin dans le corollaire \ref{TrèsConformeEstConforme}, un sous-groupe global $L$-conforme (resp. $L$-bon) est conforme (resp. bon), mais la réciproque est fausse.

\begin{rmqs}
${}$
    \begin{enumerate}
        \item On ne prend pas la même convention que Bruhat et Tits dans \cite{BT2}, et que Prasad dans \cite{kaletha_prasad_2023}. Pour Prasad, $G(\Kt)^1_\Omega$ désigne le fixateur de $\Omega$ sous l'action de $G(\Kt)^1$, tandis que $G(\Kt)^{\dagger}_\Omega$ désigne le stabilisateur de $\Omega$ sous l'action de $G(\Kt)^1$. Bruhat et Tits prennent une convention analogue.

        \item Les notations alambiquées "$1$" et "$b$" étaient déjà présentes dans la littérature. On a donc choisi de leur donner un nom de telle sorte à les retenir plus aisément ("$1$" est associé au caractère "uniforme" et "$b$" est associé au caractère "bon"). On a également rajouté la notion de "conformité" (associée à "$c$"), qui, bien que pratique, n'était pas présente dans la littérature.

        \item On aurait pu définir une notion de caractère "$\Lt$-uniforme" (ou "très uniforme"). Toutefois, $(G(\Ktnr)^1)^{\Galnr}=G(\Kt)^1$ d'après la fin de la section \ref{SectionExistenceImmeuble}. Cela est donc équivalent à la notion de caractère "uniforme".
    \end{enumerate}
\end{rmqs}

\begin{rmq}
    Tout fixateur d'une partie $\Omega$ de $\ImmBT(G)$ sous l'action d'un sous-groupe global $H$ est le multistabilisateur de $(x)_{x\in \Omega}$ sous $H$. Si de plus $\Omega$ est une union finie de facettes, il s'agit également du multistabilisateur sous $H$ de la famille donnée par les sommets incidents à $\Omega$ (en nombre fini).
\end{rmq}

Le principal résultat au sujet des sous-groupes globaux est le suivant :

\begin{lem}\label{HTransitif}
    Un sous-groupe global agit transitivement sur les couples $(\Ac,\Cc)$ d'appartements et de chambres incluses dans cet appartement. Il agit également en préservant les types.
\end{lem}

\begin{proof}
    Il suffit de prouver le résultat pour $G(K)^+$. On peut également revenir au cas semi-simple adjoint. En effet, l'action sur l'immeuble se factorise par $Z(G)(\Kt)$ et \cite[Corollaire 6.3.]{HomAbstraits} implique que $G(\Kt)^+\rightarrow G^{\mathrm{ad}}(\Kt)^+$ est surjective.
    
    On sait qu'il s'existe une donnée de racine valuée associé à $G(\Kt)$. Les parahoriques (au sens de \cite{BT1}, c'est à dire les stabilisateurs de facettes sous l'action de $G(\Kt)^c$) sont décrits dans \cite[(7.1.1.)]{BT1} et engendrent donc $G(\Kt)^+Z(\Kt)^1$ ("$H$" vaut $Z(\Kt)^1$, avec $Z$ un sous-groupe de Levi minimal, puisque $G$ est semi-simple. Aussi, les groupes de racines affines engendrent $G(\Kt)^+$). Mais par définition, ce groupe vaut également $G(\Kt)^c$ (et même $G(\Kt)^b$ puisque $G$ est supposé semi-simple). Il agit donc transitivement sur les couples appartements-chambres qui nous intéressent (cf. \cite[(2.2.6)]{BT1}).
    Comme \linebreak $G(\Kt)^+\subset G(\Kt)^+Z(\Kt)^1=G(\Kt)^b$, on en déduit en particulier que $G(\Kt)^+$ agit en préservant les types.

    Prenons $\Ac$ et $\Ac'$ des appartements et, $\Cc\subset \Ac$ et $\Cc'\subset \Ac'$, des chambres de ces appartements. Prenons $g\in G(\Kt)^b$ tel que $g\cdot(\Ac,\Cc)=(\Ac',\Cc')$. Prenons $Z$ relativement à $\Ac$ et écrivons la décomposition $g=g^+z$ donnée par $G(\Kt)^b=G(\Kt)^+Z(\Kt)^1$. Puisque $Z(\Kt)^1$ fixe $\Ac$ (et donc $\Cc$), on a :
    $$g^+\cdot (\Ac,\Cc)=g^+\cdot (z\cdot (\Ac,\Cc))=g\cdot (\Ac,\Cc)=(\Ac',\Cc').$$

    D'où le résultat.
\end{proof}

On en déduit alors :

\begin{prop}\label{BNpaire}
    Choisissons un appartement $\Ac$ et une chambre $\Cc\subset \Ac$. Tout sous-groupe global conforme $K$ de $G(\Kt)$ définit une $BN$-paire saturée en posant $B=K_\Cc$ et $N=K_{\Ac}$, les stabilisateurs de $\Cc$ et $\Ac$ sous l'action de $K$. L'immeuble associé est exactement $\ImmBT(G)$ et son groupe de Weyl est le groupe de Weyl affine de l'immeuble.
    
    Modulo conjugaison par $K$, cette $BN$-paire ne dépend pas du choix du couple $(\Ac,\Cc)$.
\end{prop}

\begin{proof}
   D'après le lemme précédent, on rentre dans le cadre d'application de \cite[3.11. Proposition]{FiniteBNPairsTits}, qui nous donne le résultat.
\end{proof}

\begin{rmqs} \label{RmqSSGroupesGlobaux}
${}$
    \begin{enumerate}
        \item La terminologie \textit{sous-groupe global} est en fait inspirée de la proposition \ref{BNpaire} : un sous-groupe global est suffisamment gros pour déterminer un ensemble suffisamment riche de sous-groupes "locaux" donnés par les stabilisateurs de parties bornées de l'immeuble $\ImmBT(G)$.
        \item La notion de sous-groupe global de $G(\Ktnr)$ invariant par $\Galnr$ englobe celle des sous-groupes considérés dans \cite[3.5.]{BT3} dans le cas où le corps résiduel $\kappa$ est parfait. \linebreak En effet, Tits impose plutôt de contenir $G(\Knr)^0$, le sous-groupe engendré par les sous-groupes parahoriques sur $\Ktnr$, aussi appelé la composante résiduellement neutre de $G_{\Ktnr}$, au lieu de $G(\Knr)^+$, et $G(\Knr)^+ \subset G(\Knr)^0$ (cf. 3e paragraphe de \cite[5.2.11.]{BT2}).
    \end{enumerate}
\end{rmqs}

De ceci, on en déduit quelques résultats élémentaires autour des facettes et des sous-groupes globaux :

\begin{prop}\label{PropDesHFacettes}
    Soit $H$ un sous-groupe global de $G(\Kt)$ et deux facettes $\Fc$ et $\Fc'$ dans $\ImmBT(G)$. On a :
    \begin{enumerate}
        \item Le sous-groupe $H_{\Fc}$ est transitif sur les appartements contenant $\Fc$.
        \item Si $H_{\Fc'}\subset H_\Fc$, alors $\overline{\Fc}\subset \overline{\Fc'}$. La réciproque est vraie si $H$ est de plus conforme.
        \item On a $H_{\Fc} = H_\Fc$ si et seulement si $\Fc=\Fc'$. 
    \end{enumerate}
\end{prop}

\begin{proof}
${}$
    \begin{enumerate}
        \item Puisque $H$ est global, il suffit de montrer le résultat pour $H=G(\Kt)^+$. Il s'agit de \cite[Proposition 1.5.13.(1)]{kaletha_prasad_2023} appliqué au système de Tits de $G(\Kt)^+$ (cf. la proposition \ref{BNpaire}).

        \item[(2) et (3)] Observons que $H_{\Fc'}\subset H_\Fc$ implique $G(\Kt)^+_{\Fc'}\subset G(\Kt)^+_\Fc$. Comme $G(\Kt)^+$ induit un système de Tits dont l'immeuble est exactement $\ImmBT(G)$ (cf. la proposition \ref{BNpaire}), on a une correspondance entre les paraboliques du système de Tits pour l'inclusion (qui sont les stabilisateurs de facettes) et les facettes de l'immeuble pour l'incidence. D'où $\overline{\Fc}\subset \overline{\Fc'}$ si et seulement si $G(\Kt)^+_{\Fc'}\subset G(\Kt)^+_\Fc$. Le même raisonnement s'applique également à un sous-groupe conforme arbitraire.
    \end{enumerate}
\end{proof}

On a également des résultats de décomposition :

\begin{lem}\label{décompSSGroupes}
    Soit $H$ un sous-groupe global de $G(\Kt)$.
    On a $H=G(\Kt)^+H_{(\Ac,\Cc)}$, où $\Ac$ est un appartement de $\ImmBT(G)$ et $\Cc$ est une chambre dans $\Ac$.
    Par ailleurs, $H^c_{(\Ac,\Cc)}=H^{\mathrm{f}}_{\Ac}$ et $H^b_{(\Ac,\Cc)}=H^{1,\mathrm{f}}_{\Ac}$.
    En particulier, on a $G(\Kt)^b=G(\Kt)^+Z(\Kt)^1$ pour $Z$ un sous-groupe de Levi minimal de $G$.
\end{lem}

\begin{proof}
    L'inclusion réciproque est évidente. Étudions l'inclusion directe.

    Soit $h\in H$. Par transitivité de $G(\Kt)^+$ sur les couples appartements-chambres, il existe $g\in G(\Kt)^+$ tel que $g\cdot \Cc=h\cdot \Cc$ et $g\cdot \Ac = h\cdot \Ac$. Donc $h':=g^{-1}h\in H_{(\Ac,\Cc)}$. Donc $h=gh'$.

    Observons que $H^b_{(\Ac,\Cc)} = H^{b,\mathrm{f}}_\Ac$ puisque $H^b$ fixe les types, donc $\Cc$, et donc tout $\Ac$ puisque les sommets de $\Cc$ déterminent une base affine de $\Ac$. Par ailleurs, $H^{b,\mathrm{f}}_\Ac=H^{1,\mathrm{f}}_\Ac$ car $H^{1,\mathrm{f}}_\Ac$ fixe $\Cc$ et donc agit trivialement sur les types. Le même raisonnement prouve que $H^c_{(\Ac,\Cc)}=H^{\mathrm{f}}_{\Ac}$.
    
    On en déduit donc d'après la section \ref{SectionExistenceImmeuble} que $G(\Kt)^b_{(\Ac,\Cc)}$ est exactement le fixateur de l'appartement étendu $\Ac\times V_G\subset \ImmBT^e(G)$, c'est-à-dire $Z(\Kt)^1$, où $Z$ est le sous-groupe de Levi associé à $\Ac$. D'où la dernière décomposition.
\end{proof}

Ce dernier résultat permet d'en déduire une décomposition de Bruhat plus précise :

\begin{prop}[Décomposition de Bruhat]\label{Bruhat} Prenons $\Ac$ un appartement de $\ImmBT(G)$ et $\Cc$ une chambre de $\Ac$. Soit $H$ un sous-groupe global de $G(\Kt)$. On a : $H=G(\Kt)^+_{\Cc}\,H_{\Ac}\,G(\Kt)^+_{\Cc}$.\\ En particulier, $G(\Kt)=G(\Kt)^+_{\Cc}\,N(\Kt)\,G(\Kt)^+_{\Cc}$, où $N$ est le normalisateur du tore déployé maximal de $G$ associé à $\Ac$.
\end{prop}

\begin{proof}
    D'après \cite[Proposition 1.4.5.(1)]{kaletha_prasad_2023}, il existe une décomposition de Bruhat pour $G(\Kt)^+$ (puisque ce dernier détermine un système de Tits, cf. la proposition \ref{BNpaire}). \linebreak Par conséquent, $G(\Kt)^+=G(\Kt)^+_{\Cc}\,G(\Kt)^+_{\Ac}\,G(\Kt)^+_{\Cc}$.

    Observons que $H_\Cc=H_{(\Ac,\Cc)}\,G(\Kt)^+_{\Cc}=G(\Kt)^+_{\Cc}\,H_{(\Ac,\Cc)}$ puisque $G(\Kt)^+_{\Cc}$ agit transitivement sur les appartements contenant $\Cc$ d'après le lemme \ref{HTransitif}. Ensuite, constatons que \linebreak $H_\Ac=H_{(\Ac,\Cc)}\,G(\Kt)^+_{\Ac}$ puisque $G(\Kt)^+_{\Ac}$ agit transitivement sur les chambres de $\Ac$ d'après également le lemme  \ref{HTransitif}. Comme d'après le lemme \ref{décompSSGroupes}, $H=H_{(\Ac,\Cc)}\,G(\Kt)^+$, on a donc finalement :
    $$H=H_{(\Ac,\Cc)}\,(G(\Kt)^+_{\Cc}\,G(\Kt)^+_{\Ac}\,G(\Kt)^+_{\Cc}) = G(\Kt)^+_{\Cc}\left(H_{(\Ac,\Cc)}\,G(\Kt)^+_{\Ac}\right)\,G(\Kt)^+_{\Cc}=G(\Kt)^+_{\Cc}\,H_{\Ac}\,G(\Kt)^+_{\Cc}.$$
\end{proof}

Prouvons maintenant quelques résultats de compatibilité des sous-groupes globaux aux extensions non ramifiées. Avant cela, on a besoin de montrer le lemme élémentaire suivant :

\begin{lem}\label{GK+Galois}
    Soit $G'$ un groupe réductif sur un corps $K'$ et $L'/K'$, une extension galoisienne de groupe de Galois $\Gamma'$. Le sous-groupe $G'(L')^+$ est $\Gamma'$-invariant et on a :
    $$G'(\Kt')^+ \subset (G'(\Lt')^+)^{\Gamma'}=G'(\Lt')^+\cap G'(\Kt').$$
\end{lem}

\begin{proof}
    La première assertion provient de \cite[6.1.]{HomAbstraits}. En effet, $\sigma \in \Gamma'$ définit un isomorphisme $\sigma : G'_{\Lt'}\rightarrow G'_{\Lt'}$, et donc envoie $G'(\Lt')^+$ vers $G(\Lt')^+$. D'où la $\Gamma'$-invariance. 

    D'autre part, \cite[6.1.]{HomAbstraits} donne aussi $G'(\Kt')^+ \subset G'(\Lt')^+$. On a donc le résultat en utilisant la $\Gamma'$-invariance. 
\end{proof}

Ceci nous permet d'obtenir :

\begin{prop}\label{PropDesH}
    On a :
    \begin{enumerate}
        \item Tout sous-groupe global $H$ admet un plus grand sous-groupe global respectivement uniforme, bon, $\Lt$-bon, conforme, $\Lt$-conforme donné respectivement par $H^1$, $H^b$, $H\cap G(\Lt)^b$ $H^{c}$, $H\cap G(\Lt)^c$ (et donc en particulier un plus grand sous-groupe global respectivement très bon et très conforme donné par respectivement $H^{tb}$ et $H^{tc}$).
        \item Si $\Ht$ est un sous-groupe global respectivement uniforme, bon, conforme, $\Gamma$-invariant de $G(\Lt)$, alors $\Ht^{\Gamma}$ est un sous-groupe global respectivement uniforme, $\Lt$-bon, $\Lt$-conforme de $G(\Kt)$.
        \item Si $\Ht$ est un sous-groupe global $\Gamma$-invariant de $G(\Lt)$, alors $\Ht^1$, $\Ht^b$ et $\Ht^c$ sont également $\Gamma$-invariants.
    \end{enumerate}
\end{prop}

\begin{proof}
    ${}$
    \begin{enumerate}
        \item Pour le premier point, il suffit de montrer que les sous-groupes en question sont globaux. On a d'après les lemmes \ref{HTransitif}, \ref{GK+Galois} et le premier point du corollaire \ref{TrèsConformeEstConforme} :
        \[\begin{tikzcd}
        	&& {G(\Lt)^c\cap H} & {H^c} \\
        	{G(\Kt)^+} & {G(\Lt)^+\cap H} & {G(\Lt)^b\cap H} & {H^b} & {H^1.}
        	\arrow["\text{\ref{TrèsConformeEstConforme}.(1)}"',hook, from=1-3, to=1-4]
        	\arrow["\text{\ref{GK+Galois}}"',hook, from=2-1, to=2-2]
        	\arrow["\text{\ref{HTransitif}}"',hook, from=2-2, to=2-3]
        	\arrow[hook, from=2-3, to=1-3]
        	\arrow["\text{\ref{TrèsConformeEstConforme}.(1)}"',hook, from=2-3, to=2-4]
        	\arrow[hook, from=2-4, to=1-4]
        	\arrow[hook, from=2-4, to=2-5]
        \end{tikzcd}\]
        \item Puisque $\widetilde{H}^\Gamma \subset \widetilde{H}$, il suffit seulement de montrer que $G(\Kt)^+\subset \widetilde{H}^\Gamma$. Or, on a $G(\Kt)^+\underset{\text{\ref{GK+Galois}}}{\subset} (G(\Lt)^+)^\Gamma \subset \widetilde{H}^\Gamma$. D'où le résultat.
        \item Le troisième point se ramène au cas où $\Ht=G(\Lt)$. Pour $G(\Lt)^1$, cela a déjà été fait à la fin de la section \ref{SectionExistenceImmeuble}. Pour le reste, utilisons le lemme \ref{décompSSGroupes}. Étant donné un appartement $\Ac\subset \ImmBT(G_{\Lt})$ et une $\Lt$-chambre $\Cc\subset \Ac$, on a : $G(\Lt)^*=G(\Lt)^+ G(\Lt)^*_{(\Ac,\Cc)}$ pour $*\in \{b,c\}$. Or, $G(\Lt)^+$ est $\Gamma$-invariant d'après le lemme \ref{GK+Galois}. Il suffit donc de montrer que l'orbite sous Galois de $G(\Lt)^*_{(\Ac,\Cc)}$ est dans $G(\Lt)^*$ pour tout $*\in \{b,c\}$. 

        Or, le lemme \ref{décompSSGroupes} montre également que $G(\Lt)^b_{(\Ac,\Cc)}=G(\Lt)^{1,\mathrm{f}}_\Ac$ et que \linebreak $G(\Lt)^c_{(\Ac,\Cc)}=G(\Lt)^\mathrm{f}_\Ac$. Mais pour $\sigma \in \Gamma$, $\sigma(G(\Lt)^{1,\mathrm{f}}_\Ac)=G(\Lt)^{1,\mathrm{f}}_{\sigma(\Ac)}\subset G(\Lt)^b$, et de même $\sigma(G(\Lt)^\mathrm{f}_\Ac)=G(\Lt)^\mathrm{f}_{\sigma(\Ac)}\subset G(\Lt)^c$. Ceci donne donc le résultat comme voulu.
    \end{enumerate}
\end{proof}

Introduisons quelques notations supplémentaires qui vont nous servir par la suite :

\begin{déf}
    Soient $H$ un sous-groupe global de $G(K)$ et $F$ une facette de type $T$. Notons :
    \begin{itemize}
        \item $\Xi_H$, l'image de $H$ par $\xi$ (qui induit donc $H/H^{c} \cong \Xi^H$).
        \item $\Xi_{H,T}$, l'image de $H_F$ par $\xi$ (qui induit donc $H_F/H^{c}_F \cong \Xi_{H,T}$). C'est aussi l'ensemble $\{w\in \Xi_H \mid w\cdot T = T\}$ puisque $H^{c}\subset H$ agit transitivement et de manière conforme sur les chambres.
        \item $\mathrm{Orb}(F)_{H}$, l'orbite de $F$ par $H$.
        \item $\mathrm{Orb}(T)_{\Xi_H}$ (ou même $\mathrm{Orb}(T)_{H}$), l'orbite de $T$ par $\Xi_H$.
    \end{itemize}
\end{déf}

Généralisons maintenant la notion de facette et les objets associés. Cette généralisation est peu coûteuse pour la suite et ajoute une richesse supplémentaire à notre problème général.

\begin{déf}\label{DefMultifacette}
    Appelons \textbf{multifacette} toute union de facettes incluse dans la même adhérence d'une chambre. Pour un tel objet, on peut définir le \textbf{type} (ou \textbf{multitype}, pour insister sur le fait que cela est relatif à une multifacette) comme étant l'ensemble des types des différentes facettes la composant. Le morphisme $\xi$ s'étend également de manière naturelle et les résultats de la section précédente relatives aux facettes s'étendent également.

    On dit qu'une multifacette est \textbf{fortement invariante} par l'action d'un groupe si chacune des facettes la composant est invariante (il ne suffit donc pas que la multifacette soit invariante en tant qu'objet géométrique). On définit la même notion pour les multitypes.

    On dit également qu'un groupe agissant sur $\ImmBT(G)$ par automorphismes polysimpliciaux agit de manière \textbf{conforme} sur une multifacette $\Fc$ s'il l'envoie sur des multifacettes de même type.

    Si $\Fc$ est une multifacette de décomposition en facettes $\bigsqcup_{i\in I} \Fc_i$, alors pour tout sous-groupe global $H$ de $G(\Kt)$, on note $H_{(\Fc)}:=H_{(\Fc_i)_{i\in I}}:=\bigcap_{i\in I}H_{\Fc_i}$. Ce groupe est appelé \textbf{sous-groupe multistabilisateur de la multifacette} $\boldsymbol{\Fc}$ \textbf{relativement à} $\boldsymbol{H}$.

    Plus généralement, on utilise la notation $(\Fc)$ pour préciser que l'on regarde bien $\Fc$ en tant que multifacette et non en tant que partie de l'immeuble (on fait de même pour les multitypes).
\end{déf}

\begin{rmq}
    On voit donc que l'utilisation de multifacettes donne lieu à une plus grande famille de sous-groupes que les seuls stabilisateurs de facettes. En particulier, cela donne accès aux fixateurs de facettes, en prenant par exemple la multifacette associée aux sommets incidents à une facette.

    Remarquons toutefois que, dans le cas conforme, stabiliser une facette et préserver son type implique en fait de la fixer. Dans ce cas, le multistabilisateur d'une multifacette dont on préserve le type n'est autre que le fixateur de l'union des facettes la composant : \linebreak la notion de multistabilisateur n'a donc d'intérêt que si l'on considère des sous-groupes globaux non conformes.
\end{rmq}

\begin{rmq}
    Comme pour les facettes, l'adhérence topologique d'une multifacette \linebreak $\Fc:=\bigsqcup_i\Fc_i$ est exactement la réunion des sous-facettes des $\Fc_i$. En effet, cela est une conséquence du fait que $\overline{\Fc}=\overline{\bigcup_i\Fc_i}=\bigcup_i\overline{\Fc}_i$. On définit également la relation d'incidence comme étant une inclusion au niveau des adhérences.
\end{rmq}

\section{Quelques compléments sur la descente non ramifiée}

Rappelons que, d'après Rousseau dans \cite[Proposition 2.4.6]{TheseRousseau}, le groupe de Galois $\Gamma$ agit par automorphismes polysimpliciaux sur $\ImmBT(G_{\Lt})$ de manière compatible avec l'action de $G(\Lt)$. D'après le théorème de descente modérément ramifiée (\cite[Proposition 5.1.1.]{TheseRousseau}), l'ensemble des points fixes s'identifie de façon unique à $\ImmBT(G)$. On peut d'ailleurs choisir une métrique invariante sur $\ImmBT(G_{\Lt})$ (cf. \cite[\S 2.2]{TheseRousseau}) de telle sorte que $\Gamma$ agisse par isométrie (\cite[Remarque 2.4.7.(f)]{TheseRousseau}), et donc de telle sorte que $\ImmBT(G)\subset\ImmBT(G_{\Lt})$ soit un plongement isométrique. Sous ce choix, $\ImmBT(G)$ est également un fermé convexe de $\ImmBT(G_{\Lt})$. En effet, $\Gamma$ agit continûment sur $\ImmBT(G_{\Lt})$ (puisqu'il agit par isométries), d'où le caractère fermé. Le caractère convexe provient ensuite de l'unicité de la géodésique reliant deux points (puisque $\Gamma$ agit par isométries).
\medskip

En particulier, puisque $\Gamma$ agit par automorphismes polysimpliciaux, il envoie facettes sur facettes. En outre, cette action sur les facettes se factorise en une action sur les types (et même sur le diagramme de Dynkin affine relatif). Il suffit en effet de voir que, étant donné $\Fc$ une $L$-facette et $g\in G(\Lt)^c$, les facettes $\sigma(\Fc)$ et $\sigma(g\cdot \Fc)$ ont même type. Comme $\sigma(g\cdot \Fc)=\sigma(g)\cdot \sigma(\Fc)$ et que $G(\Lt)^c$ est $\Gamma$-invariant (cf. point (3) de la proposition \ref{PropDesH}), on a le résultat. Cette constatation s'étend bien entendu aux multifacettes.
\medskip

Introduisons alors la définition suivante (déjà présente dans \cite[9.2.4]{kaletha_prasad_2023}) :

\begin{déf}
    On appelle $\boldsymbol{\Gamma}$\textbf{-multifacette} une $\Lt$-multifacette fortement $\Gamma$-invariante. En particulier, une $\boldsymbol{\Gamma}$\textbf{-facette} est une $\Lt$-facette $\Gamma$-invariante.

    On définit également un $\boldsymbol{\Gamma}$\textbf{-sommet} (resp. une $\boldsymbol{\Gamma}$\textbf{-chambre}) comme étant une $\Gamma$-facette minimale (resp. maximale) parmi les $\Gamma$-facettes.
\end{déf}

Rappelons aussi que le théorème de descente non ramifiée a été montré originellement par Bruhat et Tits (dans \cite[5.]{BT2}) et généralisé par Prasad (dans \cite[Theorem 3.8.]{unram_prasad}). Ce théorème fournit un dictionnaire plus précis que le théorème de descente modérément ramifiée (notamment une forte compatibilité au niveau des facettes et des sous-groupes parahoriques).
\medskip

On se propose de développer quelques compléments à ce théorème. Avant cela, on a besoin de montrer le lemme suivant :

\begin{lem}\label{EgaliteAdherence}
    Soit $\Fc$ une $\Gamma$-multifacette de $\ImmBT(G_{\Lt})$.\\ On a l'égalité : $\overline{\Fc\cap \ImmBT(G)}=\overline{\Fc}\cap \ImmBT(G)$.
\end{lem}

\begin{proof}
    Prouvons d'abord le cas où $\Fc$ est une facette.\\
    Observons déjà que $\Fc\cap \ImmBT(G)\subset \overline{\Fc}\cap \ImmBT(G)$. Par conséquent, $\overline{\Fc\cap \ImmBT(G)}\subset \overline{\Fc}\cap \ImmBT(G)$.

    \noindent Montrons l'inclusion réciproque. Prenons $x\in \Fc\cap \ImmBT(G)$ et $y\in \overline{\Fc}\cap \ImmBT(G)$. Comme $\Fc$ est convexe, la géodésique $[x,y]\subset \overline{\Fc}$ est tel que la géodésique à moitié ouverte $[x,y[$ soit incluse dans $\Fc$ (cf. \cite[II.\S 2.6. Proposition 16.]{BourbakiEVT}).  

    Par ailleurs, puisque $\ImmBT(G)$ est convexe et que $x$ et $y$ sont dans $\ImmBT(G)$, la géodésique $[x,y]$ est en fait incluse dans $\ImmBT(G)$. Par conséquent, $[x,y[$ est incluse dans $\Fc\cap \ImmBT(G)$. Ceci implique que $y$ est dans $\overline{\Fc\cap \ImmBT(G)}$. D'où l'inclusion réciproque.

    Montrons maintenant le cas général. Notons $\Fc=\bigsqcup_i \Fc_i$ la décomposition en facettes de $\Fc$. On a :
    \begin{alignat*}{2}
        \overline{\Fc\cap \ImmBT(G)}=\overline{(\bigcup_i\Fc_i)\cap \ImmBT(G)}& &\\
        =\overline{\bigcup_i(\Fc_i\cap \ImmBT(G))}&=\bigcup_i\overline{\Fc_i\cap \ImmBT(G)}\underset{\underset{\text{facettes}}{\text{cas des}}}{=}\bigcup_i(\overline{\Fc_i}\cap \ImmBT(G))& &=(\bigcup_i\overline{\Fc_i})\cap \ImmBT(G)\\
        && &= \overline{\Fc}\cap \ImmBT(G).
    \end{alignat*}
    D'où le résultat.
\end{proof}

On a ainsi :

\begin{prop}\label{CorrespUnram}
    On a la correspondance $G(\Kt)$-équivariante croissante pour l'inclusion (resp. pour l'incidence) suivante : \upshape
    \begin{align*}
        \Biggl\{
        \begin{gathered}
          \Gamma\text{-multifacettes}  \\
          \text{ de } \ImmBT(G_{\Lt})
        \end{gathered}
        \Biggr\}
        & \cong
        \Biggl\{
        \begin{gathered}
        \Kt\text{-multifacettes } \\
          \text{ de } \ImmBT(G)
        \end{gathered}
        \Biggr\}\\
        \bigsqcup_i \Fc_i & \overset{\alpha}{\mapsto} \bigsqcup_i \Fc_i^{\Gamma} = \bigsqcup_i \Fc_i\cap \ImmBT(G)\\
        \bigsqcup_i\widetilde{F}_i & \overset{\beta}{\mapsfrom} \bigsqcup_i F_i
    \end{align*}
    \noindent\textit{où $F\mapsto \widetilde{F}$ associe à une $\Kt$-facette l'unique $\Lt$-facette contenant son barycentre.}
    
    \textit{On a donc en particulier que, sous cette correspondance, un $\Gamma$-sommet correspond à un $\Kt$-sommet et une $\Gamma$-chambre correspond à une $\Kt$-chambre.}
\end{prop}

\begin{proof}
    Pour alléger la preuve, on écrit seulement le cas des facettes. Il suffit de raisonner facette par facette pour avoir le cas des multifacettes.

    La remarque \cite[5.1.5.1 Remark (c)]{RousseauNewBook} énonce explicitement la bonne définition et même la surjectivité de la flèche directe au niveau des facettes. Réciproquement, pour une \linebreak $\Kt$-facette $F$, la $\Lt$-facette $\widetilde{F}$ est $\Gamma$-invariante puisqu'elle est l'unique $\Lt$-facette contenant le barycentre de $F$, lui-même fixé par $\Gamma$. D'où la bonne définition de la flèche réciproque.

    Observons alors que $\widetilde{\Fc^{\Gamma}}=\Fc$ car les deux facettes contiennent le barycentre de $\Fc^{\Gamma}$. Réciproquement, $(\widetilde{F})^{\Gamma}=F$ car les deux facettes contiennent le barycentre de $F$.

    Notons que les deux ensembles sont $G(\Kt)$-stables. La flèche directe est évidemment \linebreak $G(\Kt)$-équivariante puisque tout élément de $G(\Kt)$ est fixé par $\Gamma$ et puisque l'action de $\Gamma$ sur $\ImmBT(G_{\Lt})$ est compatible à l'action de $G(\Lt)$. La flèche réciproque l'est donc également.

    La croissance pour l'inclusion est bien sûr évidente dans les deux sens.
    
    Regardons l'incidence pour la flèche directe. D'après le lemme \ref{EgaliteAdherence} , on a \linebreak $\overline{\Fc}\cap \ImmBT(G)=\overline{\Fc\cap \ImmBT(G)}$. Par conséquent, si une facette $\Gamma$-invariante $\Fc'$ est dans $\overline{\Fc}$, alors $\Fc'\cap \ImmBT(G)\subset \overline{\Fc}\cap \ImmBT(G)=\overline{\Fc\cap \ImmBT(G)}$. Autrement dit, ${\Fc'}^{\Gamma}$ est incident à  $\Fc^{\Gamma}$. C'est ce que l'on voulait.
    
    Pour la flèche réciproque, si $\overline{F}\subset \overline{F'}$, alors le barycentre de $F$ est contenu dans $\overline{F'}\subset \overline{\widetilde{F'}}$. Donc $\widetilde{F}\subset \overline{\widetilde{F'}}$ puisque $\widetilde{F}$ l'unique $\Lt$-facette contenant le barycentre.
\end{proof}

\pagebreak

Rappelons les définitions suivantes :

\begin{déf}
    ${}$
    \begin{enumerate}
        \item On dit que $G$ est \textbf{résiduellement déployé} si $G$ et $G_{\Ktnr}$ ont même rang semi-simple relatif.

        \item On dit que $G$ est \textbf{résiduellement quasi-déployé} s'il existe une $\Ktnr$-chambre $\Galnr$-invariante dans $\ImmBT(G_{\Ktnr})$ (ou encore s'il existe une $\Galnr$-chambre qui est une $\Ktnr$-chambre).
    \end{enumerate}
\end{déf}

On a également une correspondance au niveau des types :

\begin{prop}\label{UnramTypes}
    ${}$
    \begin{enumerate}
        \item La correspondance de la proposition \ref{CorrespUnram} préserve les types. 
    
        \noindent En particulier, l'orbite par l'action d'un sous-groupe global conforme de $G(\Kt)$ d'une $\Gamma$-multifacette (resp. d'une $\Kt$-multifacette) décrit exactement les $\Gamma$-multifacettes de même $\Lt$-type (resp. les $\Kt$-multifacettes de même $\Kt$-type).

        \item Notons $\typebt_{\mathrm{max}}$, le type d'une $\Gamma$-chambre (qui est indépendant du choix de la \linebreak $\Gamma$-chambre). On a donc les bijections naturelles $\Xi$-équivariantes croissantes pour l'inclusion et l'incidence (sous un sens évident) suivante : \upshape
        \begin{align*}
            \Biggl\{
            \begin{gathered}
            \Lt\text{-multitypes fortement} \\
            \Gamma \text{-inv.} \text{ de } \ImmBT(G_{\Lt}) \text{ dans } \typebt_{\mathrm{max}}
            \end{gathered}
            \Biggr\}
            \overset{\sim}{\leftarrow}
            \Biggl\{
            \begin{gathered}
              \text{Ensembles de } \Gamma\text{-multifacettes de } \\
              \ImmBT(G_{\Lt})\text{ de même } \Lt\text{-multitype}
            \end{gathered}
            \Biggr\}
            \begin{gathered}
            \overset{\overline{\alpha}}{\rightarrow} \\
            \overset{\overline{\beta}}{\leftarrow}
            \end{gathered}
            \Biggl\{
            \begin{gathered}
            \Kt\text{-multitypes } \\
              \text{ de } \ImmBT(G)
            \end{gathered}
            \Biggr\}
        \end{align*}    
        \textit{En particulier, si $G$ est résiduellement quasi-déployé, l'ensemble de gauche est exactement celui des $\Lt$-multitypes fortement $\Gamma$-invariants de $\ImmBT(G_{\Lt})$.}
        \end{enumerate}
\end{prop}

\begin{proof}
    ${}$
    \begin{enumerate}
        \item Considérons deux $\Gamma$-multifacettes $\widetilde{\Fc}$ et $\widetilde{\Fc}'$ et prenons $H$ un sous-groupe $\Lt$-conforme de $G(\Kt)$. Notons également $\Fc:=\widetilde{\Fc}^{\Gamma}$ et $\Fc':=(\widetilde{\Fc}')^{\Gamma}$.
    
        Supposons que $\Fc$ et $\Fc'$ soient de même $\Kt$-type. Alors, $\widetilde{\Fc}$ et $\widetilde{\Fc}'$ ont même $\Lt$-types. En effet, il existe $g\in H$ tel que $g\cdot \Fc = \Fc'$. Par bijectivité et $G(\Kt)$-équivariance de la correspondance de la proposition \ref{CorrespUnram}, on a $g\cdot \widetilde{\Fc}=\widetilde{\Fc}'$. D'où le résultat puisque $H$ ne change pas les $\Lt$-types.
    
        Supposons maintenant que $\widetilde{\Fc}$ et $\widetilde{\Fc}'$ soient de même $\Lt$-type. On sait qu'il existe $g\in H$ tel que $g\cdot \Fc$ et $\Fc'$ vivent dans la même adhérence d'une $\Kt$-chambre. Notons $\widetilde{\Cc}$ la $\Gamma$-facette correspondante. Or, $(g\cdot \widetilde{\Fc})^{\Gamma}=g\cdot \Fc$. Ceci signifie que $g\cdot \widetilde{\Fc}$ et $\widetilde{\Fc}'$ sont dans l'adhérence de $\widetilde{\Cc}$ par croissance. En particulier, ils vivent dans la même adhérence d'une $\Lt$-chambre. Comme $g$ ne change pas les $\Lt$-types, cela signifie que $g\cdot \widetilde{\Fc}=\widetilde{\Fc}'$. En particulier, $g\cdot \Fc$ et $\Fc'$ sont égaux. Comme $g$ ne change pas non plus les $\Kt$-types, on en déduit que $\Fc$ et $\Fc'$ ont le même $\Kt$-type. D'où le résultat.

        \item La correspondance du point (2) donnée par $\overline{\alpha}$ et $\overline{\beta}$ s'obtient alors en factorisant les applications $\alpha$ et $\beta$ de la proposition \ref{CorrespUnram} au niveau des orbites par $H$. En effet, d'une part, l'orbite d'une $\Kt$-multifacette par $H$ est en correspondance avec les $\Kt$-multitypes. D'autre part, l'orbite d'une $\Gamma$-multifacette par $H$ décrit des $\Gamma$-multifacettes qui sont par ailleurs de même $\Lt$-type par caractère $\Lt$-conforme. Réciproquement, les $\Gamma$-multifacettes de même $\Lt$-type sont en fait toutes décrites d'après le point (1).

        \pagebreak
        
        Prouvons maintenant la première bijection du point (2). Prenons une $\Gamma$-chambre $\widetilde{\Cc}$. Comme $\widetilde{\Cc}$ est en correspondance avec l'ensemble des $\Lt$-multitypes dans $\typebt_{\mathrm{max}}$, on peut relever un $\Lt$-multitype $\typebt$ fortement $\Gamma$-invariant vivant dans $\typebt_{\mathrm{max}}$ en une $\Lt$-multifacette $\widetilde{\Fc}$ dans $\widetilde{\Cc}$. Mais comme $\widetilde{\Cc}$ est $\Gamma$-invariant, l'orbite de $\widetilde{\Fc}$ par $\Gamma$ reste dans $\widetilde{\Cc}$. Comme $\typebt$ est fortement $\Gamma$-invariant, $\Gamma$ agit de manière conforme sur $\widetilde{\Fc}$, et donc $\widetilde{\Fc}$ est fortement $\Gamma$-invariant. Ceci montre donc la surjectivité, l'injectivité étant bien sûr évidente.
    
        Si $G$ est résiduellement quasi-déployé, alors $\typebt_{\mathrm{max}}$ est le type d'une $\Lt$-chambre. Ceci donne le résultat.
    
        Enfin, notons que ces correspondances sont bien entendu $G(\Kt)$-équivariantes. Comme $G(\Kt)^c$ agit trivialement sur les $\Kt$-multitypes, il agit également trivialement sur les autres ensembles et l'action se factorise donc partout par $G(\Kt)/G(\Kt)^c\cong \Xi$. Ceci montre en particulier que l'orbite par l'action d'un sous-groupe global conforme de $G(\Kt)$ d'une $\Gamma$-multifacette décrit exactement les $\Gamma$-multifacettes de même $\Lt$-type, d'où la seconde remarque de la proposition.
    \end{enumerate}
\end{proof}

Établissons maintenant quelques corollaires à la proposition \ref{UnramTypes} :

\begin{coro}\label{TrèsConformeEstConforme}
    ${}$
    \begin{enumerate}
        \item Tout sous-groupe global de $G(\Kt)$ qui est $\Lt$-conforme est conforme. En particulier, étant donné $H$, un sous-groupe global de $G(\Kt)$, on a $H\cap G(\Lt)^c\subset H^c$ et \linebreak $H\cap G(\Lt)^b\subset H^b$.
        \item Tout sous-groupe global de $G(\Kt)$ est conforme si et seulement s'il agit de manière conforme sur les $\Gamma$-multifacettes de $\ImmBT(G_{\Lt})$.
    \end{enumerate}
\end{coro}

\begin{rmq}\label{GKbGalois}
    Attention ! Il est possible que l'inclusion $G(\Kt)^{tc} \subset G(\Kt)^c$ soit stricte, et donc qu'un sous-groupe global conforme ne soit pas très conforme. Un contre-exemple où $G$ est l'unique forme interne de $\mathrm{PGL}_2$ sur $\QQ_p$, avec $p$ premier (puisque $H^2(\QQ_p,\mu_2)=\ZZ/2\ZZ$), est donné en \cite[Example 2.6.31]{kaletha_prasad_2023}. Elle est donc adjointe, anisotrope et résiduellement quasi-déployée. Autrement dit, $G(\Kt)$ permute les sommets de l'échelonnage sur $\Ktnr$, c'est-à-dire $A_1$ ($\dynkin A2$).
\end{rmq}

Notons également le résultat suivant :

\begin{prop}\label{CompatibiliteStab}
    Soit $\Ht$ un sous-groupe global $\Gamma$-invariant de $G(\Lt)$. Prenons $\widetilde{\Fc}$ une \linebreak $\Gamma$-multifacette de $\ImmBT(G_{\Lt})$. Posons $H:=\Ht^\Gamma$ et $\Fc:=\widetilde{\Fc}^\Gamma$. Alors $$\Ht_{(\widetilde{\Fc})}\cap H = H_{(\widetilde{\Fc})} = H_{(\Fc)}.$$
\end{prop}

\begin{proof}
    Le résultat se ramène bien évidemment au cas des facettes. Soit $h\in H_{\Fc}$ Comme $\emptyset\not = \Fc\subset h\cdot \widetilde{\Fc}\cap \widetilde{\Fc}$, on a $h\cdot \widetilde{\Fc}= \widetilde{\Fc}$. Donc $h\in H_{\widetilde{\Fc}}$. Réciproquement, si $h\in H_{\widetilde{\Fc}}$, prenons $x\in \Fc$. Alors $h\cdot x \in \widetilde{\Fc}$. Mais $\sigma(h\cdot x) = \sigma(h)\cdot \sigma(x) = h\cdot x$ pour tout $\sigma \in \Gamma$. Donc $h\cdot x \in \Fc$ et $h\in H_{\Fc}$.
\end{proof}

De ceci, on peut introduire la définition suivante :

\begin{déf} \label{defStab}
     Prenons $\widetilde{H}$ un sous-groupe global $\Galnr$-invariant de $G(\Ktnr)$. Posons \linebreak $H:=\widetilde{H}^{\Galnr}$.
    \begin{enumerate}
    \item Considérons $\Fc$, une facette de $\ImmBT(G)$ et $\widetilde{\Fc}$ sa $\Galnr$-facette associée par la correspondance de la proposition \ref{CorrespUnram}.
        \begin{enumerate}
            \item On dit qu'un $\Rt$-modèle lisse et séparé de $G$ ayant comme $\Rtnr$-points le groupe $\widetilde{H}_{\widetilde{\Fc}}$ (resp. $\widetilde{H}^{\mathrm{f}}_{\widetilde{\Fc}}$) est un \textbf{schéma en groupes stabilisateur (resp. fixateur) de $\boldsymbol{\Fc}$ relativement à $\boldsymbol{H}$}. Il est aussi appelé un \textbf{modèle de Bruhat-Tits de $\boldsymbol{H_{\Fc}}$ (resp. $\boldsymbol{H^{\mathrm{f}}_{\Fc}}$}).
            \item Son groupe des $\Rt$-points est donné par $(\widetilde{H}_{\widetilde{\Fc}})^{\Galnr}=H_{\widetilde{\Fc}}=H_{\Fc}$ (resp. \linebreak $(\widetilde{H}^{\mathrm{f}}_{\widetilde{\Fc}})^{\Galnr}=H^{\mathrm{f}}_{\widetilde{\Fc}}=H^\mathrm{f}_{\Fc}$) d'après la proposition \ref{CompatibiliteStab}.
        \end{enumerate} 
    \item Supposons cette fois que $\Fc$ soit une multifacette.
        \begin{enumerate}
            \item On dit qu'un $\Rt$-modèle lisse et séparé de $G$ ayant comme $\Rtnr$-points le groupe $\widetilde{H}_{(\widetilde{\Fc})}$ est un \textbf{schéma en groupes multistabilisateur de $\boldsymbol{\Fc}$ relativement à $\boldsymbol{H}$}. Il est aussi appelé un \textbf{modèle de Bruhat-Tits de $\boldsymbol{H_{(\Fc)}}$}. 
            
            \item Son groupe des $\Rt$-points est donné par $(\widetilde{H}_{(\widetilde{\Fc})})^{\Galnr}=H_{(\widetilde{\Fc})}=H_{(\Fc)}$ d'après la proposition \ref{CompatibiliteStab}.
        \end{enumerate}  
    \end{enumerate}
    Si $\widetilde{H}=G(\Ktnr)$, {\normalfont relativement à $H$} peut être omis dans les définitions précédentes.
\end{déf}

\begin{rmq}
    Les définitions précédentes sont bien sûr compatibles aux extensions algébriques galoisiennes non ramifiées $\Ktnr/\Lt/\Kt$ sous un sens évident.
\end{rmq}

\begin{rmq}
    Un $\Rt$-schéma lisse et affine est unique si l'on fixe ses $\Rtnr$-points (cf. \cite[Corollary 2.10.11]{kaletha_prasad_2023}). Par conséquent, en reprenant les notations de la définition, il y a au plus un seul modèle de Bruhat-Tits affine étant donné le choix de $H^{\mathrm{n.r.}}$ et de $\Fc$.

    La question de l'unicité dans le cas où le modèle n'est pas affine sera discutée dans un article ultérieur.
\end{rmq}

\begin{rmq}
    Cette définition inclut en particulier les schémas en groupes définis par Bruhat et Tits dans \cite{BT2} et également les schémas en groupes définis dans \cite{kaletha_prasad_2023}. Il inclut également les modèles de Néron des tores (qui donne donc un exemple de situation où le modèle n'est pas nécessairement affine).
    La question de l'existence, sous certaines hypothèses, des modèles de Bruhat-Tits (notamment lorsqu'ils ne sont pas affines) sera abordée dans un article ultérieur.
\end{rmq}

\section{Résultats cohomologiques théoriques}

Dans toute la suite, on note. $\widetilde{\xi}$, le morphisme type associé à $G_{\Lt}$. On note également $\DynD$ (resp. $\widetilde{\DynD}$) le diagramme de Dynkin affine relatif de $G$ (resp. $G_{\Lt}$).
\medskip

Considérons également $\Xi^\mathrm{ext}$ (resp. $\Xi_{\Lt}^\mathrm{ext}$), le sous-groupe d'automorphismes de Dynkin de $\DynD$ (resp. $\widetilde{\DynD}$) induit par les automorphismes polysimpliciaux d'un appartement de $\ImmBT(G)$ (resp. $\ImmBT(G_{\Lt})$) qui induisent vectoriellement un élément du groupe de Weyl vectoriel (cf. \cite[Definition 1.3.71]{kaletha_prasad_2023}). Plus précisément, cette construction est indiquée dans \cite[Remark 1.3.76]{kaletha_prasad_2023}.

Notons que l'action de $G(\Lt)$ sur $\ImmBT(G_{\Lt})$ est compatible à l'action de $G(\Lt)$ sur l'immeuble vectoriel de $G$ sur $\Lt$, au sens où cette dernière donne l'action vectorielle sous-jacente. De plus, l'action vectorielle est, sur chaque appartement, induite par des éléments du groupe de Weyl vectoriel, et en conséquence préserve les types vectoriels (cf. \cite[2.2.16.(c) Theorem.]{RousseauNewBook}). Ceci implique que l'image du morphisme type sur $\Kt$ (resp. $\Lt$) est incluse dans $\Xi^\mathrm{ext}$ (resp. $\Xi_{\Lt}^\mathrm{ext}$).

De plus, l'action de Galois sur $\ImmBT(G_{\Lt})$ est compatible à l'action de Galois sur l'immeuble vectoriel sur $\Lt$, et comme dans le cas affine, l'immeuble vectoriel sur $\Kt$ se plonge dans l'immeuble vectoriel sur $\Lt$, de telle sorte que toute $\Kt$-facette vectorielle est l'ensemble des points fixes d'une $\Lt$-facette vectorielle $\Gamma$-invariante (cf. \cite[2.2.6.2.]{RousseauNewBook}).
\medskip

On définit également le \textit{type étendu} : à un couple $(\Fc,*)$ composé d'une $\Kt$-multifacette et d'un point de $V_G$, on associe $(\mathcal{T},*)$, le couple formé du type de $\Fc$ et de $*$ (vu donc dans $\widetilde{\DynD}\times V_G$). L'action de $G(\Kt)$ sur $(\Fc,*)$ induit une action sur $(\mathcal{T},*)$ donnée par $g\cdot (\typebt,*)=(\xi(g)\cdot \typebt,g\cdot *)$ et donc un morphisme $\xi^e$ associé, dont le noyau est par définition $G(\Kt)^b:=G(\Kt)^c\cap G(\Kt)^1$. Étant donné un sous-groupe global $H$ de $G(\Kt)$, on note \linebreak $\Xi^e_{H}:=\xi^e(H)\cong H/H^b$. 

Bien entendu, on généralise tout cela sur $\Lt$, et l'action de $\Gamma$ sur $\ImmBT^e(G_{\Lt})$ se factorise par $\widetilde{D}\times V_{G_{\Lt}}$. On note $\widetilde{\xi}^e$ le morphisme associé sur $\Lt$.
\medskip

Commençons par le théorème suivant :

\begin{thm}\label{SuiteExacteType}
    Soit $\Ht$ un sous-groupe global $\Gamma$-invariant de $G(\Lt)$. Notons $H:=\Ht^{\Gamma}$.
    \begin{enumerate}
        \item 
        \begin{enumerate}
            \item Le groupe $\Xi_{\Ht}$ est muni de l'action de $\Gamma$ par conjugaison (donnée par \linebreak $\sigma \mapsto (\omega\mapsto\sigma \circ \omega \circ \sigma^{-1})$), de telle sorte que l'on ait la suite exacte de $\Gamma$-groupes :
            \[\begin{tikzcd}
            	1 & {\Ht^c} & {\Ht} & {\Xi_{\Ht}} & 1.
            	\arrow[from=1-1, to=1-2]
            	\arrow[from=1-2, to=1-3]
            	\arrow["{\widetilde{\xi}}", from=1-3, to=1-4]
            	\arrow[from=1-4, to=1-5]
            \end{tikzcd}\]
            
            \item De même, $\Xi^e_{\Ht}$ est muni de l'action de $\Gamma$ par conjugaison, de telle sorte à ce que l'on ait la suite exacte de $\Gamma$-groupes :
            \[\begin{tikzcd}
            	1 & {\Ht^b} & {\Ht} & {\Xi^e_{\Ht}} & 1.
            	\arrow[from=1-1, to=1-2]
            	\arrow[from=1-2, to=1-3]
            	\arrow["{\widetilde{\xi}^e}", from=1-3, to=1-4]
            	\arrow[from=1-4, to=1-5]
            \end{tikzcd}\]
        \end{enumerate}
        
        \item Les suites exactes précédentes donnent lieu aux suites exactes suivantes d'ensembles pointés :
        \begin{enumerate}
            \item 
            \begin{tikzcd}
        	1 & {(\Xi_{\Ht})^{\Gamma}/\widetilde{\xi}(H)} & {H^1(\Gamma,\Ht^c)} & {H^1(\Gamma,\Ht)} & {H^1(\Gamma,\Xi_{\Ht}).}
        	\arrow[from=1-1, to=1-2]
        	\arrow[from=1-2, to=1-3]
        	\arrow[from=1-3, to=1-4]
        	\arrow[from=1-4, to=1-5]
            \end{tikzcd}
            
            \item \begin{tikzcd}
        	1 & {(\Xi^e_{\Ht})^{\Gamma}/\widetilde{\xi}^e(H)} & {H^1(\Gamma,\Ht^b)} & {H^1(\Gamma,\Ht)} & {H^1(\Gamma,\Xi^e_{\Ht}).}
        	\arrow[from=1-1, to=1-2]
        	\arrow[from=1-2, to=1-3]
        	\arrow[from=1-3, to=1-4]
        	\arrow[from=1-4, to=1-5]
            \end{tikzcd}
        \end{enumerate}
        
        \item On a les inclusions suivantes : $\widetilde{\xi}(H)\subset (\Xi_{\Ht,\typebt_{\mathrm{max}}})^{\Gamma} \subset (\Xi_{\Ht})^{\Gamma}$, où $\typebt_{\mathrm{max}}$ est le type d'une $\Gamma$-chambre.

        \item Le groupe $(\Xi_{\Ht,\typebt_{\mathrm{max}}})^{\Gamma}$ agit naturellement sur $\DynD$ et induit une flèche \linebreak $(\Xi_{\Ht,\typebt_{\mathrm{max}}})^{\Gamma}\rightarrow \Xi^\mathrm{ext}$. 

        \item Le noyau $\Ker \left ((\Xi_{\Ht,\typebt_{\mathrm{max}}})^{\Gamma}\rightarrow \Xi^\mathrm{ext} \right )$ est donné par les éléments de $(\Xi_{\Ht})^{\Gamma}$ fixant toutes les $\Gamma$-orbites de $\widetilde{\DynD}$ dans $\typebt_{\mathrm{max}}$, ou de manière équivalente, stabilisant $\typebt_{\mathrm{max}}$ et fixant une $\Gamma$-orbite se descendant un en $\Kt$-point spécial.
        
        \item Si $\widetilde{\DynD}$ admet un sommet spécial $x$ dans $\typebt_{\mathrm{max}}$ (par exemple si $G$ est résiduellement quasi-déployé), alors un élément $\omega$ du noyau s'écrit $\sigma \circ \phi=\phi\circ \sigma$ avec $\sigma \in \Gamma$ et $\phi \in \Aut(\widetilde{\DynD})$, ce dernier fixant $x$, les $\Gamma$-orbites dans $\typebt_{\mathrm{max}}$ et envoyant une $\Gamma$-orbite quelconque vers une autre. Par ailleurs, $\omega$ est le seul élément du noyau ayant une décomposition avec $\sigma$.
        
        \item Le cardinal du noyau est majorée par la taille de la $\Gamma$-orbite de $x$. En particulier, si $x$ est fixé par $\Gamma$ (par exemple s'il est hyperspécial, cf. la définition \ref{DefHyperspecial}), alors le noyau est trivial.

        \item La restriction de la flèche $(\Xi_{\Ht,\typebt_{\mathrm{max}}})^{\Gamma}\rightarrow \Xi^\mathrm{ext}$ à $\widetilde{\xi}(H)$ a comme image $\Xi$ et comme noyau $\widetilde{\xi}(H^c)=H^c/(H\cap \Ht^c)$. En particulier, $H^c=H\cap \Ht^c$ lorsque $G$ admet un sommet spécial sur $\Lt$ fixé par $\Gamma$.
    \end{enumerate}
\end{thm}

\begin{proof}
    ${}$
    \begin{enumerate}
        \item \begin{enumerate}
            \item D'après le point (3) de la proposition \ref{PropDesH}, $\Ht^c$ est un sous-groupe $\Gamma$-invariant de $\Ht$. Par conséquent, l'application $h\mapsto \sigma(h)\mapsto \widetilde{\xi}(\sigma(h))$ de $\Ht$ vers $\Xi_{\Ht}$ se factorise par $\Xi_{\Ht}$. On en déduit alors que l'action de $\Gamma$ sur $\Ht$ se factorise en une action de $\Gamma$ sur $\Xi_{\Ht}$ de telle sorte que la suite exacte de l'énoncé soit réalisée. La relation $\sigma(h)\cdot \Fc=\sigma(h\cdot \sigma^{-1}(\Fc))$ pour toute facette $\Fc$, tout $\sigma \in \Gamma$ et $h\in \Ht$, et le fait que tout élément de $\Gamma$ induit un automorphisme de Dynkin sur les types, implique la relation $\widetilde{\xi}(\sigma(h))=\sigma \circ \widetilde{\xi}(h) \circ \sigma^{-1}$.
            
            \item Ce point se fait de manière analogue au point précédent.
        \end{enumerate}

        \item \begin{enumerate}
            \item La suite exacte en cohomologie donne alors :
            \[\begin{adjustbox}{max size={1\textwidth}{1\textheight}}
            \begin{tikzcd}
            	1 & {(\Ht^c)^{\Gamma}} & H & {(\Xi_{\Ht})^{\Gamma}} & {H^1(\Gamma,\Ht^c)} & {H^1(\Gamma,\Ht)} & {H^1(\Gamma,\Xi_{\Ht}).}
            	\arrow[from=1-1, to=1-2]
            	\arrow[from=1-2, to=1-3]
            	\arrow[from=1-3, to=1-4]
            	\arrow[from=1-4, to=1-5]
            	\arrow[from=1-5, to=1-6]
            	\arrow[from=1-6, to=1-7]
            \end{tikzcd}
            \end{adjustbox}\]
            Elle implique alors la suite exacte :
            \[\begin{adjustbox}{max size={1\textwidth}{1\textheight}}
            \begin{tikzcd}
            	1 & {\widetilde{\xi}(H)} & {(\Xi_{\Ht})^{\Gamma}} & {H^1(\Gamma,\Ht^c)} & {H^1(\Gamma,\Ht)} & {H^1(\Gamma,\Xi_{\Ht}).}
            	\arrow[from=1-1, to=1-2]
            	\arrow[from=1-2, to=1-3]
            	\arrow[from=1-3, to=1-4]
            	\arrow[from=1-4, to=1-5]
            	\arrow[from=1-5, to=1-6]
            \end{tikzcd}
            \end{adjustbox}\]
            Et de même, cette dernière implique la suite exacte de l'énoncé.
            
            \item Ce point se fait de la même manière.
        \end{enumerate}
        
        \item L'inclusion $\widetilde{\xi}(H)\subset (\Xi_{\Ht,\typebt_{\mathrm{max}}})^{\Gamma}$ provient du fait que $H$ envoie une $\Gamma$-chambre sur une $\Gamma$-chambre.
        
        \item Comme $(\Xi_{\Ht,\typebt_{\mathrm{max}}})^{\Gamma}$ stabilise $\typebt_{\mathrm{max}}$ et qu'il stabilise les $\Lt$-types $\Gamma$-invariants, il induit une action sur les $\Lt$-types $\Gamma$-invariants incidents à $\typebt_{\mathrm{max}}$. Ces types sont exactement en correspondance avec les $\Kt$-types d'après la proposition \ref{UnramTypes}.

        Prenons alors $h\in \Ht$ tel que $\widetilde{\xi}(h)\in (\Xi_{\Ht,\typebt_{\mathrm{max}}})^{\Gamma}$ et $\Cc$ une $\Gamma$-chambre. Par hypothèse, $h\cdot \Cc$ est de même type que $\Cc$. Quitte à bouger $h$ par un élément de $\Ht^c$, on peut supposer qu'ils sont incidents à une même $\Lt$-chambre, et donc qu'ils sont égaux.
    
        Par hypothèse, $\widetilde{\xi}(h)=\widetilde{\xi}(\sigma(h))$. Ceci implique que $h^{-1}\sigma(h)$ fixe $\Cc$. En conséquence :
        \begin{align*}
            \Cc^\Gamma = (h\cdot \Cc)^\Gamma &= \{h\cdot x \mid \forall \sigma\in \Gamma, \sigma(h\cdot x)=h\cdot x\} \\
            &= \{h\cdot x \mid \forall \sigma\in \Gamma, h^{-1}\sigma(h)\cdot \sigma(x)=x\}=\{h\cdot x \mid \forall \sigma\in \Gamma, \sigma(x)=x\}=h\cdot \Cc^\Gamma.
        \end{align*}
        Autrement dit, $h$ stabilise la $\Kt$-chambre $\Cc^\Gamma$. Or $h$ agit par isométries sur $\ImmBT(G_{\Lt})$, et comme $\ImmBT(G)$ s'y plonge de manière isométrique, $h$ envoie $\Cc^\Gamma$ sur lui-même de manière isométrique. En particulier, l'action induite sur les $\Kt$-types préserve la structure de Dynkin. 

        Soit maintenant une $\Kt$-chambre vectorielle $\vec{C}$ et un $\Kt$-point spécial $x\in \mathcal{C}^\Gamma$ tel que $\mathcal{C}^\Gamma \subset x + \vec{C}$. Prenons ensuite une $\Lt$-facette vectorielle $\Gamma$-invariante $\vec{\mathcal{C}}$ tel que $(\vec{\mathcal{C}})^\Gamma=\vec{C}$. 

        Comme $h^{-1}\sigma(h)$ fixe $\mathcal{C}$, il fixe $\mathcal{C}^\Gamma$ et en conséquence $x + \vec{C}$, et donc fixe vectoriellement $\vec{C}$. Ceci implique qu'il stabilise $\vec{\mathcal{C}}$, et donc le fixe (puisque l'action préserve les types).

        En faisant comme pour $\mathcal{C}$, on en déduit que $h\cdot (\vec{\mathcal{C}})^\Gamma = (h\cdot \vec{\mathcal{C}})^\Gamma$, et donc $h$ envoie la $\Kt$-chambre vectorielle $(\vec{\mathcal{C}})^\Gamma$ sur la $\Kt$-chambre vectorielle $(h\cdot \vec{\mathcal{C}})^\Gamma$. Comme $h$ envoie $\vec{\mathcal{C}}$ sur $h\cdot \vec{\mathcal{C}}$ de sorte à préserver les types, il en est de même pour $(\vec{\mathcal{C}})^\Gamma$.

        Notons alors que $\mathcal{C}^\Gamma=h\cdot \mathcal{C}^\Gamma\subset h\cdot x + h\cdot (\vec{\mathcal{C}})^\Gamma$. On conclut alors que $h$ permute $\mathcal{C}^\Gamma$ en respectant les types vectoriels, et donc la permutation induite est dans $\Xi^\mathrm{ext}$, comme voulu.
        
        \item La description de $\Ker\left(\Xi_{\Ht,\typebt_{\mathrm{max}}})^{\Gamma}\rightarrow \Xi^\mathrm{ext} \right )$ provient du fait que les $\Gamma$-orbites de $\widetilde{\DynD}$ dans $\typebt_{\mathrm{max}}$ sont en correspondance avec les sommets de $\DynD$. La condition équivalente provient du fait que le seul élément de $\Xi^\mathrm{ext}$ fixant un point spécial est l'identité (cf. \cite[Remark 1.3.76]{kaletha_prasad_2023}).

        \item Comme $\omega$ est dans le noyau, il envoie $x$ sur un élément de sa $\Gamma$-orbite. Autrement dit, il existe $\sigma\in \Gamma$ tel que $\omega \cdot x = \sigma\cdot x$. Donc $\phi :=\sigma^{-1}\circ \omega$ fixe $x$. Comme $\sigma$ et $\omega$ fixent les $\Gamma$-orbites dans $\typebt_{\mathfrak{\max}}$, il en est de même pour $\phi$. Par ailleurs, comme $\sigma^{-1}\,\omega\,\sigma =\omega$, on a également $\sigma \circ \phi = \phi \circ \sigma$. Enfin, comme $\sigma$ et $\omega$ envoient une $\Gamma$-orbite vers une autre, il en est de même pour $\phi$. 

        D'après \cite[Remark 1.3.76]{kaletha_prasad_2023}, le morphisme $\omega' \mapsto \omega' \cdot x$ de $\Xi_{\Ht}$ vers les points spéciaux de $\widetilde{\DynD}$ est injectif. On en déduit également que $\omega$ est le seul élément du noyau ayant $\sigma$ dans sa décomposition puisque $\omega\cdot x = \sigma\cdot x$.

        \item On réutilise \cite[Remark 1.3.76]{kaletha_prasad_2023}. Comme un élément du noyau envoie $x$ vers un élément de sa $\Gamma$-orbite, on en déduit la majoration. Si $G$ admet un point hyperspécial, il est $\Gamma$-invariant et son orbite est réduite à lui-même. D'où le résultat.

        \item Observons qu'un élément $h\in H$ est tel que $\widetilde{\xi}(h)$ est dans le noyau $\Ker\left(\Xi_{\Ht,\typebt_{\mathrm{max}}})^{\Gamma}\rightarrow \Xi^\mathrm{ext} \right)$ si et seulement s'il fixe les $\Gamma$-orbites de $\widetilde{\mathscr{D}}$ dans $\typebt_{\mathrm{max}}$ d'après le point (5). Cela est équivalent à demander que $\widetilde{\xi}(h)$ stabilise les $K$-types d'après \ref{UnramTypes}, ou encore que $h\in H^c$. Le noyau est donc donné par $\widetilde{\xi}(H^c)$. 
        
        D'après la proposition \ref{CorrespUnram}, le groupe $H^c$ agit transitivement sur les $\Gamma$-chambres, puisqu'agit transitivement sur les $K$-chambres. En conséquence, tout élément de $\widetilde{\xi}(H)$ provient d'un $h\in H$ qui stabilise une certaine $\Gamma$-chambre $\mathcal{C}$. Comme vu dans le point (4), l'action de $h$ sur $\mathcal{C}^\Gamma$ détermine l'image de $\widetilde{\xi}(h)$ dans $\Xi^\mathrm{ext}$. Or, cette action n'est autre que l'action naturelle de $h$ sur la $K$-chambre $\mathcal{C}^\Gamma$ dans l'immeuble $\ImmBT(G)\cong \ImmBT(G_{\Ktnr})^\Gamma$. Par définition, cette action définit un élément de $\Xi$.
        
        D'après le point (7), si $G$ admet un sommet spécial sur $L$ fixé par $\Gamma$, alors le noyau $\Ker\left(\Xi_{\Ht,\typebt_{\mathrm{max}}})^{\Gamma}\rightarrow \Xi^\mathrm{ext} \right )$ est trivial. En conséquence, $\widetilde{\xi}(H^c)$, et donc $H^c/(H\cap \Ht^c)$ est trivial. Ceci implique l'égalité $H^c=H\cap \Ht^c$. Le point (1) du corollaire \ref{TrèsConformeEstConforme} donne que $H\cap \widetilde{H}^c\subset H^c$.
    \end{enumerate}
\end{proof}

\begin{rmq}
    Il existe des situations où $\typebt_{\mathrm{max}}$ ne contient aucun sommet spécial de $\widetilde{\DynD}$. En effet, posons $q_0=X_1^2+X_2^2+X_3^2+X^2_4$ et considérons la forme quadratique \linebreak $q=q_0(X_1,X_2,X_3,X_4) + t\,q_0(X_1',X_2',X_3',X_4')$ dans $\RR((t))$. Elle est anisotrope de discriminant $1$, donc $\mathrm{Spin}(q)$ est un groupe anisotrope simplement connexe absolument presque simple de type ${}^1D_4$. Son indice de Tits affine est donné par 
    {\begin{dynkinDiagram}D[1]{4}
    \draw[black] (root 2) circle (.12cm);
    \draw[black,fill=black] (root 0) circle (.05cm);
    \end{dynkinDiagram}}, d'où le contre-exemple voulu puisque le point central n'est pas spécial dans $D_4$ d'après \cite[Table 1.3.5]{kaletha_prasad_2023}.
    \medskip

    En effet, observons en premier lieu qu'il y a une inclusion naturelle de $\mathrm{SO}(q_0)\times\mathrm{SO}(q_0)$ dans $\mathrm{SO}(q)$ sur $\RR((t))$. Cette inclusion se relève au niveau des revêtements simplement connexes $\mathrm{Spin}(q_0)\times\mathrm{Spin}(q_0)\rightarrow \mathrm{Spin}(q)$ d'après \cite[Exercise 6.5.2.(iii)]{ConradSGA3}. On constate ensuite que le noyau $\mu$ du relevé est inclus dans $\Ker\left(\mathrm{Spin}(q_0)\times\mathrm{Spin}(q_0)\rightarrow \mathrm{SO}(q_0)\times\mathrm{SO}(q_0)\right)$ : en conséquence $\mu$ est un sous-groupe de type multiplicatif fini déployé central.
    
    Notons $\mathcal{P}$ l'unique schéma en groupes parahorique de $\mathrm{Spin}(q)$ sur $\RR[[t]]$. Puisque $\mathrm{Spin}(q)$ est anisotrope sur $\RR((t))$ et simplement connexe, on a $\mathcal{P}(\RR[[t]])=\mathrm{Spin}(q)(\RR((t)))$ d'après \cite[5.2.10.(i) Proposition.]{BT2}. Par ailleurs, comme $q_0$ est défini et régulier sur $\RR[[t]]$, le groupe $(\mathrm{Spin}(q_0)\times\mathrm{Spin}(q_0))/\mu$ est défini et réductif sur $\RR[[t]]$. On le note alors $\mathcal{Q}$.
    
    Montrons ensuite qu'il existe un morphisme $\mathcal{Q}\rightarrow \mathcal{P}$ de $\RR[[t]]$-schémas en groupes étendant les inclusions $(\mathrm{Spin}(q_0)\times\mathrm{Spin}(q_0))/\mu\rightarrow \mathrm{Spin}(q)$ et $\mathcal{Q}(\RR[[t]]) \subset \mathcal{P}(\RR[[t]])=\mathrm{Spin}(q)(\RR((t)))$. Il suffit pour cela de montrer que $\mathcal{Q}$ est étoffé (cf. \cite[1.7.1. Définition.]{BT2}).
    
    D'après \cite[1.7.2.]{BT2}, il suffit de montrer que $\mathcal{Q}$ vérifie (ET 1) et (ET 2). \cite[1.7.3.]{BT2} nous donne déjà que (ET 1) est satisfait. (ET 2) signifie que l'image de $\mathcal{Q}(\Rt)$ vers $\mathcal{Q}(\kappa)$ est schématiquement dense dans $\mathcal{Q}_\kappa$. Comme $\mathcal{Q}$ est lisse et $\Rt$ hensélien, cela revient à montrer que $\mathcal{Q}(\kappa)$ est schématiquement dense dans $\mathcal{Q}_\kappa$ d'après le lemme de Hensel (\cite[Lemma 8.1.3]{kaletha_prasad_2023}). Cela est vrai d'après \cite[Theorem 17.93]{Milne}.
    
    Comme $\mathcal{Q}$ est réductif sur $\RR[[t]]$ et qu'il est de rang $4$, il admet $R[[t]]$-tore maximal $T$ de rang $4$. L'application induite $T\rightarrow \mathcal{P}$ a un noyau de type multiplicatif d'après \sga{Exp. IX, Théorème 6.8.}. Or, ce dernier est trivial sur la fibre générique : il est donc trivial d'après \sga{Exp. IX, Remarque 1.4.1.b)}. En conséquence, $T\rightarrow \mathcal{P}$ et en particulier le rang de $\mathcal{P}_\CC$ est d'au moins $4$.

    Également, comme le groupe est anisotrope, son indice affine contient qu'une seule orbite distinguée. Enfin, d'après \cite[3.5.2.]{Corvallis}, l'indice de Tits du quotient réductif de $\mathcal{P}_\RR$ est obtenu en supprimant tous les sommets associés à la facette de $\mathcal{P}$ dans l'indice de Tits affine de $\mathrm{Spin}(q)$ (en outre, l'unique orbite distinguée). Notons d'ailleurs que l'indice de Tits affine de $\mathrm{Spin}(q)$ admet $5$ sommets. L'observation précédente sur le rang montre qu'un sommet est au plus supprimé, et donc que l'orbite distinguée est réduite à un point.
    \medskip
    
    Il reste donc à éliminer le cas de figure suivant :
    {\begin{dynkinDiagram}D[1]{4}
    \draw[black] (root 1) circle (.12cm);
    \draw[black,fill=black] (root 0) circle (.05cm);
    \end{dynkinDiagram}} (modulo rotation). Supposons, en raisonnant par l'absurde, que son indice soit de cette forme. Cela signifie que $\mathrm{Spin}(q)$ admet un point hyperspécial (cf. la définition \ref{DefHyperspecial}) et donc un modèle réductif $G$ sur $\RR[[t]]$ (cf. le lemme \ref{modelesReductifs}). Considérons alors une forme quadratique régulière $q'$ sur $\RR$ tel que $G_\RR=\mathrm{Spin}(q')$ ($G$ est simplement connexe). Grâce à l'inclusion $\RR\rightarrow \RR[[t]]$, ceci définit un groupe réductif $\mathrm{Spin}(q')$ sur $\RR[[t]]$.
    
    Les deux groupes $G$ et $\mathrm{Spin}(q')$ sont alors des formes de $\mathrm{Spin}_8$ qui coïncident sur $\RR$. Or, d'après \sga{Exp. XXIV, Proposition 8.1.(ii)}, puisque $\Aut(\mathrm{Spin}_8)$ est lisse, on a $H^1(\RR[[t]],\Aut(\mathrm{Spin}_8))\overset{\sim}{\rightarrow} H^1(\RR,\Aut(\mathrm{Spin}_8))$. On en déduit alors que $G$ et $\mathrm{Spin}(q')$ sont isomorphes.
    
    En particulier, on a un isomorphisme entre $\mathrm{Spin}(q')$ et $\mathrm{Spin}(q)$ sur $\RR((t))$. Montrons alors que ceci implique que $q'$ et $q$ sont équivalents à homothétie par un scalaire de $\RR((t))$ près. De ceci, on en déduit alors une absurdité car $q$ n'est pas régulière lorsque l'on réduit modulo $t$, et donc pas régulière sur $\RR[[t]]$, contrairement à $q'$.

    D'après \cite[(44.8) Theorem.]{LivreDesInvolutions}, il y a une équivalence de catégories entre les algèbres trialitaires et les groupes simplement connexes de type $D_4$ au travers de $T\mapsto\mathrm{Spin}(T)$. En l’occurrence, dans le cadre de groupes de type ${}^1 D_4$, les choses se simplifient grandement. On considère les algèbres à involutions $(\mathrm{M}_8(\RR((t))),*)$ et $(\mathrm{M}_8(\RR((t))),*')$ avec \linebreak $*:=X\mapsto M_q^{-1}\,{}^tX\, M_q$ et $*':=X\mapsto M_{q'}^{-1}\,{}^tX\, M_{q'}$, où $M_q$ et $M_{q'}$ sont respectivement les matrices des formes quadratiques $q$ et $q'$. Les groupes "$\mathrm{Spin}$" associés sont simplement $\mathrm{Spin}(q)$ et $\mathrm{Spin}(q')$. Les algèbres sont donc isomorphes. Pour conclure, on utilise ensuite \cite[(12.34) Proposition.]{LivreDesInvolutions} qui donne que $(\mathrm{M}_8(\RR((t))),*)$ et $(\mathrm{M}_8(\RR((t))),*')$ sont des algèbres à involutions isomorphes si et seulement si $q$ et $q'$ sont équivalentes à homothétie près.
\end{rmq}

\pagebreak

Montrons maintenant le théorème suivant, qui est le c{\oe}ur de cette partie :

\begin{thm}\label{NoyauH1}
    Soit $\widetilde{\Omega}$ une union disjointe de parties $\bigsqcup_{i\in I} \widetilde{\Omega_i}$ de $\ImmBT(G_{\Lt})$ où chaque $\widetilde{\Omega_i}$ est $\Gamma$-invariant.
    Prenons $\Ht\subset G(\Lt)$, un sous-groupe global $\Gamma$-invariant et posons $H:=\Ht^{\Gamma}$. Choisissons $*$, un point quelconque de la partie vectorielle $V_G$ de $\ImmBT^e(G)$.
    \begin{enumerate}
        \item On a les isomorphismes naturels (où les quotients considérés sont des quotients d'actions) : \begin{enumerate}
            \item $(\mathrm{Orb}((\widetilde{\Omega}_i)_{i\in I})_{\Ht})^{\Gamma}/H \overset{\sim}{\rightarrow} \Ker \left( H^1(\Gamma,\Ht_{(\widetilde{\Omega}_i)_{i\in I}}) \rightarrow H^1(\Gamma,\Ht) \right ).$
            \item $(\mathrm{Orb}((\widetilde{\Omega}_i)_{i\in I},\ast))_{\Ht})^{\Gamma}/H \overset{\sim}{\rightarrow} \Ker \left( H^1(\Gamma,\Ht^1_{(\widetilde{\Omega}_i)_{i\in I}}) \rightarrow H^1(\Gamma,\Ht) \right ).$
            \end{enumerate}
        L'action sur les familles est celle terme à terme.   
        \item Prenons cette fois une $\Gamma$-multifacette $\widetilde{\Fc}$ de type $\typebt$. On a les isomorphismes suivants induits en passant aux types : 
        \begin{enumerate}
            \item $(\mathrm{Orb}((\widetilde{\Fc}))_{\Ht})^{\Gamma}/H\overset{\sim}{\rightarrow}\left(\{\omega \cdot\typebt\prec \typebt_{\mathrm{max}} \mid \omega \in \Xi_{\Ht} \}^{\Gamma}\right)/\,\Xi_H.$
            \item $(\mathrm{Orb}((\widetilde{\Fc}),*)_{\Ht})^{\Gamma}/H\overset{\sim}{\rightarrow}\left(\{(\omega \cdot\typebt,\omega \cdot *) \mid \omega \cdot\typebt\prec \typebt_{\mathrm{max}}\text{ , } \omega \in \Xi^e_{\Ht} \}^{\Gamma}\right)/\,\Xi^e_{H}.$
        \end{enumerate}
    \end{enumerate}
\end{thm}

\begin{proof}
    ${}$
    \begin{enumerate}
        \item \begin{enumerate}
            \item Le résultat provient de \cite[I.\S5.4., Corollaire 1.]{CohoGalois}.
            En effet, il suffit d'appliquer le tout au $\Gamma$-morphisme $\Ht_{(\widetilde{\Omega}_i)_{i\in I}}\rightarrow \Ht$, et d'observer que \linebreak $\Ht/\Ht_{(\widetilde{\Omega}_i)_{i\in I}} \cong \mathrm{Orb}((\widetilde{\Omega}_i)_{i\in I})_{\Ht}$. en tant que $\Gamma$-ensemble muni d'une action de $\Ht$ (terme à terme).
            
            \item  Le second point se prouve de la même manière. En effet, il suffit d'observer que $\Ht^1_{(\widetilde{\Omega}_i)_{i\in I}}=\Ht_{((\widetilde{\Omega}_i)_{i\in I},*)}$.
        \end{enumerate}

        \item 
        \begin{enumerate}
            \item Notons déjà que la flèche $(\mathrm{Orb}((\widetilde{\Fc}))_{\Ht})^{\Gamma}\rightarrow \{\omega \cdot\typebt\prec \typebt_{\mathrm{max}} \mid \omega \in \Xi^{\Ht} \}^{\Gamma}$ est bien définie puisque toute $\Gamma$-multifacette est incident à une $\Gamma$-chambre, et cette incidence passe aux types. Par ailleurs, d'après la proposition \ref{UnramTypes}, l'action de $H$ sur les types fortement $\Gamma$-invariants se factorise par $\Xi^H$. Ceci prouve la bonne définition de la flèche de l'énoncé.
    
            Montrons l'injectivité. Prenons $g,g'\in H^{{\Lt}}$ tels que $g\cdot \widetilde{\Fc}$ et $g'\cdot \widetilde{\Fc}$ soient \linebreak $\Gamma$-fortement invariants. Supposons également qu'il existe $h\in H$ tel que $h\cdot (g\cdot \widetilde{\Fc})$ et $g'\cdot \widetilde{\Fc}$ aient même $\Lt$-type. D'après la proposition \ref{UnramTypes}, cela signifie qu'il existe $h_b\in H^{b}$ tel que $(h_b h g)\cdot \widetilde{\Fc}$ et $g'\cdot \widetilde{\Fc}$ soient égaux. Puisque $h_b\,h\in H$, cela signifie donc que $g\cdot \widetilde{\Fc}$ et $g'\cdot \widetilde{\Fc}$ sont dans la même orbite par $H$. D'où l'injectivité.
            
            Montrons maintenant la surjectivité. Prenons $\widetilde{\Cc}$, une $\Gamma$-chambre tel que $\widetilde{\Fc}$ soit incident à celle-ci. Soit $h\in \Ht$ tel que $\xi(h)\cdot \typebt$ soit un type fortement $\Gamma$-invariant incident à $\typebt_\mathrm{max}$. Il se relève en la $\Lt$-multifacette $h\cdot \widetilde{\Fc}$, que l'on peut supposer incident à $\widetilde{\Cc}$, quitte à bouger $h$ par un élément de $\Ht^b$.
            
            Soit donc $\sigma \in \Gamma$. On a alors $\sigma(h\cdot \widetilde{\Fc})$ incident à $\sigma(\widetilde{\Cc})=\widetilde{\Cc}$. Or, comme $\xi(h)\cdot \typebt$ est $\Gamma$-invariant, $\sigma(h\cdot \widetilde{\Fc})$ est également de ce type. On a donc $\sigma(h\cdot \widetilde{\Fc})$ et $h\cdot \widetilde{\Fc}$ de même type dans $\widetilde{\Cc}$, ils sont donc égaux en tant que multifacette. Par conséquent, $h\cdot \widetilde{\Fc}$ est $\Gamma$-fortement invariant et la surjectivité est prouvée comme voulu. 

            \pagebreak

            \item Comme précédemment, la flèche (définie en prenant le type sur le premier facteur), est bien définie pour les mêmes raisons. Ajoutons également que l'action de $H$ sur l'ensemble de gauche se factorise par $\Xi^e_{H}$ car $H^b\subset H^c$ agit trivialement sur les types fortement $\Gamma$-invariants (comme montré précédemment), et $H^b\subset H^1$ agit trivialement sur $*\in V_G$. La preuve de l'injectivité et de la surjectivité se fait, \textit{mutatis mutandis}, comme le cas précédent.
        \end{enumerate}
    \end{enumerate}
\end{proof}

\begin{rmq}
    En combinant le point (1) et le point (2) dans le cadre d'une multifacette, on en déduit que les noyaux du point (1) ne dépendent que des types, plus précisément que du diagramme de Dynkin affine relatif sur $\Lt$ muni de son action de Galois et l'ensemble de sommets $\typebt_{\mathrm{max}}$ (ce que l'on pourrait appeler un $\Lt/\Kt$-indice de Tits affine) et de l'action de $\Ht$ dessus.

    Pour réaliser des calculs dans le cas où $\Lt=\Ktnr$, on peut notamment s'aider de la classification des indices de Tits affines faite dans \cite[4. Classification.]{Corvallis} pour les corps locaux, ou encore de la récente classification des groupes résiduellement quasi-déployés lorsque $\kappa$ est parfait dans \cite[6.5.13]{RousseauNewBook}.

    On peut également déterminer la liste des indices de Tits affines dans le cas hyperspécial (cf. définition \ref{DefHyperspecial}), et donc réaliser un calcul, en utilisant la liste des indices de Tits (classiques) dans \cite{BoulderSS} sur lequel on adjoint un point hyperspécial à l'aide de la liste des diagrammes de Dynkin affines (cf. \cite[Table 1.3.4 / 1.3.5]{kaletha_prasad_2023}).
\end{rmq}

\begin{rmq}
    Les deux ensembles du point (2) (a) du théorème \ref{NoyauH1} sont donc finis puisqu'il y a un nombre fini de types distincts.
\end{rmq}

\begin{rmq}\label{CasSSSC}
    Si $G$ est semi-simple simplement connexe, quasi-déployé sur $\Ktnr$, il a déjà été prouvé dans \cite[5.2.10.(ii)]{BT2} que $(\mathrm{Orb}(\widetilde{\Fc})_{G(\Ktnr)})^{\Galnr}/G(\Kt)=1$ pour toute \linebreak\linebreak $\Galnr$-facette $\Fc$. On peut aussi le déduire de \cite[5.2.10.(i)]{BT2} (qui dit que $G(\Ktnr)$ est conforme) et du point (2) (a) du théorème \ref{NoyauH1}.
\end{rmq}

On déduit du théorème \ref{NoyauH1} les cas particuliers suivants :

\begin{coro}\label{InjComplet}
    Soit $\widetilde{\Fc}$ une $\Gamma$-multifacette de $\ImmBT(G_{\Lt})$.
    Prenons $\Ht\subset G(\Lt)$ global, \linebreak $\Gamma$-invariant, et agissant de manière conforme sur $\widetilde{\Fc}$. On a :
    $$\Ker \left(H^1(\Gamma,\Ht_{(\widetilde{\Fc})})\rightarrow H^1(\Gamma,\Ht)\right)=1.$$
\end{coro}

\begin{proof}
Notons $\typebt$, le type de $\widetilde{\Fc}$. L'hypothèse signifie que $\Ht$ agit trivialement sur $\typebt$. Par conséquent, $(\mathrm{Orb}(\typebt)_{\Ht})^{\Gamma}$ est trivial. On conclut alors grâce au théorème \ref{NoyauH1}.
\end{proof}

\begin{coro}\label{TrivialiteCasChambres}
    Prenons $\Ht\subset G(\Lt)$, un sous-groupe global $\Gamma$-invariant.
    Soit $\widetilde{\Cc}$ une \linebreak $\Gamma$-chambre de $\ImmBT(G_{\Lt})$. On a :
    $$\Ker\left( H^1(\Gamma,\Ht_{\widetilde{\Cc}})\rightarrow H^1(\Gamma,\Ht) \right) = 1.$$
\end{coro}

\begin{proof}
    Notons $\typebt_\mathrm{max}$, le type de $\widetilde{\Cc}$. D'après le théorème \ref{NoyauH1}, le noyau de l'énoncé est égal au sous-ensemble de $(\mathrm{Orb}(\typebt_\mathrm{max})_{\Ht})^{\Gamma}$ composé des types incidents à $\typebt_{\mathrm{max}}$. Cet ensemble est donc bien évidemment réduit à $\{\typebt_\mathrm{max}\}$.
\end{proof}

Attardons-nous maintenant sur quelques résultats exprimant dans quelle mesure l'extension $\Lt/\Kt$ peut être changée pour les calculs de cohomologie.

\begin{lem}\label{ActionMemeRang}
    Supposons que $G$ ait même rang relatif sur $\Lt$ que sur $\Kt$. Prenons $S$ un $\Kt$-tore déployé maximal. Notons $Z:=Z_G(S)$ et $N:=N_G(S)$. Prenons également un sous-groupe global $\Gamma$-invariant $\Ht$ de $G(\Lt)$ contenant $Z(\Lt)^1$ (ou de manière équivalente, $G(\Lt)^b$) et notons $H:=\Ht^\Gamma$. On a les assertions suivantes :

    \begin{enumerate}
        \item La flèche naturelle $(H\cap N(\Kt))/Z(\Kt)^1\rightarrow (\Ht\cap N(\Lt))/Z(\Lt)^1$ est un isomorphisme.
        
        \item L'immeuble $\ImmBT^e(G)$ s'identifie canoniquement à un sous-ensemble de $\ImmBT^e(G_{\Lt})$ qui envoie un appartement sur $\Kt$ sur un appartement $\Gamma$-invariant sur $\Lt$.
        
        \item Considérons $\Ac$ l'appartement étendu associé à $S$ dans $\ImmBT^e(G)\subset\ImmBT^e(G_{\Lt})$. Le point (1) signifie également que $H_{\Ac}$ et $\Ht_{\Ac}$ ont la même image dans $\Aut_{\mathrm{aff}}(\Ac)$.
        
        \item On a les isomorphismes naturels $\Xi_H\cong \Xi_{\Ht}$ et $\Xi^e_{H}\cong \Xi^e_{\Ht}$.
    \end{enumerate}
\end{lem}

\begin{proof}
    Notons que $\Ht$ contient $Z(\Lt)^1$ si et seulement s'il contient $G(\Lt)^b$ puisque ce dernier vaut $G(\Lt)^+Z(\Lt)^1$ d'après le lemme \ref{décompSSGroupes}.
    \begin{enumerate}
        \item On a le diagramme commutatif à lignes exactes suivant :
        \[\begin{tikzcd}
        	1 & {(H\cap N(\Kt))/(H\cap Z(\Kt))} & {(H\cap N(\Kt))/Z(\Kt)^1} & {(H\cap Z(\Kt))/Z(\Kt)^1} & 1 \\
        	1 & {(\Ht\cap N(\Lt))/(H\cap Z(\Kt))} & {(\Ht\cap N(\Lt))/Z(\Lt)^1} & {(\Ht\cap Z(\Lt))/Z(\Lt)^1} & 1.
        	\arrow[from=1-1, to=1-2]
        	\arrow[from=1-2, to=1-3]
        	\arrow[hook, from=1-2, to=2-2]
        	\arrow[from=1-3, to=1-4]
        	\arrow[hook, from=1-3, to=2-3]
        	\arrow[from=1-4, to=1-5]
        	\arrow[hook, from=1-4, to=2-4]
        	\arrow[from=2-1, to=2-2]
        	\arrow[from=2-2, to=2-3]
        	\arrow[from=2-3, to=2-4]
        	\arrow[from=2-4, to=2-5]
        \end{tikzcd}\]
        
        Notons déjà que $N(\Kt)/Z(\Kt)\rightarrow N(\Lt)/Z(\Lt)$ est un isomorphisme puisque $N/Z$ est un groupe fini constant (cf. \cite[25.16]{Milne}). Comme $H Z(\Kt) = G(\Kt)$ et \linebreak $\Ht Z(\Lt)=G(\Kt)$ (puisque $G(\Kt)^+Z(\Kt)=G(\Kt)$ et $G(\Lt)^+Z(\Lt)=G(\Lt)$ d'après \cite[6.11.(i) Proposition.]{HomAbstraits}), $H\cap N(\Kt)$ (resp. $\Ht\cap N(\Lt)$) se surjecte dans $N(\Kt)/Z(\Kt)$ (resp. $N(\Lt)/Z(\Lt)$). Ceci donne que la flèche verticale de gauche est un isomorphisme.
    
        Prenons une base $\mathcal{B}$ de caractères de $Z$. Elle est donc de cardinal $r$, où $r$ est le rang relatif de $G$ sur $\Kt$. Comme il a même rang relatif sur $\Lt$, il s'agit également d'une base de caractères de $Z_{\Lt}$. La flèche $z\mapsto (v(\chi(z)))_{\chi\in \mathcal{B}}$ définit un isomorphisme de $Z(\Kt)/Z(\Kt)^1$ vers $\ZZ^r$. Comme $\Lt/\Kt$ est non ramifié, elle s'étend en un isomorphisme $Z(\Lt)/Z(\Lt)^1\cong \ZZ^r$. D'où un isomorphisme naturel $Z(\Kt)/Z(\Kt)^1\cong Z(\Lt)/Z(\Lt)^1$. Étant donné $h\in \Ht\cap Z(\Lt)$, il existe donc $z\in Z(\Kt)$ et $z^1\in Z(\Lt)^1$ tel que $h=zz^1$. Comme $Z(\Lt)^1\subset \Ht$, il s'avère que $z\in Z(\Kt)\cap \Ht=Z(\Kt)\cap H$. Ceci prouve donc que la flèche verticale de droite est un isomorphisme.
    
        En conclusion, la flèche verticale centrale est un isomorphisme puisque c'est le cas des flèches de gauche et de droite. D'où l'isomorphisme souhaité. 

        \item Il s'agit d'une conséquence immédiate de \cite[Proposition 2.3.1.]{TheseRousseau} et du théorème de descente non ramifiée.

        \item Rappelons que l'on a une flèche d'action $N(\Lt)\rightarrow \Aut_{\mathrm{aff}}(\Ac)$ dont le noyau est $Z(\Lt)^1$ (cf. fin de la section \ref{SectionExistenceImmeuble}). D'après le point précédent, cette flèche est compatible à celle associée à $N(\Kt)$. Il suffit alors d'observer que $H_\Ac=H\cap N(\Kt)$ et $\Ht_{\Ac}=\Ht\cap N(\Lt)$ pour conclure grâce au point (1).

        \item Observons tout d'abord que $H^b=\Ht^b\cap H$ et que $H^c=\Ht^c\cap H$. En effet, c'est une conséquence de la proposition \ref{UnramTypes} puisque $\Gamma$ agit trivialement sur $\ImmBT^e(G_{\Lt})$ et donc sur les types, et l'égalité des rangs relatifs signifie que toute $\Lt$-chambre est une $\Kt$-chambre. En d'autres termes, être conforme sur $\Lt$ est équivalent à être conforme sur $\Kt$.
        
        Montrons alors que $\Ht=H \, \Ht^b$ pour conclure. Cela impliquera bien entendu $\Ht=H \, \Ht^c$. Comme $Z(\Lt)^1\subset \Ht$, on a en fait $\Ht^b=G(\Lt)^b$. Le point (1) dit que $(H\cap N(\Kt)) \, Z(\Lt)^1=(\Ht\cap N(\Lt))=\Ht_{\Ac}$. En multipliant par $G(\Lt)^+$ et en utilisant le lemme \ref{décompSSGroupes}, on obtient $(H\cap N(\Kt)) \, G(\Lt)^b = \Ht_{\Ac} \, G(\Lt)^+=\Ht$. D'où le résultat.
    \end{enumerate}
\end{proof}

\begin{prop}
    Considérons $\Lt'/\Kt$, une extension galoisienne non ramifiée contenant $\Lt$ de groupe de Galois $\Gamma'$. Supposons que $G$ ait le même rang relatif sur $\Lt'$ que sur $\Lt$. Soit $\widetilde{\Fc}'$, une $\Gamma'$-multifacette. Elle induit une $\Gamma$-multifacette que l'on note $\widetilde{\Fc}$.
    
    Prenons également un sous-groupe global $\Gamma$-invariant $\Ht'$ de $G(\Lt')$ contenant $G(\Lt')^b$, et notons $\Ht:=(\Ht')^{\Gal(\Lt'/\Lt)}$ et $H:=(\Ht')^{\Gamma'}$. On a les égalités :
    \begin{enumerate}
         \item \begin{enumerate}
            \item $\Ker \left( H^1(\Gamma,\Ht^c) \rightarrow H^1(\Gamma,\Ht) \right )=\Ker \left( H^1(\Gamma',(\Ht')^c) \rightarrow H^1(\Gamma',\Ht') \right)$.
            \item $\Ker \left( H^1(\Gamma,\Ht^b) \rightarrow H^1(\Gamma,\Ht) \right )=\Ker \left( H^1(\Gamma',(\Ht')^b) \rightarrow H^1(\Gamma',\Ht') \right)$.
        \end{enumerate}

        \item \begin{enumerate}
            \item $\Ker \left( H^1(\Gamma,\Ht_{\widetilde{\Fc}}) \rightarrow H^1(\Gamma,\Ht) \right )=\Ker \left( H^1(\Gamma',(\Ht')_{\widetilde{\Fc}'}) \rightarrow H^1(\Gamma',\Ht') \right)$.
            \item $\Ker \left( H^1(\Gamma,\Ht^1_{\widetilde{\Fc}}) \rightarrow H^1(\Gamma,\Ht) \right)=\Ker \left( H^1(\Gamma',(\Ht')^1_{\widetilde{\Fc}'}) \rightarrow H^1(\Gamma',\Ht') \right)$.
        \end{enumerate}
    \end{enumerate}
\end{prop}

\begin{proof}
    Observons que $\Lt'/\Lt$ est galoisien. D'après le lemme \ref{ActionMemeRang}, puisque $G$ a même rang relatif sur $\Lt'$ que sur $\Lt$, l'inclusion $\ImmBT^e(G_{\Lt})\subset \ImmBT^e(G_{\Lt'})$ envoie un appartement sur un appartement invariant par $\Gal(\Lt'/\Lt)$. Cette inclusion est donc telle que $\Gamma$ et $\Gamma'$ agissent de la même manière sur les $\Lt$-types (qui donc sont en correspondance avec les $\Lt'$-types). Par ailleurs, notons qu'une $\Gamma'$-chambre est naturellement une $\Gamma$-chambre, de sorte que le type $\typebt_{\mathrm{max}}$ est le même sur $L$ et sur $L'$. Notons alors que facette $\widetilde{\Fc}$ n'est autre que l'image de $\widetilde{\Fc}'$ sous cette inclusion.
    \begin{enumerate}
        \item \begin{enumerate}
            \item D'après le point (2) du théorème \ref{SuiteExacteType} (et son caractère fonctoriel), il suffit de montrer que la flèche $(\Xi_{\Ht})^{\Gamma}/\widetilde{\xi}(H)\rightarrow (\Xi_{\Ht'})^{\Gamma'}/\widetilde{\xi'}(H)$ est un isomorphisme. C'est une conséquence immédiate de l'isomorphisme $\Xi_{\Ht}\overset{\sim}{\rightarrow}\Xi_{\Ht'}$ d'après le point (4) du lemme \ref{ActionMemeRang}, ce dernier isomorphisme identifiant également $\widetilde{\xi}(H)$ avec $\widetilde{\xi'}(H)$.

            \item Ce cas se traite de la même manière.
        \end{enumerate}
        
        \item \begin{enumerate}
            \item On a le diagramme commutatif suivant d'après le point (2) (a) du théorème \ref{NoyauH1} :
            \[\begin{tikzcd}
            	{\Ker \left( H^1(\Gamma,\Ht_{\widetilde{\Fc}}) \rightarrow H^1(\Gamma,\Ht) \right )} & {\Ker \left( H^1(\Gamma',\Ht'_{\widetilde{\Fc}'}) \rightarrow H^1(\Gamma',\Ht') \right )} \\
            	{(\mathrm{Orb}(\widetilde{\Fc})_{\Ht})^{\Gamma}/H} & {(\mathrm{Orb}(\widetilde{\Fc'})_{\Ht'})^{\Gamma'}/H} \\
            	{\left(\{\omega \cdot\typebt\prec \typebt_{\mathrm{max}} \mid \omega \in \Xi_{\Ht} \}^{\Gamma}\right)/\Xi_{H}} & {\left(\{\omega \cdot\typebt\prec \typebt_{\mathrm{max}} \mid \omega \in \Xi_{\Ht'} \}^{\Gamma'}\right)/\Xi_{H}}
            	\arrow[from=1-1, to=1-2]
            	\arrow["\cong", from=2-1, to=1-1]
            	\arrow[from=2-1, to=2-2]
            	\arrow["\cong"', from=2-1, to=3-1]
            	\arrow["\cong"', from=2-2, to=1-2]
            	\arrow["\cong", from=2-2, to=3-2]
            	\arrow[from=3-1, to=3-2]
            \end{tikzcd}\]
            et la dernière flèche horizontale est un isomorphisme d'après le point (4) du lemme \ref{ActionMemeRang} puisque ce dernier donne $\Xi_{\Ht}\overset{\sim}{\rightarrow}\Xi_{\Ht'}$.
            
            \item Ce cas se traite de manière analogue.
        \end{enumerate}
    \end{enumerate}
\end{proof}

\begin{rmq} \label{MemeRangNoyauTriv}
    On peut en particulier appliquer cette proposition lorsque $\Lt'=\Ktnr$ et avec une extension galoisienne finie $\Lt/\Kt$ telle que $G$ a même rang sur $\Lt$ que sur $\Ktnr$. En conséquence, les noyaux sont triviaux si $G$ est déployé sur $\Kt$, ou encore si $G$ est semi-simple et résiduellement déployé. 

    Les noyaux sont également triviaux dans le cas où $G$ est un groupe absolument presque simple quasi-déployé sur $K$ et déployée par une extension totalement ramifiée, car dans ce cas, $G$ est résiduellement déployé. En effet, un groupe quasi-déployé admet une $K$-extension galoisienne déployante minimale $K'$ (qui n'est autre que l'extension galoisienne de groupe de Galois le noyau de la $*$-action). Cette extension est totalement ramifiée par hypothèse. En conséquence, $G$ est résiduellement déployé car, dans le cas contraire, il existerait une extension non ramifiée entre $K$ et $K'$.
\end{rmq} 

Notons que, d'après \cite[2.9. Calculs galoisiens.]{GilleSemisimpleSchémas}, étant donné un $\Rt$-schéma en groupes $\Gc$, tout $\Gamma$-cocycle dans $Z^1(\Gamma, \Gc(\Rt_{\Lt}))$  (où $\Rt_{\Lt}$ est l'anneau d'entiers de $\Lt$) définit un $\Gc$-torseur sur $\Rt$, et donc un élément de $H^1(\Rt,\Gc)$ (en fait, tout $\Gc$-torseur sur $\Rt$ trivialisé sur $\Rt_L$ provient d'un unique tel cocycle d'après \cite[Lemme 2.2.1.]{GilleSemisimpleSchémas}). De même, un cocycle de $Z^1(\Gamma,G(\Lt))$ définit un élément de $H^1(\Kt,G)$. Par ailleurs, deux torseurs sont isomorphes si et seulement si les cocycles associés sont cohomologues.

On définit donc le tordu de $\Gc$ par un cocycle $z\in Z^1(\Gamma, \Gc(\Rt_{\Lt}))$, noté ${}^z\Gc$, comme étant le tordu au travers du torseur qu'il induit (cf. \cite[2.1.]{GilleSemisimpleSchémas}). On définit de même le tordu de $G$ par un cocycle de $Z^1(\Gamma,G(\Lt))$. Deux cocycles cohomologues induisent bien entendu des tordus isomorphes.

Notons enfin que tordre $\Gc$ par un cocycle $z\in Z^1(\Gamma, \Gc(\Rt_{\Lt}))$ est compatible à la torsion de $\Gc(R_L)$ par le même cocycle dans le sens suivant : $({}^z \Gc)(R_L)$ est égal à ${}^z(\Gc(R_L))$ en tant que $\Gamma$-module (il en est bien sûr de même pour $G$).
\medskip

Terminons cette section avec quelques propriétés sur le comportement par torsion par un cocycle de l'immeuble, des facettes, et des stabilisateurs :
\medskip

\begin{prop}\label{TorsionImmeuble}
    Prenons un cocycle $z \in Z^1(\Gamma,G(L))$.
    \begin{enumerate}
        \item Les tordus ${}^zG(L)$ et ${}^z\ImmBT(G_{L})$ sont tels que d'une part, ${}^zG(L)$ soit $G(L)$ muni de l'action donnée par $\sigma \star g:=z(\sigma)\sigma(g)z(\sigma)^{-1}$, et d'autre part, tel que ${}^z\ImmBT(G_{L})$ soit $\ImmBT(G_{L})$ muni de l'action $\sigma \star x = z(\sigma)\sigma(x)$. Par ailleurs, ces deux actions sont compatibles, autrement dit : $\sigma\star(g\cdot x)=(\sigma\star g)\cdot(\sigma\star x)$. De plus, $({}^z\ImmBT(G_L))^\Gamma$ s'identifie à $\ImmBT({}^zG)$. 

        \item Soit $z'$ un cocycle cohomologue à $z$ via un élément $g_0\in G(L)$ (de telle sorte à ce que $z'=\sigma\mapsto g_0^{-1} z(\sigma) \sigma(g_0)$). Alors on a les isomorphismes suivants :
        \begin{align*}
        \begin{gathered}
        {}^{z'}G(L) \overset{\sim}{\rightarrow} {}^{z}G(L) \\
        g\mapsto g_0 g g_0^{-1}
        \end{gathered}
        & 
        \text{  et  }
        \begin{gathered}
        {}^{z'}\ImmBT(G_{L}) \overset{\sim}{\rightarrow} {}^{z}\ImmBT(G_{L}) \\
        x\mapsto g_0 \cdot x
        \end{gathered}
        \end{align*}
        où le premier isomorphisme est un isomorphisme de $\Gamma$-groupes, et où le second isomorphisme est un isomorphisme de $\Gamma \ltimes G(L)$-ensembles. Par ailleurs, ${}^zG$ et ${}^{z'}G$ sont isomorphes.
    \end{enumerate}
\end{prop}

\pagebreak

\begin{proof}
    ${}$
    \begin{enumerate}
        \item Il s'agit d'une conséquence immédiate de \cite[5.3. Torsion]{CohoGalois} et notamment de \cite[Proposition 34.]{CohoGalois} pour la compatibilité. En effet, $\ImmBT(G_{\Lt})$ peut être vu comme un $\Gamma$-ensemble muni d'une action compatible du $\Gamma$-groupe $G(\Ktnr)$.
        
        Pour ce qui est des points fixes, cela provient du théorème de descente modérément ramifiée (\cite[Proposition 5.1.1.]{TheseRousseau}), notant que ${}^z G$ est réductif sur $\Kt$ puisque l'est sur $\Lt$, et du fait que $({}^z G)(L) = {}^z(G(L))$ en tant que $\Gamma$-groupes.

        \item La vérification est immédiate. Pour ce qui est de l'isomorphisme entre ${}^zG$ et ${}^{z'}G$, cela est une conséquence du fait que $z$ et $z'$ proviennent de torseurs isomorphes puisqu'ils sont cohomologues.
    \end{enumerate}
\end{proof}

\begin{prop}
    Prenons $\widetilde{\Fc}$ une $\Gamma$-multifacette de $\ImmBT(G_{\Lt})$ et $z \in Z^1(\Gamma,G(L)_{(\widetilde{\Fc})})$. Écrivons $\widetilde{\Fc}=\bigsqcup_{i\in I}\widetilde{\Fc_i}$, sa décomposition en $\Gamma$-facettes et notons $\Fc=\bigsqcup_{i\in I} \Fc_i$ la $\Kt$-multifacette associée. Le cocycle $z$ définit pour tout $i\in I$ une classe dans $Z^1(\Gamma,G(L)_{\widetilde{\Fc_i}})$ et dans $Z^1(\Gamma,G(L))$ que l'on note aussi $z$. Alors :

    \begin{enumerate}
        \item La multifacette tordue ${}^z\widetilde{\Fc}$ est également compatible à ${}^z\ImmBT(G_{L})$, dans le sens suivant : $\widetilde{\Fc}$ est $\Gamma$-fortement invariante pour l'action $\star$ introduite dans le point (1) la proposition \ref{TorsionImmeuble} et le $\Gamma \ltimes G(L)_{(\widetilde{F})}$-ensemble induit est exactement ${}^z\widetilde{\Fc}$. L'ensemble des points fixes est noté ${}^z \Fc$. Ce dernier est une multifacette de $\ImmBT({}^z G)$ et admet la décomposition en facettes ${}^z \Fc=\bigsqcup_{i\in I} {}^z \Fc_i$.

        \item Soit $z'\in Z^1(\Gamma,G(L)_{(\widetilde{\Fc})})$ un cocycle cohomologue à $z$ dans $Z^1(\Gamma,G(L))$, donc via un élément $g_0\in G(L)$ (de telle sorte à ce que $z'=\sigma\mapsto g_0^{-1} z(\sigma) \sigma(g_0)$). Alors l'isomorphisme du point (2) de la proposition \ref{TorsionImmeuble} envoie ${}^{z'}\widetilde{\Fc}$ sur la facette $g_0 \widetilde{\Fc}$ qui est $\Gamma$-invariante dans ${}^{z}\ImmBT(G_{L})$. En conséquence, ${}^z\Fc$ est envoyé sur $(g_0 \widetilde{\Fc})^\Gamma$.

        En particulier, si $g_0\in G(L)_{(\widetilde{F})}$ (et donc $z$ et $z'$ sont cohomologues dans $Z^1(\Gamma,G(L)_{(\widetilde{\Fc})})$, alors ${}^{z'}\widetilde{\Fc}$ est envoyé sur ${}^{z}\widetilde{\Fc}$ (et ${}^z\Fc$ sur ${}^{z'}\Fc$).
    \end{enumerate}
\end{prop}

\begin{proof}
    ${}$
    \begin{enumerate}
        \item Comme précédemment, on utilise \cite[5.3. Torsion]{CohoGalois}. Puisque $\widetilde{\Fc}$ est fortement \linebreak $\Gamma$-invariante et que $G(L)_{(\widetilde{\Fc})}$ opère sur $\widetilde{\Fc}$ de manière compatible à $\Gamma$, on peut considérer le tordu ${}^z\widetilde{\Fc}$. On obtient de même les ${}^z\widetilde{\Fc}_i$ pour tout $i\in I$. Comme, pour tout $i\in I$, on a $\widetilde{\Fc}_i\subset \widetilde{\Fc}\subset\ImmBT(G_L)$ en tant que $\Gamma \ltimes G(L)_{(\widetilde{F})}$-ensembles, alors de même, ${}^z\widetilde{\Fc}_i\subset {}^z\widetilde{\Fc}\subset{}^z\ImmBT(G_L)$ en tant que $\Gamma \ltimes G(L)_{(\widetilde{F})}$-ensembles. D'où la compatibilité et la forte $\Gamma$-invariance de ${}^z\widetilde{\Fc}$. La proposition \ref{CorrespUnram} nous dit alors que ${}^z \Fc$ est une multifacette de $\ImmBT({}^z G)$ de décomposition $\bigsqcup_{i\in I}{}^z \Fc_i$. 

        \item La vérification est immédiate.
    \end{enumerate}
\end{proof}

\begin{prop}
    Prenons $\Ht$, un sous-groupe global $\Gamma$-invariant de $G(\Lt)$ et posons \linebreak $H:=\Ht^{\Gamma}$. Prenons également $z\in Z^1(\Gamma,\widetilde{H})$. On a :
    
    \begin{enumerate}
        \item Le sous-groupe ${}^z H:=({}^z \Ht)^{\Gamma}$ est global dans $({}^z G)(\Kt)$.
    \end{enumerate}
    
    Supposons de plus que $z\in Z^1(\Gamma,\widetilde{H}_{(\widetilde{\Fc})})$ (qui définit bien entendu un cocycle dans $Z^1(\Gamma,G(L)_{(\widetilde{\Fc})})$ et dans $Z^1(\Gamma,\widetilde{H})$ que l'on note aussi $z$). On a :
    
    \begin{enumerate}
        \setcounter{enumi}{1}
        \item Le sous-groupe ${}^z H_{(\Fc)}:=({}^z \Ht_{(\widetilde{\Fc})})^{\Gamma}$ est le multistabilisateur de ${}^z \Fc$ dans $\ImmBT({}^z G)$ relativement à ${}^z H$. Autrement dit, ${}^z H_{(\Fc)}=({}^z H)_{({}^z \Fc)}$. 

        \item Supposons que $\Lt=\Ktnr$. Si de plus $H_{(\Fc)}$ admet un modèle de Bruhat-Tits $\mathcal{H}_{(\Fc)}$, alors $({}^z H)_{({}^z \Fc)}$ également et un est donné par ${}^z \mathcal{H}_{(\Fc)}$.
    \end{enumerate}
\end{prop}

\begin{proof}
    Observons que les groupes tordus ${}^z \Ht_{(\widetilde{\Fc})}$, ${}^z \Ht_{\Fc_i}$ pour tout $i\in I$, ${}^z \Ht$ et ${}^z G(\Lt)$ sont munis de $\Gamma$-actions compatibles. Notons aussi que ${}^z \Ht$ est global puisque son groupe sous-jacent est $\Ht$.
    
    \begin{enumerate}
        \item Comme ${}^z \Ht$ est global, ${}^z H$ est aussi global d'après le point (2) de la proposition \ref{PropDesH}.

        \item Pour ce qui est du multistabilisateur, on observe d'après la proposition \ref{CompatibiliteStab} et la compatibilité des $\Gamma$-groupes :
        $$({}^z H)_{({}^z \Fc)}={}^z H \cap {}^z \Ht_{(\widetilde{\Fc})} = ({}^z \Ht_{(\widetilde{\Fc})})^{\Gamma} = {}^z H_{(\Fc)}.$$

        \item Supposons maintenant que $H_{(\Fc)}$ admette un modèle de Bruhat-Tits $\mathcal{H}_{(\Fc)}$. Observons alors que le groupe ${}^z \mathcal{H}_{(\Fc)}$ a comme $\Rtnr$-points (muni de sa $\Gamma$-action) ${}^z \Ht_{(\widetilde{\Fc})}$ et donc comme $\Rt$-points ${}^z H_{(\Fc)}=({}^z H)_{({}^z \Fc)}$. C'est donc bien un modèle de Bruhat-Tits de $({}^z H)_{({}^z \Fc)}$.
    \end{enumerate}
\end{proof}

\section{Cas des points hyperspéciaux}

Intéressons-nous au cas des points hyperspéciaux. Rappelons la définition :

\begin{déf}\label{DefHyperspecial}
    On dit que $x\in \ImmBT(G)$ est un \textbf{point hyperspécial} de $G$ si $G$ est déployé sur $\Ktnr$ et si $x$ est un sommet spécial de $\ImmBT(G_{\Ktnr})$ (via l'identification $\ImmBT(G)\cong \ImmBT(G_{\Ktnr})^{\Galnr}$).
\end{déf}

La notion de point hyperspécial dépend donc de $\ImmBT(G)$, mais également de $G$. Notons aussi qu'un point hyperspécial est un sommet de $\ImmBT(G)$ d'après la proposition \ref{CorrespUnram}. Rappelons maintenant quelques résultats reliant les points hyperspéciaux et les modèles réductifs :

\begin{lem}\label{modelesReductifs}
    Soit $G$ un groupe réductif sur $\Kt$ et $x$ un point hyperspécial de $G$ (le groupe $G$ est donc déployé sur $\Ktnr$).
    On a les énoncés suivants :
    \begin{enumerate}
        \item Le modèle de Bruhat-Tits affine de $G(\Kt)^1_x$ est réductif (à fibres connexes) sur $\Rt$.
        \item Réciproquement, tout modèle réductif de $G$ s'obtient de cette manière.
        \item En particulier, si $G$ est un $\Rt$-groupe réductif, alors $D(G)$ (resp. $G$) est $\Kt$-anisotrope si et seulement si $G(\Rt)=G(\Kt)^1$ (resp. $G(\Rt)=G(\Kt)$).
    \end{enumerate}
\end{lem}

\begin{proof}
    ${}$
    \begin{enumerate}
        \item Par définition, le point $x$ est spécial dans $\ImmBT(G_{\Ktnr})$. D'après \cite[4.6.22.]{BT2} et \cite[4.6.28.(ii)]{BT2}, le modèle affine associé à $G(\Ktnr)^1_x$ est réductif (à fibres connexes). Comme le modèle affine de $G(\Kt)^1_x$ est simplement le descendu à $\Rt$ (qui existe en utilisant le procédé \cite[5.1.30.]{BT2} puisque $x$ est $\Galnr$-invariant), on a le résultat.
        
        \item Réciproquement, \cite[4.6.31.]{BT2} nous dit qu'un modèle réductif $\Gc$ de $G$ est isomorphe sur $\Rtnr$ au schéma associé à $G(\Ktnr)^1_x$ pour un certain point spécial $x\in \ImmBT(G_{\Ktnr})$. Comme $\Gc$ est défini sur $\Rt$, $G(\Ktnr)^1_x$ est $\Galnr$-invariant et donc \linebreak $G(\Ktnr)^1_x=G(\Ktnr)^1_{\sigma(x)}$ pour tout $\sigma\in \Galnr$. D'après le point (3) de la proposition \ref{PropDesHFacettes}, $x$ est donc aussi $\Galnr$-invariant. Il provient donc d'un point sur $\ImmBT(G)$ qui est donc hyperspécial et alors $\Gc(\Rt)=G(\Kt)^1_x$.
        
        \item Enfin, si $D(G)$ est anisotrope, $G(\Rt)$ est le stabilisateur sous l'action de $G(\Kt)^1$ de l'unique point (hyperspécial) de $\ImmBT(G)$. C'est donc exactement $G(\Kt)^1$. Si de plus $G$ est anisotrope, $G(\Kt)^1=G(\Kt)$ et on a le résultat.

        Réciproquement, si $G(\Rt)=G(\Kt)$, alors $G(\Kt)$ est borné, et donc ne peut pas contenir l'image d'un $\Kt$-cocaractère (qui est non borné). Donc $G$ est $\Kt$-anisotrope.
        Si cette fois $G(\Rt)=G(\Kt)^1$, alors $D(G)(\Kt)\subset G(\Kt)^1$ est borné et on raisonne comme précédemment.
    \end{enumerate}
\end{proof}

On en déduit donc :

\begin{prop}\label{critereAnisotropie}
    Soit $G$, un groupe réductif sur $\Rt$. Les propriétés suivantes sont équivalentes :
    \begin{enumerate}
        \item $D(G)$ (resp. $G$) est anisotrope sur $\kappa$.
        \item $D(G)$ (resp. $G$) est anisotrope sur $\Kt$.
        \item $G(\Rt)=G(\Kt)^1$ (resp. $G(\Rt)=G(\Kt)$).
    \end{enumerate}
\end{prop}

\begin{proof}
    L'équivalence entre $(2)$ et $(3)$ est une conséquence du lemme \ref{modelesReductifs}. Montrons maintenant que $(1)$ et $(2)$ sont équivalents.

    Rappelons que $G$ est isogène à $D(G)\times R(G)$, de telle sorte à ce qu'il y a une correspondance entre les cocaractères non centraux de $G$ et les cocaractères de $D(G)$, et les cocaractères centraux et les cocaractères de $R(G)$.

    Rappelons d'après la décomposition de \sga{Exp. XXVI, Corollaire 3.5.} que le \linebreak $\Rt$-schéma des sous-groupes paraboliques propres de $G$, noté $\underline{\mathrm{Par}}(G)^+$ est lisse et projectif. D'après le critère valuatif de propreté, on a $\underline{\mathrm{Par}}(G)^+(\Rt)=\underline{\mathrm{Par}}(G)^+(\Kt)$. Or, on a une flèche naturelle $\underline{\mathrm{Par}}(G)^+(\Rt)\rightarrow \underline{\mathrm{Par}}(G)^+(\kappa)$. Par conséquent, si $\underline{\mathrm{Par}}(G)^+(\Kt)$ est non vide, alors $\underline{\mathrm{Par}}(G)^+(\kappa)$ aussi. Donc si $G$ a un cocaractère non central sur $\Kt$, alors il en a un sur $\kappa$.
    
    Or, la lissité permet également d'utiliser le lemme de Hensel, de telle sorte que la flèche $\underline{\mathrm{Par}}(G)^+(\Rt)\rightarrow \underline{\mathrm{Par}}(G)^+(\kappa)$ est surjective. On a donc finalement :
    $$\underline{\mathrm{Par}}(G)^+(\Kt)=\underline{\mathrm{Par}}(G)^+(\Rt)\twoheadrightarrow \underline{\mathrm{Par}}(G)^+(\kappa)$$

    Par conséquent, si $G$ n'admet aucun parabolique propre sur $\Kt$, alors il n'en admet pas non plus sur $\kappa$. Autrement dit, si $G$ n'a pas de cocaractère non central sur $\Kt$, alors il n'en admet pas non plus sur $\kappa$.

    Maintenant, occupons-nous des caractères centraux. On peut revenir au cas d'un tore $T$. C'est une conséquence immédiate du fait que :
    $$\Hom_{{\Ktnr}}(\GG_{m,{\Ktnr}},T_{\Ktnr})\overset{\sim}{\leftarrow} \Hom_{{\Rtnr}}(\GG_{m,{\Rtnr}},T_{\Rtnr}) \overset{\sim}{\rightarrow} \Hom_{{\kappa^s}}(\GG_{m,{\kappa^s}},T_{\kappa^s})$$ 
    en tant que $\Gal(\kappa^s/\kappa)$-groupes, puisque $T$ est déployé sur $\Rtnr$.
\end{proof}

Ceci étant, il existe une preuve relativement simple du cas des tores qui semble être connue de certains spécialistes, mais dont on ne trouve pas de référence dans la littérature. \linebreak C'est l'objet du lemme suivant :

\begin{lem}\label{RedTores}
    Soit un $\Kt$-tore déployé sur $\Ktnr$. Il admet donc un modèle torique sur $\Rt$ que l'on note $T$. Notons $\widehat{T}^\circ=\Hom_{{\Rt^{\mathrm{n.r.}}}}(\mathbb{G}_{m,{\Rt^{\mathrm{n.r.}}}},T_{{\Rt^{\mathrm{n.r.}}}})$, c'est un $\Galnr$-groupe. 
    \begin{enumerate}
        \item On a un isomorphisme canonique de $\Galnr$-modules :
    $$ T(\Ktnr)^1\times \widehat{T}^\circ=T(\Rtnr)\times \widehat{T}^\circ \cong  T(\Ktnr).$$

        \item Pour tout $i\geq 1$, on a :
    $$\Ker\left( H^i(\Galnr,T(\Ktnr)^1)\rightarrow H^i(\Galnr,T(\Ktnr)) \right) = 0.$$
    \end{enumerate}
\end{lem}

\begin{proof}
    ${}$
    \begin{enumerate}
        \item On a la suite exacte naturelle de $\Galnr$-groupes :
    \[\begin{tikzcd}
    0 \arrow{r} & ({\Rt^{\mathrm{n.r.}}})^\times \arrow{r} & (\Kt^{\mathrm{n.r.}})^\times \arrow{r} & \ZZ
    \arrow{r} & 0.
    \end{tikzcd}\]
    Elle est scindée par $1\mapsto \pi$ (où $\pi$ est une uniformisante de $\Kt$, et donc aussi de $\Ktnr$). Cette section est $\Galnr$-invariante. En tensorisant la suite exacte précédente par $\widehat{T}^\circ$, on obtient la suite exacte de $\Galnr$-groupes :
    \[\begin{tikzcd}
    0 \arrow{r} & T({\Rt^{\mathrm{n.r.}}})\arrow{r} & T({\Kt^{\mathrm{n.r.}}}) \arrow{r} & \widehat{T}^\circ
    \arrow{r} & 0.
    \end{tikzcd}\]
    car on a des isomorphismes canoniques $\widehat{T}^\circ \otimes_\ZZ ({\Rt^{\mathrm{n.r.}}})^\times \cong T({\Rt^{\mathrm{n.r.}}})$ et \linebreak $\widehat{T}^\circ \otimes_\ZZ ({\Kt^{\mathrm{n.r.}}})^\times \cong T({\Kt^{\mathrm{n.r.}}})$ donnés par $\theta \otimes x\mapsto \theta(x)$. Elle est aussi scindée par $\theta\mapsto \theta\otimes \pi \cong \theta \mapsto \theta(\pi)$, section qui est aussi $\Galnr$-invariante. D'où l'isomorphisme de $\Galnr$-modules.
    
    Notons d'ailleurs que, par définition de la suite exacte, on a $T(\Rtnr)=T(\Ktnr)^1$.

        \item La dernière assertion provient alors du fait que cet isomorphisme induit l'isomorphisme canonique : 
    $$H^i(\Galnr,T(\Rtnr))\times H^i(\Galnr,\widehat{T}^\circ)\cong H^i(\Galnr,T(\Ktnr)).$$
    
    et donc l'injectivité voulue.
    \end{enumerate}
\end{proof}

Notons que Bruhat et Tits ont déjà traité le cas des points hyperspéciaux lorsque $G$ est semi-simple dans \cite[5.2.14. Proposition.]{BT2}. On peut en fait ajuster leur preuve pour inclure le cas réductif :

\begin{prop}\label{CasReductif}
    Soit $x$ un point hyperspécial de $G$. On a :
    $$\Ker\left( H^1(\Galnr,G(\Ktnr)^1_{x})\rightarrow H^1(\Galnr,G(\Ktnr)) \right) = 1.$$
\end{prop}

\begin{proof}
    Rappelons que, comme $x$ est hyperspécial, il s'agit d'un sommet à la fois dans $\ImmBT(G)$ et dans $\ImmBT(G_{\Ktnr})$. Prenons $\overline{x}=(x,\lambda)$, un point de $\ImmBT^e(G)$. Grâce au théorème \ref{NoyauH1}, la question revient à montrer que $(\mathrm{Orb}(\overline{x})_{G(\Ktnr)})^{\Galnr}/G(\Kt)=1$.
    
    Prenons alors un point $\overline{y}=(y,\mu)$ de $\ImmBT^e(G)$ tel qu'il existe $g\in G(\Ktnr)$ tel que $g\cdot \overline{x} = \overline{y}$. Par définition, $x$ et $y$ (resp. $\overline{x}$ et $\overline{y}$) sont des sommets hyperspéciaux dans $\ImmBT(G_{\Ktnr})$ (resp. un point dans $\ImmBT^e(G)$ et dans $\ImmBT^e(G_{\Ktnr})$). Comme $G(K)^+$ agit transitivement sur les \linebreak $K$-chambres et est très conforme, il existe $g'\in G(\Kt)^+$ tel que $x':=g'\cdot x$ soit toujours hyperspécial et tel que $x'$ et $y$ soient dans la même adhérence d'une $\Kt$-chambre (et donc même adhérence d'une $\Ktnr$-chambre, que l'on note $\mathcal{C}$). Par ailleurs, $g'\cdot \lambda = \lambda$ puisque $G(\Kt)^+\subset G(\Kt)^1$. Quitte à remplacer $\overline{x}:=(x,\lambda)$ par $(x',\lambda)$, on peut donc supposer cela. On peut en fait supposer que $x$ et $y$ vivent dans un $\Ktnr$-appartement spécial (de telle sorte à ce que le $\Ktnr$-tore déployé maximal associé $T$ soit défini sur $\Kt$ et contienne un $\Kt$-tore déployé maximal). Comme $G$ est déployé sur $\Ktnr$, le tore $T$ est un $\Ktnr$-tore maximal, et donc est également son propre centralisateur.

    Notons $\mathcal{I}:=G(\Kt)^+_{\mathcal{C}}$. Notons $N(\Ktnr)$ le normalisateur associé à un appartement où se trouve cette chambre. La décomposition de Bruhat (proposition \ref{Bruhat})
    donne alors que $G(\Ktnr)=\mathcal{I} \: N(\Ktnr) \: \mathcal{I}$. On peut donc écrire $g=i\,n\,i'$ avec des notations évidentes. \linebreak Par conséquent, $i\,n\,i'\cdot \overline{x} = \overline{y}$. Donc $n\cdot \overline{x} = \overline{y}$ puisque $\mathcal{I}$ fixe $\overline{x}$ et $\overline{y}$ (car fixant la chambre où ils sont et $\mathcal{I}\subset G(\Ktnr)^1$).

    Par ailleurs, puisque $x$ est spécial sur $\Ktnr$, $G(\Ktnr)^b_x\cap N(\Ktnr)$ se surjecte sur le groupe de Weyl (vectoriel) de $G_{\Ktnr}$, c'est à dire $N(\Ktnr)/T(\Ktnr)$ (cf. \cite[4.6.22.]{BT2}). Il existe donc $n'\in G(\Ktnr)^b_x\cap N(\Ktnr)$ tel que $n'$ et $n$ ont même image dans le groupe de Weyl. Autrement dit, $t:=n\,n'^{-1}\in T(\Ktnr)$. Or, $n'^{-1}\cdot \overline{x}=\overline{x}$. Donc $t\cdot \overline{x} = n\cdot \overline{x}$.

    Considérons $\sigma \mapsto t^{-1} \, \sigma(t)$. Il s'agit d'un cobord dans $B^1(\Galnr,T(\Ktnr))$ et également d'un cocycle dans $Z^1(\Galnr,T(\Ktnr)^1)$. En effet :
    $$t\cdot \overline{x} = \overline{y} = \sigma(\overline{y}) = \sigma(t\cdot \overline{x})=\sigma(t) \cdot \overline{x}$$
    puisque $\overline{x}$ et $\overline{y}$ sont $\Galnr$-invariants. Par conséquent, $t^{-1}\,\sigma(t)$ fixe $\overline{x}$. Or, $T(\Ktnr)$ agit par translation sur l'appartement (étendu). Donc s'il fixe $\overline{x}$, il fixe l'appartement (étendu). Cela signifie que l'on a en fait $t^{-1}\,\sigma(t)\in T(\Ktnr)^1$. La classe de cohomologie associée vit donc dans $$\Ker\left(H^1(\Galnr,T(\Ktnr)^1)\rightarrow H^1(\Galnr,T(\Ktnr)\right).$$
    Ce noyau est en fait trivial d'après le lemme \ref{RedTores}. Par conséquent, il existe $t'\in T(\Ktnr)^1$ tel que $\sigma \mapsto t^{-1}\,\sigma(t)=\sigma \mapsto t'^{-1}\,\sigma(t')$, ou encore tel que $t\,t'^{-1}$ soit $\Galnr$-invariant, et donc vit dans $G(\Kt)$. Par conséquent, $t\,t'^{-1}\cdot \overline{x} = t\cdot \overline{x} = \overline{y}$. Donc $\overline{x}$ et $\overline{y}$ sont dans la même orbite sous $G(\Kt)$.
\end{proof}

\begin{rmq}
    On a la factorisation $H^1(\Galnr,G(\Ktnr)^1_{\widetilde{x}})\rightarrow H^1(\Galnr,G(\Ktnr)^1) \rightarrow H^1(\Galnr,G(\Ktnr))$. Le théorème précédent implique donc :
    $$\Ker\left( H^1(\Galnr,G(\Ktnr)^1_{\widetilde{x}})\rightarrow H^1(\Galnr,G(\Ktnr)^1) \right) = 1.$$
    Une démonstration alternative de ce résultat peut être aussi obtenu en reprenant la preuve précédente avec l'immeuble réduit.
\end{rmq}

\begin{rmq}
    Il n'est toutefois pas toujours vrai que $\Ker( H^1(\Galnr,G(\Ktnr)^1)\rightarrow H^1(\Galnr,G(\Ktnr)))$ soit trivial, même si $G$ admet un point hyperspécial. Considérons l'exemple suivant :
    \medskip

    Soit $D$ une algèbre à division de degré $d$ sur un corps $k$. Considérons ici $\Kt=k((t))$. Grâce à la proposition \ref{critereAnisotropie} appliqué à $\mathrm{GL}_1(D\otimes_k k[[t]])$, l'algèbre $D$ définit une algèbre à division $D\otimes_k k((t))$ sur $k((t))$. Elle est par ailleurs déployée sur $\Kt^{\mathrm{n.r.}}$. Prenons ici $G=\mathrm{GL}_1(D)_{\Kt}$ (qui admet d'ailleurs le modèle réductif $G=\mathrm{GL}_1(D)_{k[[t]]}$). Le groupe $\mathrm{GL}_1(D)_{\Ktnr}$ admet un unique caractère donné par la norme réduite. Par conséquent, $\mathrm{GL}_1(D)(\Ktnr)^1$ est donné par le noyau de la valuation de la norme réduite sur $\Ktnr$ (qui est surjectif puisque $D$ est déployée sur $\Ktnr$).

    Or, on a la décomposition :
    $$(D\otimes_k k((t)))^\times = k((t))^\times (D\otimes_k k[[t]])^\times.$$
    En effet, un élément de $D\otimes_k k((t))^\times$ s'écrit $t^i x$ avec $x$ de réduction modulo $t$ non nulle, $x_0\in k^\times$. Par conséquent, $(t^ix_0)(x_0^{-1} x)$ est la décomposition voulue. En effet, $t^ix_0\in k((t))^\times$ et $x_0^{-1} x$ est de la forme $1-ty\in D\otimes_kk[[t]]^\times$, d'inverse donné par $\sum_{k=0}^{+\infty}(ty)^{k}\in D\otimes_kk[[t]]^\times$.
    
    Par conséquent, l'image de $(D\otimes_k k((t)))^\times$ par la valuation de la norme réduite est donnée par $k((t))^\times$ puisque $(D\otimes_k k[[t]])^\times$ est borné. Comme la norme sur $k((t))$ est compatible avec la norme réduite de $D_{k((t))}$, on a que l'image est finalement $d\ZZ$ (cf. \cite[Theorem 1.4.]{TignolWadsworth}). La suite exacte de cohomologie implique alors : $$\Ker( H^1(\Galnr,G(\Ktnr)^1)\rightarrow H^1(\Galnr,G(\Ktnr)) )= \ZZ/d\ZZ \not = 1.$$
\end{rmq}

\begin{rmq}\label{ContreEXHyperspé}
    Il s'avère que $\Ker( H^1(\Galnr,G(\Ktnr)_{\widetilde{x}})\rightarrow H^1(\Galnr,G(\Ktnr)))$ n'est pas toujours trivial. Un contre-exemple est donné dans \cite[5.2.15. Remarque.]{BT2}. Explicitons cela. 
    \medskip
    
    Prenons l'extension $\widetilde{L}/\Kt=\CC((t))/\RR((t))$, et $h$ la forme hermitienne donnée par \linebreak $z_1\overline{z_1}-z_2\overline{z_2}$. On prend $G=\mathrm{U}(h)$ (aussi noté $\mathrm{U}(1,1)$). C'est une forme quasi-déployée de $\mathrm{GL}_2$ qui vérifie d'une part $D(G)=\mathrm{SU}(h)\cong \mathrm{SL}_2$, et d'autre part $Z(G)\cong \mathrm{R}^1_{\widetilde{L}/\Kt}(\GG_m)$, qui n'est pas déployé.

    D'après le théorème \ref{NoyauH1}, il suffit de trouver deux $\widetilde{L}$-points hyperspéciaux $\Gamma$-invariants dans la même orbite par $G(\widetilde{L})$ et dont les $\widetilde{L}$-types ne sont pas conjugués par $G(\Kt)$ pour que le noyau soit non nul.
    \medskip
    
    L'immeuble de $G$ est exactement celui de $\mathrm{SL}_2$. Son diagramme de Dynkin affine relatif est donné par $\dynkin A2$ dont les deux sommets sont spéciaux (cf. \cite[4.2.23.]{BT2} et \cite[(1.4.6)]{BT1}).

    Ces deux points ne sont cependant pas conjugués dans $G(\Kt)$. En fait, ce dernier agit en préservant les types. En effet, d'une part on a $G(\Kt)=G(\Kt)^1$ puisque le radical de $G$ est anisotrope. D'autre part, comme $G_{\Lt}=\mathrm{GL}_2$, le morphisme de Kottwitz de $G$ (cf. \cite[Chapter 11]{kaletha_prasad_2023}) est obtenu en restreignant celui de $\mathrm{GL}_2$. Pour les mêmes raisons que la remarque \ref{Xi1etXiDifferents}, on conclut que $G(\Kt)^0=G(\Kt)^1$ et donc que $G(\Kt)$ agit trivialement sur les types.
    \medskip

    Pour $\mathrm{GL}_2$, on a le même diagramme de Dynkin. Les deux sommets du diagramme sont d'ailleurs fixes par Galois, car dans le cas contraire, la descente non ramifiée nous dirait que le diagramme de $G$ serait composé d'un unique point. Les deux sommets du diagramme de $G$ sont donc hyperspéciaux.
    \medskip

    Il suffit maintenant de trouver deux sommets de types différents conjugués par $\mathrm{GL}_2(\widetilde{L})$. Cela a déjà été fait dans la remarque \ref{Xi1etXiDifferents}. Ceci conclut donc.
\end{rmq}

\section{Application : cas des groupes adjoints quasi-déployés}\label{PartieCasAdjointQSplit}

Terminons enfin cet article en utilisant tout ce que l'on a montré dans les parties précédentes pour calculer de manière exacte les noyaux
$$\Ker \left(H^1(\Galnr,G(\Knr)^0_{\widetilde{\mathcal{F}}})\rightarrow H^1(\Galnr,G(\Knr))\right)$$
et
$$\Ker \left(H^1(\Galnr,G(\Knr)_{\widetilde{\mathcal{F}}})\rightarrow H^1(\Galnr,G(\Knr))\right)$$

\noindent pour les $K$-groupes $G$ semi-simples adjoints et quasi-déployés sur $\Kt$, et où $\widetilde{\mathcal{F}}$ est une \linebreak $\Galnr$-facette de l'immeuble $\ImmBT(G_{\Knr})$.\\

Occupons-nous d'abord du cas des sous-groupes parahoriques. D'après \cite[5.2.12. Proposition.]{BT2}, les sous-groupes parahoriques sur $\Kt$ sont donnés par les stabilisateurs de facettes sous l'action de la composante résiduellement neutre $G(\Kt)^0$. Or, puisque $G$ est quasi-déployé et adjoint, on a le lemme suivant :

\pagebreak

\begin{lem}
    On a les égalités : $G(\Kt)^0=G(\Kt)^b=G(\Kt)^c$.
\end{lem}

\begin{proof}
    Notons que, puisque $G$ est semi-simple, l'immeuble étendu est égal à l'immeuble réduit et donc $G(\Kt)^b=G(\Kt)^c$.

    Soit $T$ un $\Kt$-tore maximal contenant un tore déployé maximal. D'après \cite[4.4.16. Proposition.]{BT2}, il s'agit d'un tore induit. Son schéma canonique (c'est-à-dire son modèle de Néron de type fini) est donc lisse et connexe. Ses $\Rt$-points sont donnés par $T(\Kt)^1$.
    
    Par ailleurs, d'après \cite[5.2.11.]{BT2}, $G(\Kt)^0$ est engendré par $G(\Kt)^+$ et les $\Rt$-points de la composante de l'identité du schéma canonique de $T$, c'est-à-dire ici $T(\Kt)^1$ d'après la discussion précédente.

    Or, d'après le lemme \ref{décompSSGroupes}, $G(\Kt)^b=G(\Kt)^+\, T(\Kt)^1$. D'où $G(\Kt)^0=G(\Kt)^b$.
\end{proof}

La question se ramène alors à se demander si la flèche composée suivante a un noyau trivial :
$$H^1(\Galnr,G(\Ktnr)^c_{\widetilde{\Fc}})\rightarrow H^1(\Galnr,G(\Ktnr)^c)\rightarrow H^1(\Galnr,G(\Ktnr))$$

pour $\widetilde{\Fc}$, une facette $\Galnr$-invariante de $\ImmBT(G_{\Ktnr})$.
\medskip

La première flèche est de noyau trivial d'après le corollaire \ref{InjComplet}. Intéressons-nous alors à la seconde flèche.
\medskip

Pour cela, on a besoin de démontrer que tout groupe réductif quasi-déployé est résiduellement quasi-déployé. Cela est déjà connu lorsque le corps résiduel $\kappa$ est parfait (cf. \cite[Proposition 9.10.5]{kaletha_prasad_2023}). Il s'avère que le résultat subsiste en général, mais notre preuve nécessite d'utiliser la théorie des groupes pseudo-réductifs.
\medskip

Rappelons qu'un sous-groupe pseudo-parabolique d'un groupe pseudo-réductif est appelé pseudo sous-groupe de Borel s'il est un pseudo-parabolique minimal sur la clôture séparable. C'est en fait équivalent à exiger que ce soit un pseudo parabolique résoluble. Un groupe pseudo-réductif possédant un pseudo sous-groupe de Borel est dit quasi-déployé. Dans ce cas, tous ces sous-groupes pseudo-paraboliques minimaux sont des pseudo sous-groupes de Borel puisqu'ils sont conjugués. Tout ceci est expliqué au début de \cite[Section C.2]{conrad_prasad_2016}).

Ces définitions s'étendent naturellement au cas des groupes lisses connexes affines puisqu'il y a une correspondance entre ses sous-groupes pseudo-paraboliques et ceux de son quotient pseudo-réductif maximal (cf. \cite[Proposition 2.2.10]{conrad_gabber_prasad_2015}).
\medskip

Par ailleurs, introduisons (en toute généralité) la définition suivante :

\begin{déf}\label{defParahorique}
    Supposons que le modèle de Bruhat-Tits affine de $G(\Kt)^1_\Fc$ existe. On le note $\Gc^1_\Fc$. On définit alors le \textbf{schéma en groupes parahorique} (resp. le \textbf{sous-groupe parahorique}) associé à une facette $\Fc$ de $\ImmBT(G)$ comme étant $\Gc^0_\Fc:=(\Gc^1_\Fc)^\circ$ (resp. \linebreak $G(\Kt)^0_\Fc:=(\Gc^1_\Fc)^\circ(\Rt)$).
\end{déf}

\begin{rmq}
    Notons les schémas en groupes affines $\mathcal{G}^1_\Fc$ sont construits par Bruhat et Tits dans le cas où $G$ est quasi-déployé sur $\Ktnr$ dans \cite[5.1.30.]{BT2}.  

    Ceci étant, cette définition coïncide avec la définition de \cite{BT2} dans le cas quasi-déployé sur $\Ktnr$. En effet, d'une part, dans le cas quasi-déployé, \cite[4.6.21. Proposition. (ii)]{BT2} combiné avec \cite[4.6.26.]{BT2} et \cite[4.6.28. Proposition.]{BT2} implique que $(\Gc^1_\Fc)^\circ$ est bien le schéma en groupes parahorique associé au sous-groupe parahorique de \cite[5.2.6. Définition.]{BT2}. En général, comme indiqué dans \cite[5.1.30.]{BT2}, les schémas se descendent sur $\Rt$ et ses $R$-points sont les sous-groupes parahoriques d'après le dernier paragraphe de \cite[5.2.8. Proposition.]{BT2}.
\end{rmq}

\pagebreak

On peut donc proposer une généralisation de \cite[Proposition 9.10.1]{kaletha_prasad_2023}, qui donne des propositions équivalentes au fait d'être résiduellement quasi-déployé :

\begin{prop}\label{CritereResidQSplit}
    Soit $G$ un groupe réductif sur $\Kt$, quasi-déployé sur $\Ktnr$.  Les propositions suivantes sont équivalentes :
    \begin{enumerate}
        \item Il existe une $\Kt$-chambre $\Cc$ tel que le $\kappa$-groupe $(\Gc^0_{\Cc})_\kappa$ est résoluble.
        \item Il existe une $\Ktnr$-chambre $\Galnr$-invariante dans $\ImmBT(G_{\Ktnr})$.
        \item Toute $\Galnr$-chambre dans $\ImmBT(G_{\Ktnr})$ est une $\Ktnr$-chambre $\Galnr$-invariante.
        \item 
        Pour toute $\Kt$-chambre $\Cc$, le $\kappa$-groupe $(\Gc^0_{\Cc})_\kappa$ est résoluble.
        \item Pour toute $\Kt$-facette, le $\kappa$-groupe $(\Gc^0_{\Fc})_\kappa$ est quasi-déployé.
    \end{enumerate}
\end{prop}

\begin{proof}
${}$
\begin{enumerate}
    \item[$(1)\implies(2)$] La $\Kt$-chambre $\Cc$ provient d'une $\Galnr$-chambre $\widetilde{\Cc}$. Elle provient d'une $\Galnr$-chambre $\widetilde{\Cc}$ qui est donc une $\Ktnr$-chambre. Cela induit une compatibilité $(\Gc^0_{\Cc})_{\Rtnr}=\Gc^0_{\widetilde{\Cc}}$, donc $(\Gc^0_{\Cc})_{\kappa^s}=(\Gc^0_{\widetilde{\Cc}})_{\kappa^s}$ et ce dernier est donc résoluble. Il ne possède donc pas de sous-groupe parabolique non trivial. D'après la correspondance paraboliques-parahoriques sur $\Ktnr$ (\cite[5.1.32.(i) Proposition.]{BT2}), on en déduit que $\widetilde{\Cc}$ est une $\Ktnr$-chambre.
    \item[$(2)\implies(3)$] D'après la proposition \ref{CorrespUnram}, les $\Galnr$-chambres sont $G(\Kt)$-conjugués puisque les \linebreak $\Kt$-chambres le sont. Par conséquent, si une $\Galnr$-chambre est une $\Ktnr$-chambre, alors toutes les $\Galnr$-chambres le sont par conjugaison.
    \item[$(3)\implies(4)$] Prenons une $\Kt$-chambre $\Cc$. Elle provient d'une $\Galnr$-chambre $\widetilde{\Cc}$ qui est donc une $\Ktnr$-chambre. Comme précédemment, on a $(\Gc^0_{\Cc})_{\kappa^s}=(\Gc^0_{\widetilde{\Cc}})_{\kappa^s}$. Ce dernier groupe ne possède pas de sous-groupe parabolique non trivial et est donc résoluble d'après \cite[Proposition 3.5.1.(4)]{conrad_gabber_prasad_2015}, puisque son quotient pseudo-réductif est pseudo-déployé.
    \item[$(4)\implies(1)$] C'est trivial.
    \item[$(4)\implies(5)$] Soit $\Fc$, une $\Kt$-facette et $\Cc$, une $\Kt$-chambre contenant $\Fc$. D'après la correspondance paraboliques-parahoriques, l'image de $(\Gc^0_{\Cc})_\kappa\rightarrow (\Gc^0_{\Fc})_\kappa$ est un sous-groupe pseudo-parabolique de $(\Gc^0_{\Fc})_\kappa$. Par hypothèse, il est résoluble. Donc $(\Gc^0_{\Fc})_\kappa$ est quasi-déployé.
    \item[$(5)\implies(4)$] Réciproquement, prenons une $\Kt$-chambre $\Cc$. Par hypothèse, $(\Gc^0_{\Cc})_\kappa$ est quasi-déployé. Il est donc résoluble car il n'admet aucun sous-groupe pseudo-parabolique non trivial.
\end{enumerate}
\end{proof}

\begin{rmq}
    L'hypothèse de quasi-déploiement sur $\Ktnr$ intervient notamment dans la preuve pour utiliser la correspondance paraboliques-parahoriques. On ignore si la correspondance reste valide en général, et donc a fortiori si l'hypothèse de quasi-déploiement sur $\Ktnr$ peut être supprimée.
\end{rmq}

\pagebreak

On en déduit alors le résultat que l'on souhaite :

\begin{prop}\label{QSplitEstResQSplit}
    Tout $\Kt$-groupe réductif quasi-déployé est résiduellement quasi-déployé.
\end{prop}

\begin{proof}
    Prenons un groupe réductif $G$ quasi-déployé. D'après la proposition \ref{CritereResidQSplit}, il suffit de montrer que, pour toute facette $\Kt$-facette $\Fc$, le $\kappa$-groupe $(\Gc^0_{\Fc})_\kappa$ est quasi-déployé.  
    
    Prenons $S$, un tore déployé maximal de $G$. Son centralisateur dans $G$ est un tore $T$. D'après \cite[4.6.4.(ii) Proposition.]{BT2} et \cite[4.6.26.]{BT2}, $\Gc^0_{\Fc}$ admet un unique sous-tore déployé fermé $\mathcal{S}$ de fibre générique $S$ et son centralisateur $\mathcal{T}$ dans $\Gc^0_{\Fc}$ est la composante de l'identité du modèle de Néron de $T$. En particulier, le centralisateur de $\mathcal{S}_\kappa$ est $\mathcal{T}_\kappa$, qui est commutatif. Prenons alors un cocaractère tel que le centralisateur de son image dans son pseudo-parabolique associé soit celui de $\mathcal{S}_\kappa$, c'est-à-dire $\mathcal{T}_\kappa$. Ce sous-groupe pseudo-parabolique est donc résoluble. Cela prouve que $(\Gc^0_{\Fc})_\kappa$ est quasi-déployé.
\end{proof}

Revenons à notre problème. Notons pour la suite $\xi^\mathrm{n.r.}$ le morphisme "type" sur $\Ktnr$, et $\Xi^\mathrm{n.r.}$ l'image de $G(\Ktnr)$ par ce morphisme. On peut alors prouver la trivialité du noyau de la seconde flèche :

\begin{prop}\label{EgaliteTypeAdjoint}
    On a l'égalité $\xi^\mathrm{n.r.}(G(\Kt))=(\Xi^\mathrm{n.r.})^{\Galnr}$ et le fait que ces deux groupes soient canoniquement isomorphes à $\Xi^\mathrm{ext}$. En conséquence : $$\Ker\left(H^1(\Galnr,G(\Ktnr)^c)\rightarrow H^1(\Galnr,G(\Ktnr))\right)=1.$$
\end{prop}

\begin{proof}
    Notons que $G$ est résiduellement quasi-déployé d'après la proposition \ref{QSplitEstResQSplit}. Par conséquent, d'après le point (4) du théorème \ref{SuiteExacteType}, il y a un morphisme canonique $(\Xi^\mathrm{n.r.})^{\Galnr}\rightarrow \Xi^\mathrm{ext}$. Sa restriction à $\xi^\mathrm{n.r.}(G(\Kt))$ a comme image $\Xi$, qui vaut dans notre cas de figure $\Xi^\mathrm{ext}$ d'après \cite[Proposition 6.6.2]{kaletha_prasad_2023}. Le morphisme canonique précédent est donc surjectif.

    Par ailleurs, puisque $G$ est quasi-déployé, il admet un sommet spécial qui le reste après passage à n'importe quelle extension séparable. En effet, le cas déployé est évident. On obtient alors le cas général par descente quasi-déployée (cf. \cite[4.2.3.-4.2.4.]{BT2}, une valuation de Chevalley sur un groupe déployé représente un point spécial, et cette valuation, donc ce point, se descend).
    
    Le point (7) du théorème \ref{SuiteExacteType} dit alors que $\Ker \left ( (\Xi^\mathrm{n.r.})^{\Galnr}\rightarrow \Xi^\mathrm{ext} \right )$ est trivial.

    On en déduit donc finalement que $\xi^\mathrm{n.r.}(G(\Kt))$ et $(\Xi^\mathrm{n.r.})^{\Galnr}$ sont isomorphes à $\Xi^\mathrm{ext}$ à travers le même morphisme, de telle sorte à ce que $\xi^\mathrm{n.r.}(G(\Kt))=(\Xi^\mathrm{n.r.})^{\Galnr}$ comme voulu. La trivialité du noyau provient alors du point (2) du théorème \ref{SuiteExacteType}.
\end{proof}

De tout ceci, on en déduit finalement le théorème :

\begin{thm}\label{ThmParahoQSplit}
    Soit $G$ un groupe semi-simple adjoint et quasi-déployé sur $\Kt$. On a :
    $$\Ker \left(H^1(\Galnr,G(\Knr)^0_{\widetilde{\mathcal{F}}})\rightarrow H^1(\Galnr,G(\Knr))\right)=1$$
    \noindent où $\widetilde{\mathcal{F}}$ est une facette $\Galnr$-invariante de l'immeuble $\ImmBT(G_{\Knr})$.
\end{thm}
\medskip

Intéressons-nous cette fois au cas des stabilisateurs de facettes. On souhaite alors déterminer le noyau de l'application :
$$H^1(\Galnr,G(\Ktnr)_{\widetilde{\Fc}})\rightarrow H^1(\Galnr,G(\Ktnr))$$
pour $\widetilde{\Fc}$, une facette $\Galnr$-invariante de $\ImmBT(G_{\Ktnr})$.
\medskip

La stratégie est de se ramener au cas absolument presque simple et de réaliser des calculs explicites.
\medskip

Notons $G_{\Kt}:=\prod_{i\in I}G_i$, la décomposition de $G_{\Kt}$ en produit de groupes $\Kt$-presque simples. On a alors la bijection équivariante et compatible à Galois : $\ImmBT(G_{\Ktnr})\cong \prod_{i\in I} \ImmBT(G_{i,\Ktnr})$, et donc une décomposition $\widetilde{\Fc}=\prod_{i\in I} \widetilde{\Fc_i}$ en facettes $\Galnr$-invariantes. Ceci donne alors :
$$\Ker \left(H^1(\Galnr,G(\Ktnr)_{\widetilde{\Fc}})\rightarrow H^1(\Galnr,G(\Ktnr))\right) = \prod_{i\in I} \Ker\left(H^1(\Galnr,G_i(\Ktnr)_{\widetilde{\Fc}_i})\rightarrow H^1(\Galnr,G_i(\Ktnr))\right).$$

Le problème se ramène alors au cas où $G$ est un $\Kt$-groupe $\Kt$-presque simple. Il s'écrit donc $G:=\mathrm{R}_{\Lt/\Kt}(G')$ où $G'$ est un $\Lt$-groupe adjoint absolument presque simple et $\Lt/\Kt$ est une extension finie séparable. Le calcul du noyau se ramène ensuite au cas absolument presque simple grâce au lemme suivant :

\begin{lem}\label{CompNoyauResWeil}
    Soit $\Lt/\Kt$, une extension finie séparable et $H'$ un groupe réductif sur $\Lt$. \linebreak Notons $\Galnr_{\Lt}:=\Gal(\Ltnr/\Lt)$ et $H:=\mathrm{R}_{\Lt/\Kt}(H')$. Prenons une facette $\Galnr$-invariante $\widetilde{\Fc}$ dans $\ImmBT(H_{\Ktnr})$. Elle induit une $\Kt$-facette dans $\ImmBT(H)\cong \ImmBT(H')$, et correspond alors à une facette $\Galnr_{\Lt}$-invariante $\widetilde{\Fc}'$ dans $\ImmBT(H'_{\Ltnr})$. On a alors les identifications :
    $$H^1(\Galnr,H(\Ktnr))=H^1(\Galnr_{\Lt},H'(\Ltnr))$$
    $$H^1(\Galnr,H(\Ktnr)_{\widetilde{\Fc}})=H^1(\Galnr_{\Lt},H'(\Ltnr)_{\widetilde{\Fc}'})$$
    et ce, de manière fonctorielle, de telle sorte que :
    $$\Ker \left(H^1(\Galnr,H(\Ktnr)_{\widetilde{\Fc}})\rightarrow H^1(\Galnr,H(\Ktnr))\right)=\Ker \left(H^1(\Galnr_{\Lt},H'(\Ltnr)_{\widetilde{\Fc}'})\rightarrow H^1(\Galnr_{\Lt},H'(\Ltnr))\right).$$
\end{lem}

\begin{proof}
    Remarquons que $H_{\Ktnr}$ est donné par $\mathrm{R}_{\Lt\otimes_{\Kt}\Ktnr/\Ktnr}(H'_{\Lt\otimes_{\Kt}\Ktnr})$. 
    
    Posons $\Lt_{\mathrm{n.r.}}:=\Ktnr\cap \Lt$, l'extension maximale non ramifiée de $\Kt$ dans $\Lt$. On a alors : $\Lt\otimes_{\Lt_{\mathrm{n.r.}}}\Ktnr \cong \Lt \Ktnr = \Ltnr$. Considérons l'identification $\Galnr_{\Lt}:=\Gal(\Ltnr/\Lt)\cong \Gal(\Ktnr/\Lt_{\mathrm{n.r.}})$. Il s'agit d'un sous-groupe ouvert de $\Galnr$. Posons $\Sigma:=\Hom_{\Kt}(\Lt_{\mathrm{n.r.}},\Ktnr)$ et observons également les isomorphismes de $\Galnr$-modules suivants :
    \begin{align*}
        \Ktnr \otimes_{\Kt} \Lt\cong (\Ktnr \otimes_{\Kt} \Lt_{\mathrm{n.r.}})\otimes_{\Lt_{\mathrm{n.r.}}}\Lt & \cong (\prod_{\sigma\in\Sigma}{}^\sigma\Ktnr) \otimes_{\Lt_{\mathrm{n.r.}}}\Lt\\
        {} & \cong \prod_{\sigma\in \Sigma}({}^\sigma\Ktnr\otimes_{\Lt_{\mathrm{n.r.}}}\Lt) \cong \prod_{\sigma\in \Sigma}{}^\sigma\Ltnr
    \end{align*}
    
    Or, comme $\Ktnr/\Lt_{\mathrm{n.r.}}$ est séparable, $\Sigma$ se relève dans $\Galnr$. Notons toujours $\Sigma$ un de ses relevés. Il s'agit alors d'un ensemble de représentants dans $\Galnr$ pour $\Galnr/\Galnr_{\Lt}$. On peut donc utiliser le lemme de Shapiro (en cohomologie des groupes) :
    $$H^1(\Galnr,H(\Ktnr))=H^1(\Galnr,H'(\Ktnr \otimes_{\Kt} \Lt))=H^1(\Galnr,\prod_{\sigma\in \Sigma}{}^\sigma H'(\Ltnr))=H^1(\Galnr_{\Lt},H'(\Ltnr)).$$
    
    Occupons-nous désormais de $H^1(\Galnr,H(\Ktnr)_{\widetilde{\Fc}})$. Notons que :
    $$\mathrm{R}_{\Lt\otimes_{\Kt}\Ktnr/\Ktnr}(H'_{\Lt\otimes_{\Kt}\Ktnr})=\prod_{\sigma\in \Sigma}{}^\sigma\mathrm{R}_{\Ltnr/\Ktnr}(H'_{\Ltnr}).$$ La compatibilité des immeubles aux restrictions de Weil séparables (cf. preuve de \cite[Proposition 5.1.5.]{TheseRousseau}) et au produit donne les bijections équivariantes et compatibles à Galois :
    $$\ImmBT(H_{\Ktnr})\cong \prod_{\sigma\in \Sigma} \ImmBT({}^\sigma\mathrm{R}_{\Ltnr/\Ktnr}(H'_{\Ltnr}))\cong \prod_{\sigma\in \Sigma} {}^\sigma\ImmBT(H'_{\Ltnr}).$$

    \pagebreak
    
    Notons que $\widetilde{\Fc}'$ est l'image de $\widetilde{\Fc}$ dans $\ImmBT(H'_{\Ltnr})$ (c'est à dire en regardant le facteur tel que $\sigma=\id$). Dans ce cas, l'image de $\widetilde{\Fc}$ sous la correspondance ci-dessus est $({}^\sigma\widetilde{\Fc}')_{\sigma \in \Sigma}$. Cette identification induit alors l'identification : $H(\Ktnr)_{\widetilde{\Fc}}\cong \prod_{\sigma\in \Sigma}{}^\sigma H'(\Ltnr)_{{}^\sigma\widetilde{\Fc}'}=\prod_{\sigma\in \Sigma}{}^\sigma( H'(\Ltnr)_{\widetilde{\Fc}'})$. On peut donc une fois de plus appliquer le lemme de Shapiro :
    $$H^1(\Galnr,H(\Ktnr)_{\widetilde{\Fc}})=H^1(\Galnr,\prod_{\sigma\in \Sigma}{}^\sigma( H'(\Ltnr)_{\widetilde{\Fc}'}))=H^1(\Galnr_{\Lt},H'(\Ltnr)_{\widetilde{\Fc}'}).$$
    La fonctorialité de l'isomorphisme de Shapiro permet d'en déduire l'égalité des noyaux voulue.
\end{proof}

Continuons notre investigation. D'après la remarque \ref{MemeRangNoyauTriv}, les groupes semi-simples résiduellement déployés sont tels que le noyau
$$H^1(\Galnr,G(\Ktnr)_{\widetilde{\Fc}})\rightarrow H^1(\Galnr,G(\Ktnr))$$
est trivial pour toute $\Galnr$-facette $\widetilde{\Fc}$. Comme on s'est ramené au cas absolument presque simple, et que $G$ est quasi-déployé, cela élimine donc les groupes déployés, et les groupes de la forme ${}^2X_y$ déployés par une extension (quadratique) ramifiée (comme expliqué dans la remarque \ref{MemeRangNoyauTriv}). Par ailleurs, on remarque que les groupes de type ${}^6D_4$ et ${}^3D_4$ ont même rang, on peut donc éliminer la situation d'un groupe de type ${}^6D_4$ devenant de type  ${}^3D_4$ sur $\Knr$. Enfin, puisqu'un groupe de type ${}^6D_4$ est déployé par une extension de groupe de Galois $S_3$, il n'existe aucune extension galoisienne telle qu'il devienne de type ${}^2D_4$.
\medskip

Il ne reste donc que les groupes de type ${}^2E_6$, ${}^2A_n$ (pour $n\geq 1$), ${}^2D_n$ (pour $n\geq 4$), ${}^3D_4$ et ${}^6D_4$ déployés sur $\Ktnr$ à traiter.
\medskip

Ensuite, comme $G$ est quasi-déployé, il est résiduellement quasi-déployé d'après la proposition \ref{QSplitEstResQSplit}, et donc $\typebt_\mathrm{max}$, le type d'une $\Galnr$-chambre, est exactement le type d'une \linebreak $\Ktnr$-chambre. Le point (2) (a) du théorème \ref{NoyauH1} se simplifie alors en :
$$\Ker \left(H^1(\Galnr,G(\Ktnr)_{\widetilde{\Fc}})\rightarrow H^1(\Galnr,G(\Ktnr))\right)\cong (\mathrm{Orb}(\typebt)_{\Xi^{\mathrm{n.r.}}})^{\Galnr}/\Xi$$

où $\typebt$ est le $\Ktnr$-type de $\widetilde{\Fc}$.
\medskip

Par ailleurs, puisque $G$ est quasi-déployé et adjoint, d'après la proposition \ref{EgaliteTypeAdjoint}, on a $\Xi^\mathrm{n.r.}=\Xi_{\Knr}^\mathrm{ext}$ et $\Xi=(\Xi^\mathrm{n.r.})^{\Galnr}=\Xi^\mathrm{ext}$. On utilise donc de manière interchangeable les notations $\Xi$ et $\Xi^\mathrm{ext}$ (resp. $\Xi^\mathrm{n.r.}$ et $\Xi_{\Knr}^\mathrm{ext}$) dans la suite.
\medskip

Récoltons quelques données au sujet des cas restants. La liste en \cite[4.2.23.]{BT2} permet alors de déterminer les échelonnages sur $\Kt$, et même plus précisément l'action de Galois sur l'échelonnage sur $\Ktnr$, et la liste \cite[Remark 1.3.76]{kaletha_prasad_2023} donne les diagrammes de Dynkin affines associés et les groupes $\Xi$ et $\Xi^\mathrm{n.r.}$. On en déduit alors la table \ref{table:diagramqsplit}. 

Notons que l'on a rajouté des numérotations sur certains diagrammes pour faciliter les raisonnements dans la suite.
\medskip

\tikzset{/Dynkin diagram/fold style/.style={white}}

\begin{sidewaystable}
\centering
\caption{Échelonnages, actions de $\Xi^\mathrm{n.r.}$, actions de Galois, et actions de $\Xi$ pour des groupes non ramifiés.}
\label{table:diagramqsplit}

\begin{tabular}{ c c c c c c }

Type \medskip & $\Xi^\mathrm{n.r.}$ &  Générateurs de \color{blue}{$\Xi^\mathrm{n.r.}$}  & Générateur de \color{red}{$\Galnr$} & $\Xi$ & Générateur de \color{green}{$\Xi$} \\ 

\begin{tabular}{ c }
    ${}^2A_{2n}$\\
    ($n\geq 1$)
\end{tabular} \medskip & $\ZZ/2n\ZZ$ &
\begin{dynkinDiagram}[affine mark=*,fold, fold radius=0.7cm,edge length=0.7cm]A[1]{**.****.**}
\draw[blue,-latex] (root 0) to [in=-165,out=-40,relative] (root 1);
\draw[blue,-latex] (root 1) to [in=-145,out=-60,relative] (root 2);
\draw[blue,dotted,-latex] (root 2) to [in=-145,out=-60,relative] (root 3);
\draw[blue,-latex] (root 3) to [in=-165,out=-40,relative] (root 4);
\draw[blue,-latex] (root 4) to [in=-165,out=-40,relative] (root 5);
\draw[blue,-latex] (root 5) to [in=-145,out=-60,relative] (root 6);
\draw[blue,dotted,-latex] (root 6) to [in=-145,out=-60,relative] (root 7);
\draw[blue,-latex] (root 7) to [in=-165,out=-40,relative] (root 8);
\draw[blue,-latex] (root 8) to [in=-165,out=-40,relative] (root 0);
\end{dynkinDiagram}
&
\begin{dynkinDiagram}[affine mark=*,fold, fold radius=0.7cm,edge length=0.7cm]A[1]{**.****.**}
\draw[red,latex-latex] (root 8) to (root 1);
\draw[red,latex-latex] (root 7) to (root 2);
\draw[red,latex-latex] (root 6) to (root 3);
\draw[red,latex-latex] (root 5) to (root 4);
\end{dynkinDiagram}
& 0 & $\emptyset$
\\

\begin{tabular}{ c }
    ${}^2A_{2n-1}$\\
    ($n\geq 1$)
\end{tabular} \medskip & $\ZZ/(2n+1)\ZZ$ &
\begin{dynkinDiagram}[affine mark=*,fold, fold radius=0.7cm,edge length=0.7cm]A[1]{**.*****.**}
\draw[blue,-latex] (root 0) to [in=-165,out=-40,relative] (root 1);
\draw[blue,-latex] (root 1) to [in=-145,out=-60,relative] (root 2);
\draw[blue,dotted,-latex] (root 2) to [in=-145,out=-60,relative] (root 3);
\draw[blue,-latex] (root 3) to [in=-165,out=-40,relative] (root 4);
\draw[blue,-latex] (root 4) to [in=-165,out=-40,relative] (root 5);
\draw[blue,-latex] (root 5) to [in=-145,out=-60,relative] (root 6);
\draw[blue,-latex] (root 6) to [in=-145,out=-60,relative] (root 7);
\draw[blue,dotted,-latex] (root 7) to [in=-165,out=-40,relative] (root 8);
\draw[blue,-latex] (root 8) to [in=-165,out=-40,relative] (root 9);
\draw[blue,-latex] (root 9) to [in=-165,out=-40,relative] (root 0);
\end{dynkinDiagram}
&
\begin{dynkinDiagram}[affine mark=*,fold, fold radius=0.7cm,edge length=0.7cm]A[1]{**.*****.**}
\draw[red,latex-latex] (root 9) to (root 1);
\draw[red,latex-latex] (root 8) to (root 2);
\draw[red,latex-latex] (root 7) to (root 3);
\draw[red,latex-latex] (root 6) to (root 4);
\end{dynkinDiagram}
& $\ZZ/2\ZZ$ & 
\begin{dynkinDiagram}[affine mark=*,fold, fold radius=0.7cm,edge length=0.7cm]A[1]{**.*****.**}
\draw[green,latex-latex] (root 0) to (root 5);
\draw[green,latex-latex] (root 1) to (root 6);
\draw[green,latex-latex] (root 2) to (root 7);
\draw[green,latex-latex] (root 3) to (root 8);
\draw[green,latex-latex] (root 4) to (root 9);
\end{dynkinDiagram}
\\

\begin{tabular}{ c }
    ${}^2D_{2n}$\\
    ($n\geq 2$)
\end{tabular} \medskip & $(\ZZ/2\ZZ)^2$ &
\begin{tabular}{ c }
    $\tau :$\begin{dynkinDiagram}[affine mark=*,labels={1,2,,,,,,3,4},label directions={,,,,,,},edge length=0.7cm]D[1]{***.*.****}
    \draw[blue,latex-latex] (root 0) to (root 7);
    \draw[blue,latex-latex] (root 1) to (root 8);
    \end{dynkinDiagram}\\
    $\tau' :$\begin{dynkinDiagram}[affine mark=*,labels={1,2,,,,,,3,4},edge length=0.7cm]D[1]{***.*.****}
    \draw[blue,latex-latex] (root 0) to (root 1);
    \draw[blue,latex-latex] (root 7) to (root 8);
    \end{dynkinDiagram}
\end{tabular}
&

$\sigma :$\begin{dynkinDiagram}[affine mark=*,labels={1,2,,,,,,3,4},edge length=0.7cm]D[1]{***.*.****}
\draw[red,latex-latex] (root 7) to (root 8);
\end{dynkinDiagram}
& $\ZZ/2\ZZ$ & 
\begin{dynkinDiagram}[affine mark=*,labels={1,2,,,,,,3,4},edge length=0.7cm]D[1]{***.*.****}
\draw[green,latex-latex] (root 0) to (root 1);
\draw[green,latex-latex] (root 7) to (root 8);
\end{dynkinDiagram}
\\

\begin{tabular}{ c }
    ${}^2D_{2n-1}$\\
    ($n\geq 3$)
\end{tabular} \medskip & $\ZZ/4\ZZ$ &
$\varphi :$\begin{dynkinDiagram}[affine mark=*,labels={1,2,,,,,3,4},edge length=0.7cm]D[1]{}
\draw[blue,-latex] (root 0) to (root 6);
\draw[blue,-latex] (root 6) to (root 1);
\draw[blue,-latex] (root 1) to (root 7);
\draw[blue,-latex] (root 7) to (root 0);
\end{dynkinDiagram}
&
$\sigma :$\begin{dynkinDiagram}[affine mark=*,labels={1,2,,,,,3,4},edge length=0.7cm]D[1]{}
\draw[red,latex-latex] (root 6) to (root 7);
\end{dynkinDiagram}
& $\ZZ/2\ZZ$ &
$$\begin{dynkinDiagram}[affine mark=*,labels={1,2,,,,,,3,4},edge length=0.7cm]D[1]{***.*.****}
\draw[green,latex-latex] (root 0) to (root 1);
\draw[green,latex-latex] (root 7) to (root 8);
\end{dynkinDiagram}$$
\\

${}^3D_4$ et ${}^6D_4$ \medskip & $(\ZZ/2\ZZ)^2$ &
\begin{dynkinDiagram}[affine mark=*,labels={1,2,0,3,4},edge length=0.7cm]D[1]{4}
\draw[blue,latex-latex] (root 0) to (root 3);
\draw[blue,latex-latex] (root 1) to (root 4);
\end{dynkinDiagram}
\begin{dynkinDiagram}[affine mark=*,labels={1,2,0,3,4},edge length=0.7cm]D[1]{4}
\draw[blue,latex-latex] (root 0) to (root 1);
\draw[blue,latex-latex] (root 3) to (root 4);
\end{dynkinDiagram}
&
\begin{dynkinDiagram}[affine mark=*,labels={1,0,2,3,4},ply=3,edge length=0.7cm]D[1]{4}
\draw[red,-latex] (root 2) to [in=-145,out=-35,relative] (root 3);
\draw[red,-latex] (root 3) to [in=-145,out=-35,relative] (root 4);
\draw[red,-latex] (root 4) to [in=-145,out=-35,relative] (root 2);
\end{dynkinDiagram}
\text{ (et }
\begin{dynkinDiagram}[affine mark=*,labels={1,0,2,3,4},ply=3,edge length=0.7cm]D[1]{4}
\draw[red,latex-latex] (root 4) to [in=-145,out=-35,relative] (root 2);
\end{dynkinDiagram}
\text{ si ${}^6D_4$)}
& $0$ & $\emptyset$
\\  

${}^2 E_6$ \medskip & $\ZZ/3\ZZ$ &
\begin{dynkinDiagram}[affine mark=*,edge length=0.525cm]E[1]{6}
\draw[blue,-latex] (root 1) to [in=-110,out=-70,relative] (root 6);
\draw[blue,-latex] (root 6) to [in=-145,out=-35,relative] (root 0);
\draw[blue,-latex] (root 0) to [in=-145,out=-35,relative] (root 1);
\draw[blue,-latex] (root 3) to [in=-120,out=-70,relative] (root 5);
\draw[blue,-latex] (root 5) to [in=-145,out=-45,relative] (root 2);
\draw[blue,-latex] (root 2) to [in=-145,out=-45,relative] (root 3);
\end{dynkinDiagram}
&
\begin{dynkinDiagram}[affine mark=*,edge length=0.525cm]E[1]{6}
\draw[red,latex-latex] (root 1) to [in=-110,out=-70,relative] (root 6);
\draw[red,latex-latex] (root 3) to [in=-120,out=-70,relative] (root 5);
\end{dynkinDiagram}
& $0$ & $\emptyset$
\end{tabular}
\end{sidewaystable}

\pagebreak

Commençons par le type ${}^2A_n$ (pour $n\geq 1$) :

\begin{prop}
    Considérons le diagramme de Dynkin affine de type $A_n$ (pour $n\geq 1$) muni de l'action de Galois $\Galnr$ donnée par la symétrie axiale de la table \ref{table:diagramqsplit}. Son $\Xi^{\mathrm{n.r.}}$ vaut alors $\ZZ/(n+1)\ZZ$ et est donné par la rotation décrite dans la table \ref{table:diagramqsplit}. Considérons un type $\typebt$ de ce diagramme et notons $m$ le cardinal de son orbite par $\Xi^{\mathrm{n.r.}}$ (on a donc $m\mid n+1$). On a :
    \begin{itemize}
        \item Si $m$ est impair, alors $(\mathrm{Orb}(\typebt)_{\Xi^\mathrm{n.r.}})^{\Galnr}$ est trivial.
        \item Si $m$ est pair, alors $n+1$ aussi, l'ensemble $(\mathrm{Orb}(\typebt)_{\Xi^\mathrm{n.r.}})^{\Galnr}$ contient $2$ éléments, et on a de plus les cas suivants :
        \begin{itemize}
            \item Si $\frac{n+1}{m}$ est impair, alors $(\mathrm{Orb}(\typebt)_{\Xi^\mathrm{n.r.}})^{\Galnr}/\Xi$ est trivial.
            \item Sinon, si $\frac{n+1}{m}$ est pair, alors $(\mathrm{Orb}(\typebt)_{\Xi^\mathrm{n.r.}})^{\Galnr}/\Xi$ contient $2$ éléments.
        \end{itemize}
    \end{itemize}
    En particulier, $(\mathrm{Orb}(\typebt)_{\Xi^\mathrm{n.r.}})^{\Galnr}/\Xi$ est trivial si $n\not \equiv 3 \pmod 4$.
\end{prop}

\begin{proof}
    Pour le type $A_n$, le groupe des automorphismes est donné par le groupe diédral. Il est donc donné par la présentation $\langle r,\sigma \mid r^{n+1} = 1, \sigma^2 =1, \sigma r \sigma =r^{-1} \rangle$. Le groupe $\Xi^{\mathrm{n.r.}}$ associé est alors le sous-groupe engendré par $r$ (qui est donc $\ZZ/ (n+1)\ZZ$), et $\Galnr$ agit au travers du sous-groupe engendré par $\sigma$.

    Considérons donc un type $\Galnr$-invariant $\typebt$, qui est donc donné par un sous-ensemble de sommets. Essayons de voir si l'orbite de $\typebt$ par $\Xi^{\mathrm{n.r.}}$ admet un autre type $\Galnr$-invariant. Soit $\typebt'$ un éventuel type de la sorte. Considérons $m\in \NN^*$, le plus petit entier strictement positif tel que $r^m \cdot \typebt = \typebt$ (c'est aussi le cardinal de l'orbite par $\Xi^{\mathrm{n.r.}}$).
    
    Considérons donc $k\in \{0,...,m-1\}$ tel que $r^k\cdot \typebt = \typebt'$. Comme $\typebt'$ est $\Galnr$-invariant, on a $\sigma \cdot \typebt' = \typebt'$. On a donc :
    $$r^k \cdot \typebt = \typebt' = \sigma \cdot \typebt' = \sigma \cdot (r^k \cdot \typebt) = \sigma r^k \cdot (\sigma\cdot \typebt) = \sigma r^k \sigma \cdot \typebt = r^{-k} \cdot \typebt$$

    On conclut donc que $r^{2k}\cdot \typebt = \typebt$. Comme $2k \in \{0,...,2m-2\}$, ou bien $2k=0$, ou bien $2k=m$ par minimalité de $m$ : c'est à dire $k=0$ ou $k=\frac{m}{2}$. Si $m$ est impair, la seconde possibilité est à proscrire et seul donc $k=0$ est valide. Sinon, les deux possibilités sont valides. Il y a donc 1 élément $\Galnr$-invariants dans l'orbite de $\typebt$ par $\Xi^{\mathrm{n.r.}}$ si $m$ est impair, et $2$ sinon.
    \medskip

    Comme $m\mid n+1$, si $n$ est pair, alors $m$ est toujours impair. Par conséquent, il y a donc 1 élément $\Galnr$-invariants dans l'orbite de $\typebt$ par $\Xi^{\mathrm{n.r.}}$. Étudions maintenant le cas où $n=2n'+1$ est impair et où $m=2m'$ est pair. On a donc $m'\mid n'+1$.

    Considérons maintenant $\Xi$, qui, selon la table \ref{table:diagramqsplit}, n'est autre que le groupe $\langle r^{n'+1} \rangle$.

    \noindent Quand est-ce que $\typebt$ et $\typebt'=r^{m'}\cdot \typebt$ sont conjugués par cette rotation ? On a :
    $$r^{n'+1} \cdot \typebt= r^{m'}\cdot \typebt \iff r^{m'\frac{n'+1}{m'}} \cdot \typebt = r^{m'}\cdot \typebt$$

    Si $\frac{n'+1}{m'}=\frac{n+1}{m}$ est impair, alors $r^{m'\frac{n'+1}{m'}}=r^{m'}$ puisque $r^{2m'}\cdot \typebt = r^m \cdot \typebt = \typebt$, et dans ce cas l'égalité est satisfaite. On en déduit alors que $\typebt$ et $\typebt'$ sont conjugués par cette rotation.

    Dans le cas contraire, si $\frac{n+1}{m}$ est pair, on a $r^{m'\frac{n'+1}{m'}}\cdot \typebt=\typebt$. Il faut donc satisfaire $\typebt=r^{m'}\cdot \typebt$. Cela est impossible par minimalité de $m$. Les deux types $\typebt$ et $\typebt'$ ne sont alors pas conjugués par cette rotation.

    Notons que le second cas ne se produit jamais si $n'$ est pair (ou encore si $n$ est congru à $1$ modulo $4$). On a donc considéré tous les cas.
\end{proof}

\begin{rmq}\label{ContreEXInjectivite}
    \cite[5.2.13]{BT2} donne un exemple de groupe quasi-déployé adjoint $G$ de type ${}^2A_3$ et déployé par une extension non ramifiée avec une $\Gamma$-facette $\widetilde{\Fc}$ tel que \linebreak $(\mathrm{Orb}(\widetilde{\Fc})_{G(\Lt)})^{\Gamma}/G(\Kt)$ est non trivial, et donc, d'après le point (1) (a) du théorème \ref{NoyauH1}, tel que \linebreak $\Ker\left(H^1(\Gamma,G(\Lt)_{\widetilde{\Fc}}) \rightarrow H^1(\Gamma,G(\Lt) \right)$ est non trivial.

    Le calcul précédent généralise en fait ce résultat. Dans l'exemple \cite[5.2.13]{BT2}, on choisit en fait une $\Galnr$-arête tel que son orbite par $\Xi^\mathrm{n.r.}$ soit de cardinal $m=2$. Comme $\frac{n+1}{m}=\frac{4}{2}=2$, la proposition précédente permet de conclure que $(\mathrm{Orb}(\typebt)_{\Xi^\mathrm{n.r.}})^{\Galnr}/\Xi$ admet deux éléments.
\end{rmq}

Occupons-nous maintenant du cas du type ${}^2D_n$ (pour $n\geq 4$).
\medskip 

Pour cela, on a besoin des morphismes introduits dans la table \ref{table:diagramqsplit}, c'est à dire $\tau,\tau',\sigma$ et $\varphi$. Ils sont par ailleurs définis peu importe si $n$ est pair ou impair. Plus précisément, $\tau$ est la symétrie par rapport à l'axe vertical central, $\tau'$ est la symétrie par rapport à l'axe horizontal (ou la rotation des deux branches extrémales), $\sigma$ est la rotation de la branche $3-4$, et $\varphi$ est en fait $\tau \circ \sigma$. Notons d'ailleurs que $\varphi^2 = \tau'$.

Introduisons également le symbole $\oplus$ pour désigner la "concaténation de types". Autrement dit, à deux types, cela associe le type donné par l'union des sommets composant chacun des types.
\medskip

On peut alors formuler le résultat :

\begin{prop}
    Considérons le digramme de Dynkin affine de type $D_{n}$ (pour $n\geq 4$) muni des actions indiquées par la table \ref{table:diagramqsplit} (pour $D_4$, on regarde le cas non trialitaire). Prenons $\typebt$, un type $\Galnr$-invariant de ce diagramme. Il s'écrit alors $\widetilde{S}\oplus \widetilde{R}$ où $\widetilde{S}$ est un type sans les quatre sommets numérotés de la table \ref{table:diagramqsplit} (donc $\Galnr$-invariant) et $\widetilde{R}$, un type $\Galnr$-invariant dont les sommets sont parmi les quatre sommets numérotés.
        Alors on a :
    \begin{enumerate}
        \item Si $\widetilde{R}$ admet zéro ou quatre sommets, alors :
        \begin{enumerate}
            \item Si $\tau(\widetilde{S})=\widetilde{S}$, alors les ensembles $(\mathrm{Orb}(\typebt)_{\Xi^\mathrm{n.r.}})^{\Galnr}$ et $(\mathrm{Orb}(\typebt)_{\Xi^\mathrm{n.r.}})^{\Galnr}/\Xi$ sont tous deux triviaux.
            \item Sinon, ils sont tous deux de cardinal $2$.
        \end{enumerate}
        \item Si $\widetilde{R}$ admet un nombre impair de sommets, alors l'ensemble $(\mathrm{Orb}(\typebt)_{\Xi^\mathrm{n.r.}})^{\Galnr}$ est de cardinal $2$ et $(\mathrm{Orb}(\typebt)_{\Xi^\mathrm{n.r.}})^{\Galnr}/\Xi$ est trivial.
        \item Si $\widetilde{R}$ admet deux sommets, alors les ensembles $(\mathrm{Orb}(\typebt)_{\Xi^\mathrm{n.r.}})^{\Galnr}$ et $(\mathrm{Orb}(\typebt)_{\Xi^\mathrm{n.r.}})^{\Galnr}/\Xi$ sont tous deux de cardinal $2$.
    \end{enumerate}
\end{prop}

\begin{proof}
    Bien entendu, puisque $\typebt$ et $\widetilde{S}$ sont $\Galnr$-invariants, il en est de même pour $\widetilde{R}$. On observe alors que les seules possibilités pour $\widetilde{R}$ sont :
    $$\emptyset,(1),(2),(1,2),(3,4),(1,3,4),(2,3,4),(1,2,3,4)$$
    Il faut donc traiter chacun de ces cas.
    
    Observons par ailleurs que l'orbite de $\typebt$ sous $\langle \varphi \rangle$ est la même que sous $\langle \tau,\tau' \rangle$. En effet, cela est une conséquence du fait que $\varphi=\tau \circ \sigma$ et $\varphi^2 = \tau'$. En conséquence, les calculs sont les mêmes peu importe si $n$ est pair ou impair.
    
    De rapides calculs grâce à la table \ref{table:diagramqsplit} permettent alors d'obtenir :
    \begin{enumerate}
        \item $\widetilde{R}=\emptyset$. On trouve que $(\mathrm{Orb}(\typebt)_{\Xi^\mathrm{n.r.}})^{\Galnr}=\{\widetilde{S},\tau(\widetilde{S})\}$. Il est donc trivial si et seulement si $\widetilde{S}=\tau(\widetilde{S})$. On observe ensuite qu'il en est de même pour $(\mathrm{Orb}(\typebt)_{\Xi^\mathrm{n.r.}})^{\Galnr}/\Xi$.
        \item $\widetilde{R}\in \{(1),(2)\}$. On trouve que $(\mathrm{Orb}(\typebt)_{\Xi^\mathrm{n.r.}})^{\Galnr}=\{\widetilde{S}\oplus(1),\widetilde{S}\oplus(2)\}$. On observe ensuite que $(\mathrm{Orb}(\typebt)_{\Xi^\mathrm{n.r.}})^{\Galnr}/\Xi$ est trivial.
        \item $\widetilde{R}=(1,2)$. On trouve $(\mathrm{Orb}(\typebt)_{\Xi^\mathrm{n.r.}})^{\Galnr}=\{\widetilde{S}\oplus(1,2),\tau(\widetilde{S})\oplus(3,4)\}$. On observe ensuite que $(\mathrm{Orb}(\typebt)_{\Xi^\mathrm{n.r.}})^{\Galnr}/\Xi$ admet deux éléments.
        \item $\widetilde{R}=(3,4)$. On trouve $(\mathrm{Orb}(\typebt)_{\Xi^\mathrm{n.r.}})^{\Galnr}=\{\widetilde{S}\oplus(3,4),\tau(\widetilde{S})\oplus(1,2)\}$. On observe ensuite que $(\mathrm{Orb}(\typebt)_{\Xi^\mathrm{n.r.}})^{\Galnr}/\Xi$ admet deux éléments.
        \item $\widetilde{R}\in \{(1,3,4),(2,3,4)\}$. On trouve $(\mathrm{Orb}(\typebt)_{\Xi^\mathrm{n.r.}})^{\Galnr}=\{\widetilde{S}\oplus(1,3,4),\widetilde{S}\oplus(2,3,4)\}$. On observe ensuite que $(\mathrm{Orb}(\typebt)_{\Xi^\mathrm{n.r.}})^{\Galnr}/\Xi$ est trivial.
        \item $\widetilde{R}=(1,2,3,4)$. On trouve $(\mathrm{Orb}(\typebt)_{\Xi^\mathrm{n.r.}})^{\Galnr}=\{\widetilde{S}\oplus(1,2,3,4),\tau(\widetilde{S})\oplus(1,2,3,4)\}$. Il est donc trivial si et seulement si $\widetilde{S}=\tau(\widetilde{S})$. On observe ensuite qu'il en est de même pour $(\mathrm{Orb}(\typebt)_{\Xi^\mathrm{n.r.}})^{\Galnr}/\Xi$. 
    \end{enumerate}
    Ceci conclut donc.
\end{proof}

\begin{prop}
    Considérons le diagramme de Dynkin affine de type $D_4$ muni de l'action de Galois $\Galnr$ donnée, ou bien par la rotation de $3$ points, ou bien toutes les permutations possibles de ces $3$ points (autrement dit le cas trialitaire de la table \ref{table:diagramqsplit}). Son groupe $\Xi^{\mathrm{n.r.}}$ vaut alors $(\ZZ/2\ZZ)^2$. On a alors que $(\mathrm{Orb}(\typebt)_{\Xi^\mathrm{n.r.}})^{\Galnr}$ est trivial pour tout type de ce diagramme.
\end{prop}

\begin{proof}
    Réutilisons la numérotation de la table \ref{table:diagramqsplit}. On identifie un type avec le $n$-uplet de ses points. Observons alors que les seuls types $\Galnr$-invariants sont $(0)$, $(1)$, $(0,1)$, $(2,3,4)$, $(0,2,3,4)$, $(1,2,3,4)$ et $(0,1,2,3,4)$. Comme l'action par $\Xi^\mathrm{n.r.}$ préserve la taille des types, on peut déjà dire que $(\mathrm{Orb}(\typebt)_{\Xi^\mathrm{n.r.}})^{\Galnr}$ est trivial pour $\typebt$ dans $\{(0,1), (2,3,4), (0,1,2,3,4)\}$. Comme $(0)$ est fixe par $\Xi^\mathrm{n.r.}$, on peut également éliminer $(0)$ et $(1)$. De même, tout type dans l'orbite de $(0,2,3,4)$ par $\Xi^\mathrm{n.r.}$ doit contenir $0$, donc $(1,2,3,4)$ ne peut pas être dans l'orbite. On a donc traité tous les cas et $(\mathrm{Orb}(\typebt)_{\Xi^\mathrm{n.r.}})^{\Galnr}$ est trivial pour tout type $\Galnr$-invariant $\typebt$.
\end{proof}

\begin{prop}
    Considérons le diagramme de Dynkin affine de type $E_6$ muni de l'action de Galois $\Galnr$ donnée par la symétrie axiale de la table \ref{table:diagramqsplit}. Son $\Xi^{\mathrm{n.r.}}$ vaut $\ZZ/3\ZZ$ et est donné par la rotation décrite dans la table \ref{table:diagramqsplit}. On a alors que $(\mathrm{Orb}(\typebt)_{\Xi^\mathrm{n.r.}})^{\Galnr}$ est trivial pour tout type de ce diagramme.
\end{prop}

\begin{proof}
    La preuve est essentiellement la même que pour le cas $A_n$. Soit $\typebt$, un type de ce diagramme. Prenons $r\in \Xi^{\mathrm{n.r.}}$ et $\sigma\in \Galnr$. Une fois encore, on a $\sigma \circ r = r^2 \circ \sigma$. Si $r\cdot \typebt$ est $\Galnr$-invariant, il est tel que : $$r \cdot \typebt = (\sigma \circ r) \cdot \typebt = (r^2 \circ \sigma) \cdot \typebt = r^2 \cdot (\sigma \cdot \typebt)=r^2 \cdot \typebt$$
    En conséquence, $\typebt=r\cdot \typebt=r^2 \cdot \typebt$ et donc $(\mathrm{Orb}(\typebt)_{\Xi^\mathrm{n.r.}})^{\Galnr}$ est trivial.
\end{proof}

\pagebreak

Résumons tout ceci grâce au tableau \ref{tab:ValeursQSplit} (en reprenant les notations des propositions précédentes) :

\begin{table}[h]
    \renewcommand{\arraystretch}{2}
    \centering
    \begin{tabular}{c c c}
        Type de $G$ \medskip & $\#(\mathrm{Orb}(\typebt)_{\Xi^\mathrm{n.r.}})^{\Galnr}$ & $\#(\mathrm{Orb}(\typebt)_{\Xi^\mathrm{n.r.}})^{\Galnr}/\Xi$ \\
        \hline
        \begin{tabular}{c c} 
            ${}^2 A_n$ déployé sur $\Ktnr$ \\
            (pour $n\geq 1$)
        \end{tabular}
        \medskip
        & 
        \begin{tabular}{c c} 
             $m$ impair : & 1 \\
             $m$ pair : & 2 
        \end{tabular}
         & 
        \begin{tabular}{c c}
             $m$ impair : & 1 \\
             $m$ pair et $\frac{n+1}{m}$ impair : & 1 \\
             $m$ pair et $\frac{n+1}{m}$ pair : & 2 \\
        \end{tabular} \\     
        \arrayrulecolor{lightgray}\hline
        \begin{tabular}{c c} 
            ${}^2 D_n$ déployé sur $\Ktnr$ \\
            (pour $n\geq 4$)
        \end{tabular}
        \medskip
        & 
        \begin{tabular}{c c}
            $\tau(\widetilde{S}) = \widetilde{S}$ et $\# \widetilde{R}\in \{0,4\}$ : & 1\\
            $\tau(\widetilde{S}) \not = \widetilde{S}$ et $\# \widetilde{R}\in \{0,4\}$ : & 2\\
            $\# \widetilde{R}=\{1,2,3\}$ : & 2 \\
        \end{tabular}
        & 
        \begin{tabular}{c c}
            $\tau(\widetilde{S}) = \widetilde{S}$ et $\# \widetilde{R}\in \{0,4\}$ : & 1\\
            $\tau(\widetilde{S}) \not = \widetilde{S}$ et $\# \widetilde{R}\in \{0,4\}$ : & 2\\
            $\# \widetilde{R}\in \{1,3\}$ : & 1 \\
            $\# \widetilde{R}=2$ : & 2 \\
        \end{tabular} \\
        \arrayrulecolor{lightgray}\hline
        Autres types \medskip & 1 & 1 \\
    \end{tabular}
    \caption{Résumé des calculs précédents.}
    \label{tab:ValeursQSplit}
\end{table}

On observe en particulier que seules les valeurs $1$ ou $2$ sont présentes, et que $2$ n'apparaît que lorsque $G$ est de type ${}^2A_n$ (pour $n\geq 1$) ou ${}^2D_n$ (pour $n\geq 4$), déployé sur $\Ktnr$.
\medskip 

En conclusion, on obtient le théorème suivant :

\begin{thm}\label{ThmStabQSplit}
    Soit $G$ un groupe semi-simple adjoint et quasi-déployé sur $K$. Soit également $\widetilde{\mathcal{F}}$, une facette $\Galnr$-invariante de l'immeuble $\ImmBT(G_{\Knr})$. Alors le noyau :
    $$\Ker\left(H^1(\Galnr,G(\Knr)_{\widetilde{\mathcal{F}}})\rightarrow H^1(\Galnr,G(\Knr)\right)$$
    est de cardinal $2^k$ où $k$ est un entier majoré par le nombre de facteurs restriction de Weil d'un groupe absolument presque simple de type ${}^2D_n$ (pour $n\geq 4$) ou ${}^2A_{4n+3}$ (pour $n\geq 0$) déployé par une extension non ramifiée.
\end{thm}

\begin{rmq}
    Bien entendu, il est possible de calculer explicitement ce noyau en se réduisant au cas absolument presque simple grâce à la compatibilité du noyau au produit et à la restriction de Weil (cf. le lemme \ref{CompNoyauResWeil}) et en utilisant la table \ref{tab:ValeursQSplit}.
\end{rmq}

\pagebreak

\bibliographystyle{alpha}
\bibliography{bibliography}

\bigskip

\end{document}